\documentclass[12pt]{amsart}
\usepackage{amsmath,amssymb, amscd, graphicx,youngtab}
\makeatletter
\usepackage{times}
\usepackage{enumerate}
\makeatother
\usepackage{enumerate}
\def\theequation{\thesection.\arabic{equation}}
\makeatletter
\@addtoreset{equation}{section}

\def\ii{\sqrt{-1}}
\def\ee{\mathrm e}

\def\wt{\mathrm {wt}}

\def\ta{{\mathsf{t}}}
\def\ker{\mathrm {ker }}
\def\Spec{\mathrm {Spec\ }}
\def\Proj{\mathrm {Proj\ }}

\def\hu{{\hat {u}}}

\def\uf{{\underline{f}}}

\def\CC{{\mathbb C}}

\def\NN{{\mathbb N}}
\def\RR{{\mathbb R}}
\def\ZZ{{\mathbb Z}}
\def\QQ{{\mathbb Q}}

\def\JJ{{\mathcal J}}
\def\cK{{\mathcal K}}

\def\JJf{{{\mathcal J}^{(4)}}}
\def\hJJf{{\hat{\mathcal J}^{(4)}}}

\def\JJt{{{\mathcal J}^{(12)}}}

\def\WWf#1{{{\mathcal W}^{(4){#1}}}}

\def\WWt#1{{{\mathcal W}^{(12){#1}}}}
\def\cM{{{\mathcal M}}}
\def\hz{{\hat{z}{}}}
\def\fbz#1{{f^{(b,z)}_{#1}}}

\def\kbz#1{{k^{(b,z)}_{#1}}}
\def\Rz#1{{R^{(z)}_{#1}}}
\def\fbw#1{{f^{(b,\hz)}_{#1}}}
\def\hkbw#1{{{\hat k}^{(b,\hz)}_{#1}}}
\def\kbw#1{{k^{(b,\hz)}_{#1}}}
\def\Rw#1{{R^{(\hz)}_{#1}}}
\def\fZ#1{{f^{(Z)}_{#1}}}

\def\fb#1{{f^{(b)}_{#1}}}

\def\fCf{{{\mathfrak C}^{(4)}}}
\def\fCt{{{\mathfrak C}^{(12)}}}

\def\sigmag#1{{{\sigma}^{(g)}_{#1}}}

\def\wpg#1{{{\wp}^{(g)}_{#1}}}
\def\wpf#1{{{\wp}^{(4)}_{#1}}}
\def\wpt#1{{{\wp}^{(12)}_{#1}}}

\def\sigmat#1{{{\sigma}^{(12)}_{#1}}}

\def\JJ{{\mathcal J}}
\def\SS{{\mathcal S}}
\def\WW{{\mathcal W}}

\def\cH{{\mathcal{H}}}
\def\cA{{\mathcal{A}}}

\def\cJ{\mathcal{J}}

\def\cO{\mathcal{O}}
\def\cL{\mathcal{L}}

\def\cP{\mathcal{P}}
\def\tP{\tilde{P}}

\def\phig#1{{\phi^{(g)}_{#1}}}
\def\Ng{{N}^{(g)}}
\def\NHf{{N_{H^1}^{(4)}}}
\def\NMf{{N_{H^0}^{{(4)}}}}
\def\Nt{{N^{{(12)}}}}
\def\omegagp#1{{\omega^{{(g)\prime}}_{#1}}}
\def\omegagpp#1{{\omega^{{(g)\prime\prime}}_{#1}}}
\def\omegafp#1{{\omega^{{(4)\prime}}_{#1}}}
\def\omegafpp#1{{\omega^{{(4)\prime\prime}}_{#1}}}
\def\omegatp#1{{\omega^{{(12)\prime}}_{#1}}}
\def\omegatpp#1{{\omega^{{(12)\prime\prime}}_{#1}}}
\def\etagp#1{{\eta^{{(g)\prime}}_{#1}}}
\def\etagpp#1{{\eta^{{(g)\prime\prime}}_{#1}}}

\def\phiMf#1{{{\phi_{H^0}^{(4)}}_{#1}}}
\def\phiHf#1{{{\phi_{H^1}^{(4)}}_{#1}}}

\def\phis#1{{\phi^{(7)}_{#1}}}
\def\phie#1{{\phi^{(8)}_{#1}}}
\def\phit#1{{\phi^{(12)}_{#1}}}
\def\nuIg#1{{\nu^{{I;g}}_{#1}}}

\def\nuIf#1{{\nu^{{I;4}}_{#1}}}

\def\nuIs#1{{\nu^{{I;6}}_{#1}}}
\def\nuIe#1{{\nu^{{I;7}}_{#1}}}
\def\nuIt#1{{\nu^{{I;12}}_{#1}}}
\def\tnuIs#1{{\tilde\nu^{{I;6}}_{#1}}}
\def\tnuIe#1{{\tilde\nu^{{I;7}}_{#1}}}
\def\nuIt#1{{\nu^{{I;12}}_{#1}}}
\def\nuIIg#1{{\nu^{{II;g}}_{#1}}}

\def\nuIIf#1{{\nu^{{II;4}}_{#1}}}

\def\nuIIt#1{{\nu^{{II;12}}_{#1}}}

\def\muc{{b}}

\def\ltwo#1{{\lambda^{(2)}_{#1}}}
\def\hltwo#1{{\hat{\lambda}{}^{(2)}_{#1}}}
\def\lthree#1{{\lambda^{(3)}_{#1}}}
\def\uab{{\hat{u}}}
\def\muf{{\mu^{{(4)}}}}
\def\muHf#1{{\mu_{H^1,#1}^{(4)}}}
\def\alphaH{{\alpha_{H^1}}}

\def\alphat{{\alpha^{{(12)}}}}
\def\muMf#1{{\mu_{H^0,#1}^{{(4)}}}}

\def\mut#1{{\mu^{{(12)}}_{#1}}}

\def\psig#1{{\psi^{{(g)}}_{#1}}}
\def\Psig#1{{\Psi^{{(g)}}_{#1}}}
\def\mug#1{{\mu^{(g)}_{#1}}}

\def\Sigmaf{{\Sigma^{{(4)}}}}
\def\Omegaf{{\Omega^{{(4)}}}}
\def\Sigmat{{\Sigma^{{(12)}}}}
\def\Omegat{{\Omega^{{(12)}}}}

\def\Omegag{{\Omega^{{(g)}}}}
\def\Pif{{\Pi^{{(4)}}}}
\def\Ff{{F^{{(4)}}}}
\def\Pit{{\Pi^{{(12)}}}}
\def\Ft{{F^{{(12)}}}}
\def\Fg{{F^{{(g)}}}}
\def\taug#1{{{\tau}^{(g)}_{#1}}}

\def\alphag#1{{\alpha^{{(g)}}_{#1}}}
\def\betag#1{{\beta^{{(g)}}_{#1}}}
\def\gammag#1{{\gamma^{{(g)}}_{#1}}}
\def\alphat#1{{\alpha^{{(12)}}_{#1}}}
\def\betat#1{{\beta^{{(12)}}}_{#1}}

\def\alphaf#1{{\alpha^{{(4)}}_{#1}}}
\def\betaf#1{{\beta^{{(4)}}}_{#1}}

\def\hzeta{{\hat{\zeta}}}

\def\LA{\langle}
\def\RA{\rangle}
\def\hb{{\hat{b}{}}}
\def\hk{{\hat{k}{}}}
\def\uhk{{\underline{\hat{k}}{}}}
\def\uk{{\underline{{k}}}}
\def\uy{{\underline{{y}}}}
\def\ux{{\underline{{x}}}}
\def\uf{{\underline{{f}}}}
\def\uR{{\underline{{R}}}}
\def\ucU{{\underline{{\mathcal{U}}}}}
\def\hw{{\hat{w}{}}}
\def\hy{{\hat{y}{}}}
\def\tc{{\tilde{c}{}}}

\def\hR{{\hat{R}{}}}

\def\fGmf{{\mathbb{G}_m^{(4)}}}
\def\fGmt{{\mathbb{G}_m^{(12)}}}
\newtheorem{theorem}{Theorem}[section]
\newtheorem{definition}[theorem]{Definition}

\newtheorem{proposition}[theorem]{Proposition}
\newtheorem{corollary}[theorem]{Corollary}
\newtheorem{remark}[theorem]{Remark}
\newtheorem{lemma}[theorem]{Lemma}

\newtheorem{example}[theorem]{Example}

\def\rank{\mathrm{rank\ }}
\def\cU{{\mathcal{U}}}

\def\book#1{\rm{#1}, }
\def\paper#1{\textit{#1}, }
\def\jour#1{\rm{#1}, }
\def\yr#1{({\rm{#1}) }}
\def\vol#1{\textbf{#1}}
\def\pages#1{\rm{#1}}

\def\publ#1{\rm{#1}, }
\def\by#1{{\rm{#1}, }}

\begin{document}

\markboth{J. Komeda, S. Matsutani, E. Previato}
{The sigma function
 for Weierstrass semigroups $\langle 3,7,8\rangle$ and 
$\langle 6,13,14,15,16\rangle$}

%%%%%%%%%%%%%%%%%%%%% Publisher's Area please ignore %%%%%%%%%%%%%%
%\catchline{}{}{}{}{}
%%%%%%%%%%%%%%%%%%%%%%%%%%%%%%%%%%%%%%%%%%%%%%%%%%%%%%%%%%%%%%%%%%%

%\title{The sigma function for Weierstrass semigroups $\langle 3,7,8\rangle$ and 
%$\langle 6,13,14,15,16\rangle$}
\title{ THE SIGMA FUNCTION FOR WEIERSTRASS SEMIGROUPS $\langle 3,7,8\rangle$ AND 
$\langle 6,13,14,15,16\rangle$}

%\author{Jiryo Komeda}
\author{ JIRYO KOMEDA}

\address{Department of Mathematics,
Center for Basic Education and Integrated Learning,
Kanagawa Institute of Technology,
Atsugi, 243-0292,
JAPAN.\\
%e-mail:
{komeda@gen.kanagawa-it.ac.jp}
}

%\author{Shigeki Matsutani}
\author{ SHIGEKI MATSUTANI}

\address{8-21-1 Higashi-Linkan Minami-ku,
Sagamihara 252-0311,
JAPAN.\\
%e-mail:
{rxb01142@nifty.com}
}

%\author{Emma Previato}
\author{ EMMA PREVIATO}

\address{Department of Mathematics and Statistics,
Boston University,
Boston, MA 02215-2411,
U.S.A.\\
%e-mail:
{ep@bu.edu}
}

\maketitle

\begin{abstract}
Compact Riemann surfaces and their abelian functions are instrumental to 
solve integrable equations; more recently the representation theory
of the Monster and related modular forms have pointed to the relevance 
of $\tau$-functions, which are in turn connected with a specific
type of abelian function, the (Kleinian) $\sigma$-function.  
Klein originally generalized Weierstrass' $\sigma$-function to 
hyperelliptic curves in a geometric way, then from  a modular point of view
to trigonal curves. Recently a 
modular generalization for all curves was given,
as well as the geometric one for certain affine planar curves,
known as $(n,s)$ curves, and their generalizations known as
``telescopic'' curves.
This paper
proposes a construction of $\sigma$-functions based on the
nature of the Weierstrass semigroup at one point of the Riemann surface
as a generalization of the construction for plane affine models of 
the Riemann surface.  Our examples are not telescopic. 
Because our definition is algebraic, in the sense of being related to
the field of meromorphic functions on the curve,
we are able to consider 
properties of the $\sigma$-functions such as
 Jacobi inversion formulae, and to observe relationships
 between their properties and those of a Norton basis
for replicable functions, in turn relevant to the Monstrous
Moonshine.
%% Keywords are optional
\end{abstract}



\section{Introduction}
We work with 
compact Riemann surfaces, and simply refer to them as
``curves'', although
in some instances we allow for singular points.
We propose a generalization of the $\sigma$-function for $(n,s)$
curves, which was developed recently, based on Klein's and
Baker's theories \cite{B1}; we parallel
 constructions given in Mumford's
Tata lectures on theta II \cite{Mu}, namely algebraic expressions
for abelian functions, based on nineteenth-century work; by
the use of specific abelian variables and the introduction of 
a generalized Weierstrass $\wp$-function,
defined for hyperelliptic curves \cite{Mu} (IIIa. Section 10), Mumford 
expresses the transcendental function theory on the Jacobian in terms
of meromorphic functions on the curve;
applications are given to 
 integrable dynamical systems.

\subsection{The $\sigma$-function}
We say that a pointed algebraic curve is of W-type
 $(n_1,...,n_s)$ ($n_1<n_2<...<n_s$) if its Weierstrass
semigroup at the point  
has minimal set of generators $\{ n_1,...,n_s\}$.
The  Kleinian $\sigma$-function
(cf. \cite{B1,B2,K,W}),
originally devised for hyperelliptic curves,
which have W-type $(2,2g+1)$, is closely related to Mumford's $\wp$.
Baker showed in essence that $\sigma$ 
satisfies the KdV hierarchy and
KP equation (as they became known in the 1970s). He introduced a
 bilinear equation, which 
was newly proposed in the 1960s as
Hirota's bilinear operator \cite{B3}. 
Baker's work 
was recently revisited and expanded  by several authors, among whom
co-authors Buchstaber, Enolskii, and Leykin 
\cite{BEL1,BEL2,BLE1,BLE2};
Grant \cite{Gr}; Matsutani \cite{Ma}; \^Onishi \cite{O}.
In subsequent work
\cite{EEL}, Eilbeck, Enolskii and Leykin gave a construction
of the Kleinian $\sigma$-function of \cite{BLE1}
generalized to curves of W-type $(n,s)$,
where $n$ and $s$ are any two relatively prime positive integers;
we call them $(n,s)$ curves for short, as in the literature
(these are the same as the $C_{rs}$ curves of \cite{MP08,MP12});
they did so by defining the fundamental differential of the second kind 
over every $(n,s)$ curve. 
We call their construction {``{EEL construction}''}.
Using the EEL construction,
properties 
such as addition formulae and the order of vanishing
of the $\sigma$-function of  $(n,s)$ curves are obtained
in \cite{EEMOP1,EEMOP2,MP08}. Nakayashiki \cite{N} investigated
properties of the function also by the use of differentials
on symmetric products of the curve.
Further Eilbeck,  Enolski and Gibbons \cite{EEG} and 
Nakayashiki \cite{N2} exhibited connections between $\sigma$-function
and Sato's $\tau$-function.

In our program, we extend the study  to affine curves
in higher-dimensional space.
The construction was given by the first- and second-named authors in \cite{MK}
for a curve of W-type $(3,4,5)$.
As in the $(n,s)$ case, we henceforth omit ``W-type'' for brevity.
Here, we generalize the $\sigma$-function to 
 a  $(3, 7, 8)$   curve with affine model in 3-dimensional space,
  and to a 
$(6, 13, 14, 15, 16)$  curve which covers it.
The motivation of this generalization comes from the
the Monstrous Moonshine
and the  Weierstrass
semigroup problem, described below in subsection \ref{moonshine}.
Following the EEL construction, we define the fundamental differential of 
the second kind, which
gives a generalized Legendre relation; this in turn gives us 
the $\sigma$-function  and we obtain the
Jacobi inversion formulae following previous work \cite{MP08}.
We note that
our Kleinian $\sigma$-function is 
again a generalization of Weierstrass' elliptic $\sigma$.

 Korotkin with Shramchenko defined a sigma function
for any compact Riemann surface \cite{KS},
 which so far has not been explicitly related with
 meromorphic functions on the
curve; on the other hand, Ayano investigated sigma functions 
for space curves of a special type \cite{A}, called telescopic curves,
but those do not include a (3,7,8) or a (6,13,14,15,16) curve.
%\bigskip
%\bigskip

\subsection{Monstrous Moonshine}\label{moonshine}
Our motivation for
 focusing on the curves
 of genus 4, $(3,7,8)$ and genus 12,
$(6,13,14,15,$ $16)$, is due  to a suggestion by John McKay.
The curve $(6,13,14,15,16)$ is
associated with a Jacobi variety whose real dimension is 24
and has the Weierstrass gap sequence
$\{1, 2, 3, $ $4, 5, 7, 8, 9, 10, 11, 17, 23 \}$,
 similar to the Norton numbers 
$\{1, 2, 3, 4, 5, 7,$ $8, 9, 11, 17, 19, $ $23 \}$.
The Norton numbers, as stated in
\cite{FMN,Mc,MS,N}, e.g.,  
are the coefficients of  replicable functions in a Norton basis;
in turn, replicable functions  
are related to the Monstrous Moonshine.
In  VI.3 of \cite{HBJ}, the \textbf{Prize Question} reads,
for $\hat{A}$ and
the (twisted) $\hat{A}$-genus (since we are quoting
\textit{verbatim}, we refer to \cite{HBJ} for technical definitions):
\textit{Does there exist a 24-dimensional, compact, orientable, differentiable
manifold $X$ with $p_1(X)=0, \ w_2(X)=0$,
$\hat{A}(X)=1$ and $\hat{A}(X,T_\mathbb{C})=0$?} 
%Here $p_1$ is the first Pontryagin class, 
 %$w_2$ is the second Stiefel-Whitney class, 
 %and $T_{\mathbb{C}}$ is the complexified tangent bundle,
 %$T_{\mathbb{C}}:=T X \otimes{\mathbb{C}}$.
For such a manifold, suitable twisted $\hat{A}$ genera 
would equal dimensions of irreducible representations of the Monster;
and if the Monster acted on $X$ by diffeomorphisms, \textit{one
would possess a key to construct a great many representations of the Monster}.
The  Witten genus is 
expressed in terms of  the Weierstrass $\sigma$
function \cite{HBJ}.

Of course, although a Jacobian is compact, orientable, differentiable
and admits a spin structure \cite{hitchin}
(equivalently, its $w_2$ is zero),
$p_1$ is non-zero, being the genus of the curve. 
Therefore, though our Jacobian has real dimension
24, we do not propose a link
with the above problem,
but
we give some remarks in Section \ref{sec:Curves&MM}
on relevant geometric properties from a viewpoint of the construction of 
the generalized $\sigma$ function associated with a $(6,13,14,15,16)$
curve.

%\bigskip
With such motivation, we sought
a $(6,13,14,15,16)$ curve.
However, as recalled in Section \ref{sec:WSemigroup}, 
it is a non-trivial problem whether there
exists a pointed curve with given  numerical semigroup
at the point. 
This article is based upon J. Komeda's proof that there exists
a non-singular curve with  Weierstrass gap sequence 
$\{1, 2, 3, 4, 5, 7, 8, 9, 10, 11, 17, 23 \}$
 at a given point (see Proposition \ref{prop:KomedaProp1});
its Weierstrass semigroup is indeed $H_{12}$, the sub-semigroup 
of 
the non-negative integers $\NN_0$ generated by
$M_{12}:=\{6, 13, 14, 15, 16\}$. The genus-12 curve constructed 
by Komeda covers a (genus-4) 
 $(3,7,8)$ curve; we begin with the construction and analysis of
its $\sigma$-function;
relationships between the algebraic functions of the curves
exhibit similar properties to those of the Monstrous Moonshine.

Let $H$ denote a subsemigroup of the additive group of non-negative integers
such that only finitely many positive integers, called the gaps of $H$, are
missing from $H$. 
 %Let $k$ be an algebraically closed field of characteristic zero. 
 Let $B_H$ denote the subring of the polynomial ring $k[t]$ generated by
the monomials $t^h$, $h\in H$. Then Spec $B_H$ is a monomial curve; i.e. an
irreducible affine curve with $\mathbb{G}_m$-action,
where $\mathbb{G}_m$ denotes the multiplicative
group of $k$ ($\mathbb{C}$ in our case). There is an induced $
\mathbb{G}_m$-action on $T^1:=T^1(B_H)$,
the Zariski tangent space of Spec $B_H$ \cite{P}, p. 4,
 and, decomposing $T^1$ into eigenspaces,
one can 
write $T^1=\sum_\nu T^1(\nu)$, where $\nu$ ranges from $-\infty$ to
$+\infty$. One says that $B_H$ is negatively graded if $T^1(\nu)=0$ for all
positive $\nu$.

In his thesis \cite{P}, Pinkham
showed that $H$ occurs as a Weierstrass semigroup if and only if the
corresponding monomial curve, $\mathrm{Spec}\,k[t^h\colon h\in H]$, can be
smoothed negatively.
Pinkham used this to show that if $H$ is a complete intersection (in
		      particular, if $H$ is generated by two elements), then
		      $H$ occurs as a Weierstrass semigroup.
This also holds for 
almost complete intersection semigroups generated by three elements 
 and was generalized to almost complete intersection semigroups
generated by four elements.
 Rim and Vitulli \cite{rimvitulli}
give a complete characterization of such negatively graded semigroup rings. 
Moreover, they 
extend Pinkham's result  
by showing that every negatively graded semigroup ring can be smoothed. By the
work of Pinkham, this implies that if $H$ is any negatively graded semigroup,
then there exist a smooth projective curve $X$ and a point $x\in X$ such that
the gaps of $H$ are the Weierstrass gaps at $x$. 

In Section \ref{sec:WSemigroup}
we give a brief review of  Weierstrass semigroups.
In Sections \ref{sec:378} and \ref{sec:613141516},
 we proceed as follows:
we define monomial curves with given Weierstrass semigroups, with
motivation from the Norton numbers; when the semigroup falls outside the
verifiably smoothable case, we give Komeda's proof that one
smooth curve with such Weierstrass semigroup exists; Pinkham's
calculation of the expected dimension yields a positive number,
therefore we can conclude that the monomial curve is smoothable.
In Section \ref{sec:sigma}, we use the local coordinates given by
the monomial presentation of the curve, to manufacture a local section
of certain meromorphic differentials, and we construct an abelian
function on the curve, the $\sigma$-function, 
 by integrating those differentials. In \cite{MK},
the original idea was implemented 
for the semigroup $\LA3,4,5\RA$, including the $\sigma$-function
 construction
and natural extensions of the ones previously developed
for $(n,s)$ curves.
The $\sigma$-function provides a stratification of the Jacobian
in Section \ref{sec:JacobiInv};
in Section \ref{sec:Curves&MM},
we make some observations on the possible links with
the representation theory of the Monster, starting with the above-mentioned
 numerical observation connecting the Norton numbers and the
differentials we constructed (\ref{sec:AppendixA}).

\bigskip

%%%%%%%%%%%%%%%%% SM 06302013
%We would like to express our thanks to John McKay for posing to us
%the problem of the Moonshine: his work on replicable functions
%led him to observations and conjectures (cf. \cite{MS})
%which we propose to relate to our work on the $\sigma$ function;
%John hosted us at Concordia in 2004 and shared his vision most generously.  
%Further the second author (S.M) has been studied the geometrical
%interpretations of the replicable function by the guide of John McKay.
%On the interpretation, Yuji Kodama gave his attentions to the Norton number.
%John McKay taught him the relations among the Coble's work, sextic
%curves and $E^8$ action. The second author is 
%also grateful to Kenichi Tamano, Victor Enolskii, Dmitry Leykin, 
%Yoshihiro \^Onishi, Takao Kato and Akira Ohbuchi for valuable 
%discussions and comments.
%%%%%%%%%%%%%%%%% SM 06302013

\bigskip
\bigskip

\section{Weierstrass Semigroups}\label{sec:WSemigroup}

%\textbf{Probably best to shorten this section, since we do not aim at
%  completeness}

For reference, 
 we give a brief overview of 
the recent study of the
``numerical semigroups'', namely those (additive)
sub-semigroups of 
the non-negative integers $\NN_0$
whose complement is finite.
We then give a new result due to Komeda 
concerned with the analogy with the Norton basis;
this was already presented  in \cite{MK}.

We use standard notation, e.g.
$h^i$ for the dimension of the 
cohomology group $H^i$ of a sheaf over the curve.
Also, a sheaf $\mathcal{O}_X(D)$ where $D$ is a divisor on $X$ may be denoted
by $D$ for short, as well as $h^i(X,\mathcal{O}_X(D))$ by $h^i(D)$.

For a numerical semigroup $H$ generated by a set $M$,
the number of elements of $L(H):=\NN_0\setminus H$, the 
``gap sequence'' of $H$,  
is called ``genus'' if it is finite, and denoted by $g(H)$.
We focus on the numerical semigroup $H_4$ generated by
$M_{4}:=\{3, 7, 8\}$, i.e., $H_4=\LA3,7,8\RA$ as well as 
the problem of the numerical semigroup $H_{12}$ generated by
$M_{12}$, i.e., $H_{12}=\LA M_{12}\RA$for the reasons explained in subsection 
\ref{moonshine}.

 The genus of $H_{4},H_{12} $ (respectively)  is  $g = 4, 12$, and
$$
L(H_{4})=\{1, 2, 4, 5\}, \ \ \ 
L(H_{12})=\{1, 2, 3, 4, 5, 7, 8, 9, 10, 11, 17, 23 \} .
$$
 %A semigroup $H$ is called primitive if 
 %the largest gap is smaller
 %than twice the smallest positive integer in $H$.
 For a gap sequence $L:=\{\ell_0 < \ell_1 < \cdots < \ell_{g-1} \}$ of 
genus $g$, let $M(L)$ be the minimal set 
of generators for the  semigroup consisting of the complement of $L$
(it is easy to show that
a minimal set of generators must be unique) and 
\begin{equation}
\alpha(L) :=\{\alpha_0(L), \alpha_1(L), \ldots, \alpha_{g-1}(L)\},
\label{eq:alphaL}
\end{equation}
where $\alpha_i(L) := \ell_i - i -1$.  When an $\alpha_i$
is repeated $j>1$ times we write $\alpha_i^j(L)$ in $\alpha(L)$.
We set $\wt(L)=\sum_{i=0}^{g-1}\alpha_i(L)$, and refer to it as
 the {\it weight} of $L$.
We let $\alpha_{\mathrm{min}}(L)$ be the smallest positive integer of $M(L)$.
We call a semigroup $H$ an $\alpha_{\mathrm{min}}(L)$-semigroup, 
so that
$H_{4}$ is a  $3$-semigroup and
$H_{12}$ is a $6$-semigroup. Moreover, 
$$
\alpha(L(M_{4})) =\{0^2, 1^2\}
\quad\mbox{and}\quad
\alpha(L(M_{12})) =\{0^5, 1^5, 6, 11\}.
$$
The $6$-semigroups with 4 generators were studied by Komeda
in \cite{K04}.

\medskip
For a 
 %complete non-singular irreducible 
 curve $X$ of genus $g$
 %over an algebraically
 %closed field $k$ of characteristic 0, with field of rational 
 %functions $k(X)$, 
 and a point $P \in X$,  the semigroup 
$$ 
 H(X,P):= \{n \in \NN_0\ |\ \mbox{there exists } f \in k(X)
                          \mbox{ such that } (f)_\infty = n P\ \}
$$
is called the Weierstrass semigroup of the point $P$.
If $L(H(X,P))$ differs from the set 
$\{1, 2, \cdots, g\}$, 
we say that $P$ is a Weierstrass point of $X$.

A numerical semigroup $H$
is said to be Weierstrass 
if there exists a pointed  curve $(X, P)$ such that 
$H=H(X,P)$.
Hurwitz posed the problem whether any numerical semigroup $H$ is
Weierstrass. 
The question was revived in the 1980s, viewed as the question
of deformations of a 
reduced complex curve singularity $(X_0,\infty )$.
In 
\cite{buchweitzgreuel1,buchweitzgreuel2}, a counterexample was given,
for the semigroup $H_B$ generated by 
13, 14, 15, 16, 17, 18, 20, 22 and 23, whose genus is 16:
 this can be seen as follows.
Assume that $H_B = H(X,P)$ 
for a pointed curve $(X, P)$, then
$X$ would have a holomorphic differential vanishing to order 
$\ell - 1$ at $P$ for any $\ell \in  L(H_B)$.
By letting $L_2(H)$ be the set of all sums of two elements of
$L(H)$ for a semigroup $H$, 
$X$ would have a holomorphic quadratic differential
 vanishing to order
$m - 2$ at $P$ for any $m \in L_2 (H_B)$, 
which contradicts  the fact that
$\# L_2(H_B) = 46 > 45 = (3 \times 16)-3 = h^0(X, \cK_X^{\otimes2} )$.
Note that we write  $\cK_X^{\otimes2}$ 
for the sheaf  of regular quadratic
differentials on $X$.

We mainly refer to the work of Komeda and co-authors
for further information \cite{K83}-\cite{KO08a}.
D. Eisenbud with J. Harris \cite{EH},
C. Maclachlan and I. Morrison jointly with H. Pinkham are other
authors who wrote on this issue.
We sketch a list of sufficient properties for $H$ to be Weierstrass:
\begin{enumerate}
\item  semigroup whose cardinality of the generators is less than 4,

\item  semigroup whose genus is less than 9,

\item  semigroup which is primitive, i.e., twice the smallest positive integer in $H(L)>$ the largest
integer in $L$, of genus $9$,
%semigroup whose all primitive sequence, twice the smallest positive
%integer in $H(L)>$ the largest integer in $L$, of genus 9,

\item  semigroup whose 
$\alpha_{\mathrm{min}}(L) = 2, 3, 4, 5$,

\item semigroup which is primitive and $\wt(L)\leqq g-1$,

\item  semigroup whose $\alpha(L) = (0^{g-2}, m, n)$ for genus $g$
and $\wt(L) = g$,

\item  semigroup whose $\alpha(L) = (0^{g-r}, m^r)$ for genus $g$, and

\item other many special cases of $n$-semigroups generated by 4 elements.

\end{enumerate}
%
%A simple computation shows that  $\#L_2(H_{12})=33 = 12 \times 3 - 3$
%and $\#L_3(H_{12})=55 = (12 - 1) \times (2 \cdot 3 -1)$.
Due to (1) or (2), $H_4$ is  Weierstrass.

As it is stated in \cite{MK}, we have the following crucial 
proposition
due to Komeda:

\begin{proposition}\label{prop:KomedaProp1}
The numerical semigroup $\LA6, 13, 14, 15, 16\RA$ is Weierstrass.
\end{proposition}

\begin{proof}
We  consider  a pointed curve $(X, P)$ with $H(X,P) = \LA3, 7, 8\RA$,
which has genus 4, 
and $\mathcal{K}_X$ a canonical divisor on $X$. 
Then,
$$
2 = h^0(4P) = 4 + 1 - 4 + h^0(\mathcal{K}_X - 4P) = 1 + h^0(\mathcal{K}_X - 4P);
$$
this implies that $\mathcal{K}_X - 4P \sim P_1 + P_2$ for some points 
$P_1$ and $P_2 \in X$. 
Moreover, since 
$$
2 = h^0(5P) = 5 + 1 - 4 + h^0(\mathcal{K}_X - 5P) = 2 + h^0(\mathcal{K}_X - 5P),
$$
then $h^0(\mathcal{K}_X - 5P) = 0$, and $P_i \neq P$ for $i = 1, 2$. 
This implies that $\mathcal{K}_X \sim 4P + P_1 + P_2$ with $P_i \neq P$ for $i = 1, 2$. 
By setting $D := 7P - P_1 - P_2$, we have deg$(2D - P) = 9 = 2 \times 4 + 1$,
and this implies that the complete linear system $|2D - P|$ is 
very ample, hence base-point free. Therefore,
$2D \sim P + Q_1 + \ldots + Q_9$ ($=$ a reduced divisor).
Let $\cL$ be the invertible sheaf $\cO_X(-D)$ on $X$ and $\phi$
an isomorphism $\cL^{\otimes2} \cong \cO_X(-P-Q_1-\cdots-Q_9) \subset
\cO_X$. 
The vector bundle $\cO_X \oplus \cL$
 has an $\cO_X$-algebra structure through $\phi$. 
The canonical $2:1$ morphism $\pi : \tilde X 
= \Spec(\cO_X \oplus \cL) \to X$
has branch locus
  $\{P,Q_1, \ldots, Q_9\}$. 
For $\tilde P$ the ramification point of $\pi$ over $P$,
we see that $H(\tilde X, \tP)= \LA6, 13, 14, 15, 16\RA$
 using the following formula:
$$
h^0(2n \tP) = h^0(nP) + h^0(nP - D)
$$
for any non-negative integer $n$. We have
$h^0(12 \tP) = h^0(6P) + h^0(6P - 7P + P_1 + P_2) 
= 3 + h^0(P_1 + P_2 - P) = 3$
because of $P_i \neq P$ for $i = 1, 2$. 
Moreover, we have $h^0(14 \tP) = h^0(7P) + h^0(P_1 + P_2) = 4 + 1 = 5$,
which implies that $13, 14 \in H(\tilde X, \tP)$. 
The equalities
$$
h^0(16\tP) = h^0(8P) + h^0(P + P_1 + P_2) 
= 5 + 3 + 1 - 4 + h^0(\cK_X - P - P_1 - P_2) = 5 + h^0(3 P) = 7
$$
guarantee that $15, 16 \in H(\tilde X, \tP)$. 
We have 
$$
h^0(18 \tP) = h^0(9P)+h^0(2P +P_1 +P_2) 
= 6+4+1-4+h^0(\cK_X -2P -P_1 -P_2) = 7+h^0(2P) = 8,
$$
which implies that $17 \not\in H(\tilde X, \tP)$ and $18 \in H(\tilde X, \tP)$. 
Moreover, we get
$$
h^0(20\tP) = h^0(10P)+h^0(3P +P_1+P_2) 
= 7+5+1-4+h^0(\cK_X-3P -P_1-P_2) = 9+h^0(P) = 10.
$$
Hence, we have $19, 20 \in H(\tilde X, \tP)$. We obtain
$$
h^0(22 \tP) 
= h^0(11P) + h^0(4P + P_1 + P_2) 
= 8 + 6 + 1 - 4 + h^0(\cK_X - 4P - P_1 - P_2) = 11 + 1 = 12.
$$
Hence, we have $21, 22 \in H(\tilde X, \tP)$. Finally, we get
$$
h^0(24 \tP) = h^0(12P) + h^0(5P + P_1 + P_2) = 9 + 7 + 1 - 4 = 13.
$$
Hence, we have $23 \not\in H(\tilde X, \tP)$ and 
$24 \in H(\tilde X, \tP)$. 
We conclude that  $H(\tilde X, \tP) = \LA6, 13, 14, 15, 16\RA$.
\end{proof}

\section{Semigroup $H_4=\langle 3,7,8\rangle$}
\label{sec:378}

In our construction of $\sigma$-functions, the tool is a description
of the affine ring of the curve minus $\infty$, the Weierstrass point 
in question.
We begin with the semigroup $H_4=\langle 3,7,8\rangle$ because to treat our
motivating example $H_{12}$, the existence proof in Proposition
\ref{prop:KomedaProp1}
uses an explicit double cover of a curve with semigroup $H_4$.

\subsection{The curve as a monomial curve}\label{sec:378_1}
%Local and global properties with $\fGm$}
A monomial curve is an irreducible affine curve with $\mathbb{G}_m$-action,
where $\mathbb{G}_m$ is the multiplicative group of the complex numbers.
To construct our curve, we consider the semigroup ring of $H_4$.
We follow Pinkham's strategy \cite{P} both in the $H_4$ and $H_{12}$ cases,
namely, we start with an irreducible curve singularity
with $\mathbb{G}_m$ action, described parametrically by its semigroup ring.
We deform the matrix of the relations. This deformation space $\bar{U}$
``classifies the set of pairs consisting of a smooth and proper 
algebraic curve $X$ together with a point $P\in X$ with given 
semigroup $H(X,P)$. If $\bar{U}$ is non void, then we have constructed
directly
a compactification of a moduli space
for curves with points of semi group $H$'' 
(\cite{P}, Introduction).
We  have a result that ensures that $\bar{U}$ is nonempty in both cases.
In the case of $H_4$, Pinkham again \cite[Section 14]{P}(Sec. 14)
proves nonemptiness for semigroups with two generators
and conjectures it for three, proved later in \cite{rimvitulli}.
%as for $H_{12}$, we gave Komeda's new proof in section \ref{}.

\begin{proposition} \label{prop:Z4}
For the $k$-algebra homomorphism,
$$
	\varphi_{4} : k[Z] := k[Z_3, Z_7, Z_8] \to k[t^a]_{a\in M_4},\
\ Z_a\mapsto t^a, 
$$
the kernel of $\varphi_4$ is generated by
$\fZ{b} = 0$ $(b = 14, 15, 16)$ where
\begin{equation}
\fZ{14} = Z_7^2 - Z_3^2 Z_8, \quad
\fZ{15} = Z_7 Z_8 - Z_3^{5}, \quad
\fZ{16} = Z_8^2 - Z_3^3 Z_7. \quad
\label{eq:rel}
\end{equation}
\end{proposition}

\begin{proof} This follows 
from a result of Herzog's \cite{Her},
who showed that for any W-semigroup, 
 the number of generators of the kernel is two or three.
\end{proof}

Here we note that these relations are given by the $2 \times 2$ minors
of
\begin{equation}
\displaystyle{
\left|\begin{matrix}
 Z_3^2 & Z_7 & Z_8 \\
 Z_7 & Z_8  & Z_3^3\\
\end{matrix} \right|}.
\label{eq:2x2minor4Z}
\end{equation}

The monomial ring $B_{H_4} := k[t^a]_{a \in M_4}$ is given by
$$
B_{H_4} \simeq k[Z_3, Z_7, Z_8]/ \ker\varphi.
$$
We say that $Z_a$ has weight $a$.

To indicate the action of $\mathbb{G}_m$, defined as
$Z_a \mapsto g^a Z_a$ for $g\in \mathbb{G}_m$,
 and the induced action on the monomial ring
$B_{H_4}$,
we use the notation $\fGmf$: the superscript (4) is a convenience since we 
are going to relate several curves and different actions.

To construct a smooth curve $X_4$ with  $H(X_4, \infty) = H_4$,
we consider two plane affine curves $X_6$ and $X_7$, both with a unique
point at  infinity which we always denote by $\infty$,
corresponding to 
the equations:
$$
   f_{6,21}(x, y_7) := y_7^3 - k_7(x) ,
\quad k_7(x) := k_3(x) k_2(x)^2,
$$
and
$$
   f_{7,24}(x, y_8) := y_8^3 - k_8(x) ,
\quad k_8(x) := k_3(x)^2 k_2(x),
$$
where for distinct $\muc_a \in \CC$ $(a = 1, 2, \ldots, 5)$,
\begin{gather*}
\begin{split}
 k_3(x) &:= (x - \muc_1)(x - \muc_2)(x - \muc_3)
= x^3 + \lthree{1} x^2 + \lthree{2} x + \lthree{3}, \\
 k_2(x) &:= (x - \muc_4) (x - \muc_5)
= x^2 + \ltwo{1} x + \ltwo{2} . \\
\end{split}
\end{gather*}
Both curves have  singular points in their affine parts;
the singular points of the affine part of $X_6$ are $(b_a,0)$,
 $a=4,5$, whereas
those of $X_7$ are $(b_a,0)$, $a=1,2,3$.
We consider  the rings $R_6:=\CC[x, y_7]/(f_{6,21}(x,y_7))$ and
$R_7:=\CC[x, y_8]/(f_{7,24}(x,y_8))$ associated with $X_6$ and $X_7$ respectively.
The genera of the semigroups of $X_6$ and $X_7$ at  $\infty$
are $6=g(\LA3,7\RA)$ and  $7=g(\LA3,8\RA)$ respectively.

Due to the normalization theorem \cite[p.5, p.68]{Gri}(p.5,p.68) 
(Theorem \ref{thm:normal}),
 every complete, reduced, irreducible algebraic curve 
admits a normalization which is
unique up to isomorphism.
Although the functions $y_7$ and $y_8$ we introduced are defined
on different curves,
 the following identities 
are consistent because the zeroes of $f_{6,21}(x,y_7)$ and
 $f_{7,24}(x,y_8)$
 are related,
\begin{equation}
	y_7 y_8 = k_2(x) k_3(x), \quad
           y_8 = \frac{y_7^2}{(x-\muc_4)(x-\muc_5)}, \quad
           y_7 = \frac{y_8^2}{(x-\muc_1)(x-\muc_2)(x-\muc_3)}. \quad
\label{eq:rel378}
\end{equation}
Consistency is achieved by implementing
 the natural action of a primitive third root of unity
 $\zeta_3$ on $y_7$ and $y_8$ in $R_6$ and $R_7$:
on the first relation,
 we make the choice $a=0$ out of the possible $0 \le a \le 2$ for
$y_7 y_8 \mapsto\zeta_3^ak_2(x) k_3(x)$.

%Corresponding to Proposition \ref{prop:Z4},
As  normalization of these singular curves,
we consider the commutative ring,
$$
R_4:=\CC[x, y_7, y_8]/ (f_{14}, f_{15}, f_{16}),
$$
and the curve $\Spec R_4$.
Here we define $f_{14}, f_{15}, f_{16} \in \CC[x, y_7, y_8]$ by
$$
f_{14} = y_7^2 - y_8 k_2(x), \quad
f_{15} = y_7 y_8 - k_2(x) k_3(x), \quad
f_{16} = y_8^2 - y_7 k_3(x),
$$
which are also viewed as the $2 \times 2$ minors of
\begin{equation}
\displaystyle{
\left|\begin{matrix}
 k_2(x)  & y_{7} & y_{8} \\
 y_{7} & y_{8}  & k_3(x)\\
\end{matrix} \right|}, \quad
\label{eq:2x2minor4}
\end{equation}
as a deformation of (\ref{eq:2x2minor4Z}).
Note that when $x$ is infinite, so are 
$y_7$ and $y_8$.
$\fGmf$ acts on $R_4$ by sending $x$ and $y_a$ 
to  $g^{-3} x$, $g^{-a} y_a$ for
$g \in \fGmf$ and
 $a =  7, 8$,
since $-3$ and  $-a$ are the weights (meaning the valuation of the point,
negative of the 
order of pole)
of $x$ and $y_a$ at $\infty$.

Due to Nagata's Jacobian criterion
in Th. 30.10 of
\cite[Theorem 30.10]{Mat}, 
$\Spec R_4$ is non-singular except at $\infty$:
\begin{proposition}
For every $(x, y_7, y_8)$ where
$f_{14}, f_{15}$ and $f_{16}$ simultaneously vanish,
$$
\rank \cU_{4} = 2,
$$
where
$$
\cU_4:=
\begin{pmatrix}
\frac{\partial f_{14}}{\partial x} &
\frac{\partial f_{14}}{\partial y_7} &
\frac{\partial f_{14}}{\partial y_8} \\
\frac{\partial f_{15}}{\partial x} &
\frac{\partial f_{15}}{\partial y_7} &
\frac{\partial f_{15}}{\partial y_8} \\
\frac{\partial f_{16}}{\partial x} &
\frac{\partial f_{16}}{\partial y_7} &
\frac{\partial f_{16}}{\partial y_8} \\
\end{pmatrix}.
$$
\end{proposition}

\begin{proof}
The matrix $\cU_4$  is equal to
$$
\cU_4 =
\begin{pmatrix}
y_8 k_2'(x) & 2 y_7   & - k_2(x) \\
(k_2(x) k_3(x))' & y_8 &  y_7   \\
y_7 k_3'(x) & -k_3(x) & 2 y_8   \\
\end{pmatrix}.
$$
When $y_7$ vanishes, and in consequence $y_8$ vanishes,
$x$ must be one of $\{\muc_a\}_{a = 1, \ldots, 5}$.
For $x = \muc_1$,
$$
\cU_4 =
\begin{pmatrix}
0 & 0   & -k_2(\muc_1) \\
k_2(\muc_1)(\muc_1 - \muc_2)(\muc_1 - \muc_3) & 0 & 0\\ 
0 & 0   & 0 \\
\end{pmatrix}.
$$ 
Here $k_2(\muc_1)$ and
$(\muc_1 - \muc_2)(\muc_1 - \muc_3)$ don't vanish 
so the rank is 2. Similarly 
 for $x = \muc_a$ $a = 2, \ldots, 5$. 
On the other hand, when $y_7$ and $y_8$ are nonzero,
$$
\begin{pmatrix}
1 & & \\ 1 & -1 & 1\\ & & 1 \\
\end{pmatrix}
\begin{pmatrix}
y_8^2 & & \\ &   y_7 y_8 &\\ & & y_7^2  \\
\end{pmatrix}
\cU_4 =
\begin{pmatrix}
1 & & \\ 1 & -1 & 1\\ & & 1 \\
\end{pmatrix}
\begin{pmatrix}
y_8^3 k_2'(x) & 2 y_7 y_8^2  & -  y_8^2k_2(x) \\
y_7 y_8(k_2(x) k_3(x))' & y_7 y_8^2 &  y_7^2y_8   \\
y_7^3 k_3'(x) & -y_7^2 k_3(x)  &  2 y_7^2 y_8   \\
\end{pmatrix}
$$
shows that the rank is 2.
\end{proof}

Lastly,  the smooth curve 
 $X_4$ associated to the quotient field of the ring
$R_4$ is obtained by completing the affine piece 
 $\Spec R_4$ by one smooth point at infinity,
which we still denote by $\infty$;
this can be seen by introducing a local parameter at $\infty$
as in Sec.2 of \cite[Section 2]{Mul} (see Appendix B).
%.
There is a cyclic action on $X_4$ by
$$
\hzeta_3(x, y_7, y_8)= (x, \zeta_3 y_7, \zeta_3^2 y_8).
$$
\begin{proposition}\label{prop:homRaR4}
There is a natural homomorphism $R_a \to R_4$ 
for each $(a=6,7)$.
% and we have a natural morphism 
%$\Spec R_4 \to \Spec R_a$ for each $(a=6,7)$.
\end{proposition}
\begin{proof}
The natural injection
$\varphi:\CC[x, y_7]\to\CC[x, y_7, y_8]$ is a ring homomorphism,
because the relation
$\varphi^{-1}(f_{14}, f_{15}, f_{16}) = (f_{6,21})$ holds.
Similarly we have the result for $a=7$.
\end{proof}
\bigskip

In conclusion, the above shows that 
$X_4$ is the normalization of $X_6$ and $X_7$.

\subsection{The Weierstrass gaps and holomorphic one forms}
The Weierstrass gap sequences at $\infty$
are given by the following table. 
Indeed, we have $H(X_4, \infty) = H_4$.
In view of the construction of the curve,  functions that
have the order of poles in the complement of the gaps 
can be given by monomials in $R_4$.
We denote by $\phig{i}$ the sequence of such monic monomials 
in $R_g$ for $g= 6, 7$ and we add the notational device
$H^0$ in the subscript to signify functions,
as in the group $H^0(X_4\backslash\infty,\mathcal{O}_{X_4})$:
$\phiMf{j}$ for $R_4$, {\it{e.g.}},
$\phiMf{0}=1$, $\phiMf{1}=x$, $\phiMf{2}=x^2$, $\phiMf{3}=y_7$,
$\phiMf{4}=y_8$, $\phiMf{5}=x^3$, e.g..

\begin{gather*}
{\tiny{
\centerline{
\vbox{
	\baselineskip =10pt
	\tabskip = 1em
	\halign{&\hfil#\hfil \cr
        \multispan7 \hfil Table 1  The Weierstrass gaps and non-gaps at $\infty$ \hfil \cr
	\noalign{\smallskip}
	\noalign{\hrule height0.8pt}
	\noalign{\smallskip}
$\ $ \strut\vrule & 
0 &1 & 2 & 3 & 4 & 5 & 6 & 7 & 8 & 9 & 10 & 11 & 12 & 13 & 14 & 15 & 16\cr
\noalign{\smallskip}
\noalign{\hrule height0.3pt}
\noalign{\smallskip}
$X_6$ \strut\vrule & 
 1& - & - & $x$ & - & - & $x^2$& $y_7$ & - & $x^3$ & $x y_7$ & -& $x^4$
& $x^2y_7$ & $y_7^2$ & $x^5$ & $x^3y_7$\cr
$X_7$ \strut\vrule & 
 1& - & - & $x$ & - & - & $x^2$& - & $y_8$ & $x^3$ & - & $xy_8$ & $x^4$
& - & $x^2y_8$ & $x^5$ & $y_8^2$ \cr
$X_4$ \strut\vrule & 
 1& - & - & $x$ & - & - & $x^2$& $y_7$ & $y_8$ & $x^3$ & $x y_7 $ 
& $x y_8$ & $x^4$ & $x^2 y_7$ & $x^2y_8$ & $y_7 y_8$ & $x^3 y_7$ \cr
\noalign{\smallskip}
	\noalign{\hrule height0.8pt}
}
}
}
}}
\end{gather*}

We define the weight $\Ng(n)$ and $\NMf(n)$ by
$$
	\NMf(n) = -\wt(\phiMf{n}), \quad
	\Ng(n) = -\wt(\phig{n}), \quad
$$
where $\wt()$ is the negative of the order of pole at $\infty$.
These are consistent with the weights corresponding to the action  $\fGmf$,
whereas $R_g$ $(g=6,7)$ provides the Weierstrass gap sequences associated
with the semigroups $\LA3,7\RA$ and $\LA3,8\RA$.

The differentials of the first kind of the singular curves of
$X_6$ and $X_7$ are given by
$$
\nuIs{i} := \frac{\phis{i-1}dx}{3y_7^2}, \quad (i=1,2,3,4,5,6),
\quad
\nuIe{i} := \frac{\phie{i-1}dx}{3y_8^2}, \quad (i=1,2,3,4,5,6,7).
$$

We introduce the monomial $\phiHf{i}$ of $R_4$ in the
following Proposition,
where we add the notational device
$H^1$ in the subscript to indicate differentials,
as in the group $H^1(X_4\backslash\infty,\mathcal{O}_{X_4})$.
The properties of these monomials will be
investigated and used in Proposition 3.5,
Lemma \ref{lem:2theta4}, Proposition \ref{prop:addition}, and by
Serre duality. 

\begin{proposition}
The holomorphic one forms over $X_4$ can be expressed by
$$
\nuIf{i} := \frac{\phiHf{i-1} dx}{3y_7 y_8}, \quad
(i=1,2,3,4),
$$
where the monomial $\phiHf{i} \in R_4$ is defined by
$$
\phiHf{0} := y_7, \quad
\phiHf{1} := y_8, \quad
\phiHf{2} := x y_7, \quad
\phiHf{3} := x y_8. \quad
$$
\end{proposition}
Further  we define $\phiHf{4} := x^2 y_7$, $\phiHf{5} := x^2 y_8$, and
for $j > 5$, and let $\phiHf{j}$  be
defined as $\phiHf{j} := \phiMf{j+5}$.
We also define the weight $\NHf(n)$ by
$$
	\NHf(n) = -\wt(\phiHf{n})
$$
namely the order of pole of $\phiHf{n}$ at $\infty$,
which is  consistent with the action coming from $\fGmf$.

\begin{remark}\label{rmk:CanonDiv}{\rm{
 We make a remark on the nature of
the canonical divisor $\mathcal{K}_4$ of $X_4$. By writing
coordinates $(x,y_7,y_8)$ for a point
on $X_4$, the divisors of our basis of one-forms are:
%\begin{gather*}
%\begin{split}
%(\nuIf{1})&=(b_4,0,0)+(b_5,0,0)+4\infty, \\
%(\nuIf{2})&=(b_1,0,0)+(b_2,0,0)+(b_3,0,0)+3\infty, \\
%(\nuIf{3})&=\sum_{a=0}^2
%(0,\zeta_3^a\sqrt{3}{k_7(0)},\zeta_3^{2a}\sqrt{3}{k_8(0)})
%+(b_4,0,0)+(b_5,0,0), \\
%(\nuIf{4})&=\sum_{a=0}^2
%(0,\zeta_3^a\sqrt{3}{k_7(0)},\zeta_3^{2a}\sqrt{3}{k_8(0)})
%+(b_1,0,0)+(b_2,0,0)+(b_3,0,0),\\
%\end{split}
%\end{gather*}
%we emphasise that the canonical divisor $\cK_{X_4}$
%must include points in $X_4 \setminus \infty$.
\begin{gather*}
\begin{split}
(\nuIf{1})&=B_4+B_5+4\infty, \\
(\nuIf{2})&=B_1+B_2+B_3+3\infty, \\
(\nuIf{3})&=\sum_{a=0}^2
(0,\zeta_3^a\root 3\of{k_7(0)},\zeta_3^{2a}\root 3\of{k_8(0)})
+B_4+B_5 +\infty, \\
(\nuIf{4})&=\sum_{a=0}^2
(0,\zeta_3^a\root 3\of{k_7(0)},\zeta_3^{2a}\root 3\of{k_8(0)})+
B_1+B_2+B_3,
\end{split}
\end{gather*}
where $B_a :=(b_a, 0, 0)$ $(a=1,2,3,4,5)$.
Looking at the divisor of $(\nuIf{1})$ we see that
$B_4$, $B_5$ and $\infty$ correspond to
$P_1$ and $P_2$ and $P$ in Proposition 2.1.
We should note that we have the following linear equivalences
which play an important role in subsection 6.1:
%\mathbf{? Incorrect? the degrees are 8,4,6 in the three lines}
\begin{gather*}
\begin{split}
 \cK_{X_4} &\sim 2 (3 \infty - (B_4+B_5 - 2\infty))\\
  &\sim 2 (3 \infty - (B_1+B_2+B_3 - 3\infty))\\
  &\sim 6 \infty - (B_1+B_2+B_3+B_4+B_5 - 5\infty)\\
\end{split}
\end{gather*}
deduced by using
%\mathbf{? Incorrect? the left-hand-side is zero because k_2k_3=y_7y_8
%is a function on the curve, but the right-hand-side cannot be zero otherwise
%the curve would be hyperelliptic}
\begin{gather*}
\begin{split}
  B_1+B_2+B_3+B_4+B_5 - 5\infty &
\sim -(B_4+B_5 -2\infty)\\
&\sim 2(B_4+B_5 -2\infty)\\
   &\sim -(B_1+B_2 +B_3 -3\infty)\\
  & \sim 2(B_1+B_2 +B_3 -3\infty).\\
\end{split}
\end{gather*}
In fact, any
 positive canonical divisor
must include points in $X_4 \setminus \infty$.
This is contrast to the fact that for every $(n,s)$ curve,
there is a canonical divisor given by $(2g-2) \infty$, cf.
\cite{MP09}; recall (letting $n=r$
as in this reference) that  $g=(r-1)(s-1)/2$.
This is the reason we used the monomials 
$\phi_{H^1}^{(4)}$ besides the $\phi_{H^0}^{(4)}$'s.
}}
\end{remark}

The following proposition is self-evident:

\begin{proposition}
For $n>0$,
$
\sum_{i = 0}^n a_i \frac{\phiHf{i} d x}{ y_7 y_8}
$
belongs to $H^0(X_4\setminus \infty, \Omega_{X_4})$,
for any constants $a_i$, 
in particular to  $H^0(X_4, \Omega_{X_4})$ for $n\le 4$,
where $ \Omega_{X_4}$ is the sheaf of holomorphic one-forms.
\end{proposition}

\begin{lemma}
For a non-negative integer $n<5$, if the one form,
$$
\sum_{i = 0}^n a_i \frac{x^i d x}{ y_7 y_8}
\equiv\sum_{i = 0}^n a_i \frac{x^i d x}{ k_2(x) k_3(x)},
$$
is holomorphic over $X_4$, then each
$a_i$ vanishes.
\end{lemma}
\begin{proof}
For $n<5$, every term in 
$\sum_{i = 0}^n a_i \frac{x^i d x}{ y_7 y_8}$ has singularities
of different valences
at points in $X_4 \setminus \infty$.
\end{proof}

By letting 
$$
\Lambda_i^{(g)} = \Ng(g) - \Ng(i-1) -g + i -1,
\qquad
\Lambda_i^{(4)} = \NHf(4) - \NHf(i-1) -5 + i,
$$
the related Young diagrams $\mathcal Y_{4}
=(\Lambda_1^{(4)}, \Lambda_2^{(4)}, \Lambda_3^{(4)}, \Lambda_4^{(4)})
$,
$\mathcal Y_{6}
=(\Lambda_1^{(6)}, \Lambda_2^{(6)}, \cdots, \Lambda_6^{(6)})$ 
and $\mathcal Y_{7}
=(\Lambda_1^{(7)}, \Lambda_2^{(7)}, \cdots, \Lambda_7^{(7)}) $
are given by
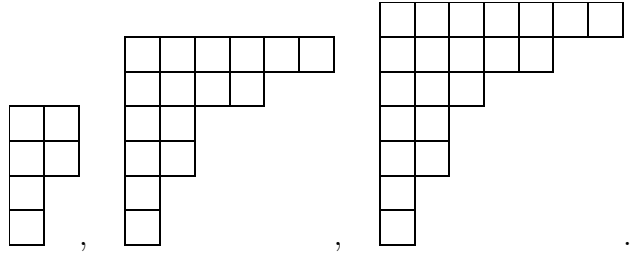
\begin{center}
\begin{equation*}
\yng(2,2,1,1), \quad
\yng(6,4,2,2,1,1), \quad
\yng(7,5,3,2,2,1,1).
\end{equation*}
\end{center}
\def\theequation{\thesection.\arabic{equation}}
Since the Young diagram $\mathcal Y_{4}$ is not symmetric,
the semigroup $H_4$ is also not symmetric, where
a numerical semigroup is symmetric if and only if $2 g - 1$
occurs in the gap sequence \cite{RG}.
We should note that the elements
of a minimal set of generators, $\alpha_{g-i}(L)$ in
(\ref{eq:alphaL}), correspond to $\Lambda_i^{(g)}$ for each $X_g$ ($g=4,6,7$).

Also self-evident is the following:
\begin{lemma}
By using the same letters in source and target under
 the homomorphism of Proposition \ref{prop:homRaR4}
(as well as identifying a coset by a standard representative),
the holomorphic one forms $\nuIf{i}$ over $X_4$ satisfy:
$$
\nuIf{i} = \tnuIs{i+3} = \tnuIe{i+4}, \quad (i=1,2,3,4),
$$
where
$$
  \tnuIs{i} := \nuIs{i} , \quad(i=4,6) , \quad
  \tnuIs{3} := \nuIs{3} + \ltwo{1}\nuIs{2} +\ltwo{2}\nuIs{1}, \quad
  \tnuIs{5} := \nuIs{5} + \ltwo{1}\nuIs{3} +\ltwo{2}\nuIs{2}, \quad
$$
$$
  \tnuIe{i} := \nuIe{i} , \quad(i=4,6) , \quad
  \tnuIe{5} := \nuIe{5} + \lthree{1}\nuIe{3} 
+\lthree{2}\nuIe{2} +\lthree{3}\nuIe{1},
$$
$$
  \tnuIe{7} := \nuIe{7} +\lthree{1}\nuIe{5} +\lthree{2}\nuIe{3}
+\lthree{3}\nuIe{2}. 
$$
\end{lemma}

As is standard practice, we choose a basis
$\alphaf{i}, \betaf{j}$  $ (1\leqq i, j\leqq 4)$
of $H_1(X,\ZZ)$ such that the intersection numbers are
$\alphaf{i}\cdot\alphaf{j}=\betaf{i}\cdot\betaf{j}= 0$ and 
$\alphaf{i}\cdot\betaf{j}=\delta_{ij}$, 
and we denote the period matrices by
\begin{equation}
\begin{split}
   \left[\,\omegafp{}  \ \omegafpp{}  \right]&= 
\frac{1}{2}\left[\int_{\alphaf{i}}\nuIf{j} \ \ \int_{\betaf{i}}\nuIf{j}
\right]_{i,j=1,2, \cdots, 4}.
   \label{eq:42.4}
\end{split}
  \end{equation} 
Let $\Lambda_4$ be a lattice 
generated by $\omegafp{}$ and $\ \omegafpp{}$.

%For a point $P \in X_4$, the Abel map is defined,
%$$
%    \hu(P) = \int^P_{\infty} \nuIf{} \in \CC^4,
%$$
%and for $k$ points $P_1, P_2, \ldots, P_k \in X_4$, the Abel map
%$\hu(P_1, \ldots, P_k)$ is defined by $\sum_{i=1}^k \hu(P_i)$.
%Then we obtain the Jacobian $\cJ_4$ by
For a point $P \in X_4$, we define the Abel map,
$$
    \hu_o(P) := \int^P_{\infty} \nuIf{} \in \CC^4,
$$
and for $k$ points $P_1, P_2, \ldots, P_k \in X_4$,
 also a shifted Abel map
\begin{equation}
\hu(P_1, \ldots, P_k):=
\hu_o(P_1, \ldots, P_k)
+\hu_o(B_4, B_5),
\label{eq:shiftedAmap}
\end{equation}
where $\hu_o(P_1, \ldots, P_k):=\sum_{i=1}^k \hu_o(P_i)$.
As shown in (\ref{eq:-1u}), the shift is devised so that the
$(-1)$-operation on a complex vector $\hu(P_1, \ldots, P_k)$
is consistent with equivalence of divisors on the curve.
The ``$u$'' are not really functions of ordered $k$-tuples
of points, since one needs to choose a path of integration, but
it is conventional (and harmless) to use this abbreviated notation.
%of the universal covering of
%$X_4$ but we refer them shortly.

Then we obtain the Jacobian $\cJ_4$ by
$$
\kappa : \CC^4 \to \cJ_4 =\CC^4/\Lambda_4,
$$
and the subvariety $\WWf{k}$ by
$$
   \WWf{k} := \kappa \hu(S^k X_4).
$$

%%%%%%%%%%%%%%%%%%%%%%%%%%%%%%%%%%%%%%%%%%%%%%%%%%%%%%%%%%%%%%%%%%%%%%%%%%%%%%%
%%%%%%%%%%%%%%%%%%%%%%%%%%%%%%%%%%%%%%%%%%%%%%%%%%%%%%%%%%%%%%%%%%%%%%%%%%%%%%%
\bigskip
Hereafter we use the convention that for $P_a \in X_4$,
$P_a$ is expressed by $(x_a, y_{7,a}, y_{8,a})$
or $(x_{P_a}, y_{7,{P_a}}, y_{8,{P_a}})$.
By letting $u:=\hu(P_1, \cdots, P_4)$, we have
\begin{equation}
\begin{pmatrix}
\partial/\partial{u_1}\\
\partial/\partial{u_2}\\
\partial/\partial{u_3}\\
\partial/\partial{u_4}
\end{pmatrix}
=
\Psi_4^{(4)-1}
\begin{pmatrix}
3y_{7,1}y_{8,1} \partial/\partial{x_1}\\
3y_{7,2}y_{8,2} \partial/\partial{x_2}\\
3y_{7,3}y_{8,3} \partial/\partial{x_3}\\
3y_{7,4}y_{8,4} \partial/\partial{x_4}
\end{pmatrix},
\end{equation}
where 
\begin{equation}
\Psi_4^{(4)}:=
\begin{pmatrix}
\phiHf{0}(P_1) & \phiHf{1}(P_1) & \phiHf{2}(P_1) & \phiHf{3}(P_1) \\
\phiHf{0}(P_2) & \phiHf{1}(P_2) & \phiHf{2}(P_2) & \phiHf{3}(P_2) \\
\phiHf{0}(P_3) & \phiHf{1}(P_3) & \phiHf{2}(P_3) & \phiHf{3}(P_3) \\
\phiHf{0}(P_4) & \phiHf{1}(P_4) & \phiHf{2}(P_4) & \phiHf{3}(P_4) 
\end{pmatrix}.
\label{eq:partialInPsi4}
\end{equation}
In other words,
$$
\sum_{i=1} \epsilon_i\frac{\partial}{\partial u_i}
=
|{\Psi_4^{(4)}}^{-1}|
\left|
\begin{matrix}
\phiHf{0}(P_1) & \phiHf{1}(P_1) & \phiHf{2}(P_1) & \phiHf{3}(P_1) &
3y_{7,1}y_{8,1} \partial/\partial{x_1}\\
\phiHf{0}(P_2) & \phiHf{1}(P_2) & \phiHf{2}(P_2) & \phiHf{3}(P_2) &
3y_{7,2}y_{8,2} \partial/\partial{x_2}\\
\phiHf{0}(P_3) & \phiHf{1}(P_3) & \phiHf{2}(P_3) & \phiHf{3}(P_3) &
3y_{7,3}y_{8,3} \partial/\partial{x_3}\\
\phiHf{0}(P_4) & \phiHf{1}(P_4) & \phiHf{2}(P_4) & \phiHf{3}(P_4) &
3y_{7,4}y_{8,4} \partial/\partial{x_4}\\
\epsilon_1 & \epsilon_2 & \epsilon_3 & \epsilon_4 &
\end{matrix}
\right|.
$$

As these relations hold for the image of the Abel map,
we also have similar relations for a stratum $\kappa \hu(S^k X_4)$ 
when $k<4$
using the submatrices of $\Psi_4^{(4)}$ as in \cite{MP12}.
Further we have a natural relation:
\begin{equation}
          \sum_{i, j=1}^4 \phiHf{i-1}(P_1) \phiHf{j-1}(P_2)
	\frac{\partial^2 }
{\partial \hu_i(P_1)\partial \hu_j(P_2)}
= 9 y_{7,1} y_{8,1} y_{7,2} y_{8,2}
	\frac{\partial^2 }{\partial x_1\partial x_2}.
\label{eq:partial_H4}
\end{equation}
%%%%%%%%%%%%%%%%%%%%%%%%%%%%%%%%%%%%%%%%%%%%%%%%%%%%%%%%%%%%%%%%%%%%%%%%%%%%%%%
%%%%%%%%%%%%%%%%%%%%%%%%%%%%%%%%%%%%%%%%%%%%%%%%%%%%%%%%%%%%%%%%%%%%%%%%%%%%%%%

\subsection{Differentials of the second and the third kinds}
\label{differential forms 4}

Following 
the EEL-construction \cite{EEL,BEL2} for  $(n,s)$ curves,
we produce an algebraic representation of
 a differential form which, up to a tensor of holomorphic one-forms,
is equal to the fundamental normalized differential of the 
second kind in Cor.2.6 of \cite[Corollary 2.6]{F1}.
We extend the EEL-construction
to the space curve $X_4$.
\begin{definition}
A two-form $\Omegaf(P_1, P_2)$ on $X_4\times X_4$ is called 
a {\it fundamental differential of the second kind} if it is symmetric, 
\begin{equation}
\Omegaf(P_1, P_2)=\Omegaf(P_2, P_1), 
\label{eq3.1.6}  
\end{equation}
 has its only pole (of second order) along the diagonal of  $X_4\times X_4$, 
and in the vicinity of each point $(P_1,P_2)$ 
 is expanded in power series as 
\begin{equation}
\Omegaf(P_1, P_2)=\Big(\frac{1}{(t_{P_1} -t_{P_2} ')^2  } +d_\ge(1)\Big)
   d t_{P_1} \otimes d t_{P_2}
\ \ (\hbox{\rm as}\  P_1\rightarrow P_2), 
\label{expansion}
\end{equation}
where $t_P$ is a local coordinate at the point $P \in X_4$.
\end{definition}

\begin{proposition}
Let $\Sigmaf\big(P_1, P_2\big)$ be the following form,
\begin{equation}
   \Sigmaf\big(P, Q\big)
   :=\frac{ y_{7,P} y_{8,P} +y_{7,P} y_{8,Q} +y_{7,Q} y_{8,P}}
{(x_P - x_Q) 3 y_{7,P} y_{8,P} } d x_P .
\end{equation}
Then $\Sigmaf(P, Q)$ has the properties 
\begin{enumerate}
\item $\Sigmaf(P, Q)$ as a function of $P$ is singular at 
$Q=(x_Q, y_{7,Q}, y_{8,Q})$ and  $\infty$, and 
is not singular at $\hzeta_3(Q)= (x_Q, \zeta_3 y_{7,Q}, \zeta_3^2 y_{8,Q})$.

\item $\Sigmaf(P, Q)$ as a function of $Q$ is singular at $P$ and
$\infty$.
\end{enumerate}
\end{proposition}

\begin{proof}
Direct computation.
\end{proof}

The following holds for non-singular $X_4$:
\begin{proposition} \label{prop:dSigma}
There exist  differentials $\nuIIf{j}=\nuIIf{j}(x,y)$ $(j=1, 2, \cdots, 4)$ 
of the second kind such that
they have
 a simple pole at $\infty$ and satisfy the relation,
\begin{equation}
\begin{split}
  &d_{Q} \Sigmaf\big(P, Q\big) - 
  d_{P} \Sigmaf\big(Q, P\big)\\
   &\quad\quad=
     \sum_{i = 1}^4 \Bigr(
         \nuIf{i}(Q)\otimes \nuIIf{i}(P)
        - \nuIf{i}(P)\otimes \nuIIf{i}(Q)
     \Bigr)
   \label{eq3.4}, 
\end{split}
\end{equation} 
where
\begin{equation}
 d_{Q} \Sigmaf\big(P, Q\big)
   :=d x_P \otimes d x_Q\frac{\partial }{ \partial x_Q}
   \frac{
y_{7,P} y_{8,P}
+y_{7,P} y_{8,Q}
+y_{7,Q} y_{8,P}}
{(x_P - x_Q) 3 y_{7,P}  y_{8,P} } d x_P .
\end{equation} 
The set of differentials $\{\nuIIf{1}$, $\nuIIf{2}$, 
$\nuIIf{3}$, $\nuIIf{4}\}$
 is
determined modulo the linear space spanned by
$\langle\nuIf{j}\rangle_{j=1, \ldots, 4}$ and it has
 representatives
\end{proposition}
\begin{equation*}
\begin{split}
\nuIIf{4} &= \frac{ -x^2 dx }{3 y_{8} } ,\\
\nuIIf{3} &= \frac{ -(2x^2 +\ltwo{1} x)  dx }{3 y_{7} } ,\\
\nuIIf{2} &= \frac{ -\left(4x^3 
 +(3\lthree{1} + 2\ltwo{1}) x^2 +(2\lthree{2} + \ltwo{1}\lthree{1}) x
+\lthree{3}\right) dx }{3 y_{8} } ,\\
\nuIIf{1} &= \frac{ -\left(5x^3 
 +(3\lthree{1} + 4\ltwo{1}) x^2 +
(\lthree{2} +2 \ltwo{1}\lthree{1} +3\ltwo{2}) x
+\ltwo{2} \lthree{1} \right) dx }{3 y_{7} } .\\
\end{split}
\end{equation*} 
We will henceforth use this basis $\nuIIf{i}$.

\begin{proof}
\begin{equation*}
\begin{split}
&   \frac{\partial }{ \partial x_Q}
   \frac{ y_{7,P} y_{8,P} +y_{7,P} y_{8,Q} +y_{7,Q} y_{8,P}}
{(x_P - x_Q) 3 y_{7,P}  y_{8,P} } d x_P \\
&=\frac{1}{(x_P - x_Q) 
9y_{7,P} y_{8,P} y_{7,Q} y_{8,Q}} 
\Bigr[
\frac{
3(y_{7,P} y_{8,P} +y_{7,P} y_{8,Q} +y_{7,Q} y_{8,P})
y_{7,Q} y_{8,Q})}{(x_P - x_Q) }\\
& \quad\quad
+ 
\Bigr(y_{7,P} 
\frac{y_{7,Q}}{y_{8,Q}}(
2k_{3,Q} k_{3,Q}' k_{2,Q} +k_{3, Q}^2  k_{2,Q}')
- y_{8,P} 
\frac{y_{8,Q}}{y_{7,Q}}(
2k_{3,Q} k_{2,Q}k_{2,Q}'+ k_{3, Q}' k_{2,Q}^2))
\Bigr) \Bigr].\\
\end{split}
\end{equation*}
Here we set $k_{a,P} = k_a(x_P)$ and $k_{a,P}' = d k_a(x_P)/d x_P$.
The result follows from the equalities:
$$
   \frac{\partial }{ \partial x_Q}
   \frac{ y_{7,P} y_{8,P} +y_{7,P} y_{8,Q} +y_{7,Q} y_{8,P}}
{(x_P - x_Q) 3 y_{7,P}  y_{8,P} }  
 -
   \frac{\partial }{ \partial x_P}
   \frac{ y_{7,Q} y_{8,Q} +y_{7,Q} y_{8,P} +y_{7,P} y_{8,Q}}
{(x_Q - x_P) 3 y_{7,Q}  y_{8,Q} }  
$$
is equal to
\begin{equation*}
\begin{split}
&=\frac{1}{(x_P - x_Q) 
9y_{7,P} y_{8,P} y_{7,Q} y_{8,Q}} 
\Bigr[
\frac{
3 (y_{7,P} y_{8,Q} +y_{7,Q} y_{8,P})
(y_{7,Q} y_{8,Q} -y_{7,P} y_{8,P}) }{(x_P - x_Q) }
\\
& \quad\quad
-
y_{7,P} y_{8,Q}\Bigr(
2k_{3,Q}' k_{2,Q} +k_{3, Q}  k_{2,Q}'
-2k_{3,P} k_{2,P}'- k_{3, P}' k_{2,P}\Bigr)\\
& \quad\quad\quad
+y_{7,Q} y_{8,P}\Bigr(
2k_{3,P}' k_{2,P} +k_{3, P}  k_{2,P}'
-2k_{3,Q} k_{2,Q}'- k_{3 Q}' k_{2,Q}
\Bigr) \Bigr]\\
&=\frac{1}{
9y_{7,P} y_{8,P} y_{7,Q} y_{8,Q}} 
\left(A_4(P, Q) - A_4(Q, P)\right),
\end{split}
\end{equation*}
where
\begin{equation*}
\begin{split}
A_4(P, Q) = 
y_{7,P} y_{8,Q}
& \Bigr(5x_Q^3 
 +(3\lthree{1} + 4\ltwo{1}) x_Q^2 +
(\lthree{2} +2 \ltwo{1}\lthree{1} +3\ltwo{2}) x_Q\\
&-4 x_P^3-(3\lthree{1}+2\ltwo{1}) x_P^2 - (2\lthree{2}
 + \ltwo{1}\lthree{1}) x_P-\lthree{3}\\
&+ (2x_Q^2 +\ltwo{1} x_Q)x_P 
 -(2x_P^2 +\ltwo{1} x_P)
-x_P^2 x_Q +\lthree{1}\ltwo{2} \Bigr).\\
\end{split}
\end{equation*}
Thus we have the result.
\end{proof}

%%%%%%%%%%%%%%%%%%%%%%%%%%%%%%%%%%%%%%%%%%%%%%%%%%%%%%%%%%%%%%%%%%%%%%%%%%%
%%%%%%%%%%%%%%%%%%%%%%%%%%%%%%%%%%%%%%%%%%%%%%%%%%%%%%%%%%%%%%%%%%%%%%%%%%%%%%

\begin{corollary}
\label{cor:Sigmaf}
\begin{enumerate}
\item
The one-form 
$$
\Pif_{P_1}^{P_2}(P):= \Sigmaf(P, P_1)dx -  \Sigmaf(P, P_2)dx  
$$
is a differential of the third kind,  whose only 
(first-order) poles are
$P=P_1$ and $P=P_2$, with residues $+1$ and $-1$ 
respectively.  

\item
The fundamental differential of the second kind
 $\Omegaf(P_1, P_2)$ is given by
\begin{equation} 
\begin{split} 
\Omegaf(P_1, P_2) &= d_{P_2} \Sigmaf(P_1, P_2) 
     +\sum_{i = 1}^g \nuIf{i}(P_1)\otimes \nuIIf{i}(P_2)\\
  &=\frac{\Ff(P_1, P_2)dx_1 \otimes dx_2}
{(x_1 - x_2)^2 9 
y_{7,P_1}
y_{8,P_1}
y_{7,P_2}
y_{8,P_2}},
\label{eq:realization4}   
\end{split} 
\end{equation}
where $\Ff$ is an element of $R_4 \otimes R_4$.
\end{enumerate}
\end{corollary}
\begin{proof} 
Straightforward computation.
\end{proof} 

\begin{lemma} 
\label{lemma:limFphi4}
We have
\begin{equation} 
\lim_{P_1 \to \infty} 
\frac{\Ff(P_1, P_2)}{\phiHf{3}(P_1)(x_1 - x_2)^2}
 = \phiHf{4}(P_2) =x_{P_2} y_{7,P_2}.
\label{eq:limF4}
\end{equation}
\end{lemma}

\begin{proof}
A calculation, using the formula for $F^{(4)}$ from \ref{cor:Sigmaf} 
and the expression for 
$A_4$ in the proof of Proposition \ref{prop:dSigma},
gives the result.
\end{proof}

\medskip

Hereafter for $P=(x, y_7, y_8)$,
we expand $h_i(P)= 3 y_7 y_8 \nuIIf{i}(P) / dx $,
$1\le i\le g,$ in the monomial  representatives for
elements of $R_4$
  to avoid carrying holomorphic
one-forms.

For later convenience we introduce the notation:
\begin{equation}
\label{eq:Omega_def4}
\begin{split}
\Omegaf^{P_1, P_2}_{Q_1, Q_2}
 &:= \int^{P_1}_{P_2} \int^{Q_1}_{Q_2} \Omegaf(P, Q) \\
 &= \int^{P_1}_{P_2} (\Sigmaf(P, Q_1) - \Sigmaf(P, Q_2)) 
 +\sum_{i = 1}^4 \int^{P_1}_{P_2} \nuIf{i}(P) 
\int^{Q_1}_{Q_2} \nuIIf{i}(P).
\end{split}
\end{equation}

%%%%%%%%%%%%%%%%%%%%%%%%%%%%%%%%%%%%%%%%%%%%%%%%%%%%%%%%%%%%%%%%%%%%%%%%%%%
%%%%%%%%%%%%%%%%%%%%%%%%%%%%%%%%%%%%%%%%%%%%%%%%%%%%%%%%%%%%%%%%%%%%%%%%%%%%%%

%%%%%%%%%%%%%%%%%%%%%%%%%%%%%%%%%%%%%%%%%%%%%%%%%%%%%%%%%%%%%%%%%%%%%%%%%%%
%%%%%%%%%%%%%%%%%%%%%%%%%%%%%%%%%%%%%%%%%%%%%%%%%%%%%%%%%%%%%%%%%%%%%%%%%%%%%%

\bigskip
\bigskip

\section{Semigroup $H=\langle 6,13,14,15,16\rangle$}
\label{sec:613141516}

In this section, we investigate a non-singular curve whose Weierstrass
semigroup at one point $\infty$ is the
 numerical semigroup $H_{12}$ generated by
 $M_{12}=\{6,13,14,15,16\}$. Again we follow Pinkham's program,
deform the local equations of a smooth affine curve with
a given monomial ring at $\infty$, and use Komeda's
result to ensure that there exist a global curve in the deformation space.

\subsection{The curve as a monomial curve}
The proof of the following Proposition is contributed by Komeda.
\begin{proposition} \label{prop:Z12} 
Let $B_{H_{12}}$ be the monomial ring which is given by
$k[t^a]_{a \in M_{12}}$
for the numerical semigroup $H_{12}$.
For the  $k$-algebra homomorphism,
$$
	\varphi_{12} : k[Z] := k[Z_6, Z_{13}, Z_{14}, Z_{15}, Z_{16}] 
\to k[t^a]_{a \in M_{12}}, \ Z_a\mapsto t^a,
$$
where $Z_a$ has weight  $a = 6, 13, 14, 15, 16$,
the kernel of $\varphi_{12}$ is generated by the following
relations $\fZ{12,b}$ $(b = 1, \cdots, 9)$,
\begin{equation}
\begin{array}{lll}
\fZ{12,1} = Z_{13}^2 - Z_{6}^2 Z_{14}, \quad
&\fZ{12,2} = Z_{13} Z_{14} - Z_{6}^2 Z_{15}, \quad
&\fZ{12,3} = Z_{14}^2 - Z_{13} Z_{15} , \\
%\begin{equation*}
\fZ{12,4} = Z_{14}^2 - Z_6^2 Z_{16}, \quad
&\fZ{12,5} = Z_{13} Z_{16} - Z_{14} Z_{15} , \quad
&\fZ{12,6} = Z_{15}^2 - Z_6^5, \\
%\end{equation*}
%\begin{equation*}
\fZ{12,7} = Z_{14} Z_{16} - Z_6^5, \quad
&\fZ{12,8} = Z_{15} Z_{16} - Z_6^3 Z_{13} ,\quad
%&\fZ{12,9} = Z_{16}^2 - Z_6 Z_{13}^2.\\
&\fZ{12,9} = Z_{16}^2 - Z_6^3 Z_{14}.\\
%Z_{14}^3 = Z_6^7, \quad
%Z_{16}^3 = Z_6^8,\quad
%Z_{13}^6=Z_6^{13}, \quad
\label{eq:Z12}
\end{array}
\end{equation}
\end{proposition}
\begin{proof}
We set $a_1=6$, $a_{2}=13$, $a_{3}=14$, $a_{4}=15$ and $a_{5}=16$.
Let $\varphi_{12} : k[Z_6,Z_{13},Z_{14},Z_{15},Z_{16}]\longrightarrow 
k[H_{12}]=k[t^h]_{h\in H_{12}}$ be the $k$-algebra 
homomorphism which sends  $Z_i$ to $t^{a_i}$ for each $i$.
Then the ideal $I=\ker \varphi$ is generated by
$$f_{12,1}^{(Z)}=Z_{13}^2-Z_6^2Z_{14}\mbox{, }
f_{12,2}^{(Z)}=Z_{13}Z_{14}-Z_6^2Z_{15}\mbox{, }
f_{12,3}^{(Z)}=Z_{14}^2-Z_{13}Z_{15},
$$
$$f_{12,4}^{(Z)}=Z_{14}^2-Z_6^2Z_{16}\mbox{, }
f_{12,5}^{(Z)}=Z_{13}Z_{16}-Z_{14}Z_{15}\mbox{, }
f_{12,6}^{(Z)}=Z_{15}^2-Z_6^5,
$$
$$f_{12,7}^{(Z)}=Z_{14}Z_{16}-Z_6^5\mbox{, }
f_{12,8}^{(Z)}=Z_{15}Z_{16}-Z_6^3Z_{13}\mbox{, }
f_{12,9}^{(Z)}=Z_{16}^2-Z_6^3Z_{14}.
$$
We set $\alpha_i=\min\{\alpha\in \mathbb{N}_0 \mid \alpha >0, 
\alpha a_i\in\langle a_{1},\ldots,a_{i-1},a_{i+1},\ldots,a_{5}\rangle\}$.
Then it is easy to check that 
$\alpha_1=5$, $\alpha_{2}=\alpha_{3}=\alpha_{4}=\alpha_{5}=2$ and

$$
5a_1=2a_{4}\mbox{, }2a_{2}=2a_1+a_{3}\mbox{, }
2a_{4}=2a_6+a_{5}\mbox{, }2a_{4}=a_{3}+a_{5}
\mbox{, }2a_{5}=a_1+2a_{2}.
$$
We look for all identities of type:
$$\beta_i a_i+\beta_ja_j=\gamma_k a_k+\gamma_\ell a_\ell+\gamma_m a_m$$
with $0<\beta_i<\alpha_i$, $0<\beta_j<\alpha_j$, 
$0<\gamma_k<\alpha_k$, $0<\gamma_\ell<\alpha_\ell$ 
and $0\leqq\gamma_m<\alpha_m$.
Then we get the following four relations:
$$
a_{2}+a_{3}=2a_1+a_{4}\mbox{, }
a_{2}+a_{4}=2a_1+a_{5}\mbox{, }
a_{2}+a_{5}=a_{3}+a_{4}\mbox{, }
a_{4}+a_{5}=3a_1+a_{3}.
$$
It is easy to check that the polynomials given in the statement belong to $I$.
Let $J$ be the ideal generated by the above nine polynomials.
To prove $I=J$ it suffices to show the following:
If 
$$
f=\prod_iZ_{a_i}^{\beta_i}-\prod_iZ_{a_i}^{\gamma_i}
\in \ker\varphi \mbox{\rm \ with }\beta_i\gamma_i=0,
 \mbox{ every }i,
$$
then $f\in J$.
\vskip3mm
\noindent
{\it Case }1. 
Let 
$Z_{a_i}^{\beta_i}-Z_{a_j}^{\gamma_j}Z_{a_k}^{\gamma_k}
Z_{a_\ell}^{\gamma_\ell}Z_{a_m}^{\gamma_m}\in I$ with 
$\beta_i\geqq \alpha_i$, $\gamma_j>0$, $\gamma_k\geqq 0$, 
$\gamma_\ell\geqq 0$ and $\gamma_m\geqq 0$.

Let $i=6$.
We have
$$
f=Z_6^{\beta_1-5}(-f_{12,6}^{(Z)})
+Z_6^{\beta_1-5}Z_{15}^2-Z_{a_j}^{\gamma_j}M\equiv 
Z_6^{\beta_1-5}Z_{15}^2-Z_{a_j}^{\gamma_j}M \mbox{ mod }J,
$$
where $M$ is some monomial.
Hence, if $j=4$, then we may decrease the degree of $f$.
If $j=3$ or $5$, using $f_{12,6}^{(Z)}-f_{12,7}^{(Z)}$ 
we may decrease the degree of $f$.
We may assume that $f=Z_6^{\beta_1}-Z_{13}^{\gamma_{2}}$ with 
$\gamma_{2}\geqq \alpha_{2}$.
Hence, we have
$$
f=Z_6^{\beta_1}-Z_{13}^{\gamma_{2}-2}f_{12,1}^{(Z)}
-Z_{13}^{\gamma_{2}-2}Z_6^2Z_{14}\equiv Z_6^{\beta_1}
-Z_{13}^{\gamma_{2}-2}Z_6^2Z_{14} \mbox{ mod }J,
$$
which implies that we may decrease the degree of $f$.

Let $i=2$.
We have
$$
Z_{13}^{\beta_{2}}=Z_{13}^{\beta_{2}-2}f_{12,1}^{(Z)}
+Z_{13}^{\beta_{2}-2}Z_6^2Z_{14}.
$$
Hence, if $j=1$ or $3$, then we may decrease the degree of $f$.
Let $f=Z_{13}^{\beta_{2}}-Z_{15}^{\gamma_{4}}Z_{16}^{\gamma_{5}}$.
Using $f_{12,8}^{(Z)}\in J$ we may assume that 
$\gamma_{4}\geqq 2$ or $\gamma_{5}\geqq 2$.
If $\gamma_{5}\geqq 2$, by $f_{12,9}^{(Z)}\in J$ 
we may decrease the degree of $f$.
Let $\gamma_{4}\geqq 2$.
Then
$$
f=Z_{13}^{\beta_{2}-2}f_{12,1}^{(Z)}
+Z_{13}^{\beta_{2}-2}Z_6^2Z_{14}-Z_{15}^{\gamma_{4}-2}
(f_{12,6}^{(Z)}-f_{12,7}^{(Z)})Z_{16}^{\gamma_{5}}
-Z_{15}^{\gamma_{4}-2}Z_{14}Z_{16}^{\gamma_{5}+1}.
$$
Hence, we may decrease the degree of $f$.

Let $i=14$.
Since we obtain
$$
Z_{14}^{\beta_{3}}=Z_{14}^{\beta_{3}-2}f_{12,4}^{(Z)}
+Z_6^2Z_{14}^{\beta_{3}-2}Z_{16}
=Z_{14}^{\beta_{3}-2}f_{12,3}^{(Z)}
+Z_{13}Z_{14}^{\beta_{3}-2}Z_{15},
$$
which implies that we may decrease the degree of $f$.

Let $i=4$.
In view of
$$
Z_{15}^{\beta_{4}}=Z_{15}^{\beta_{4}-2}(f_{12,6}^{(Z)}
-f_{12,7}^{(Z)})+Z_{14}Z_{15}^{\beta_{4}-2}Z_{16}
=Z_{15}^{\beta_{4}-2}f_{12,6}^{(Z)}+Z_6^5Z_{15}^{\beta_{4}-2},
$$
we may assume that $f=Z_{15}^{\beta_{4}}-Z_{13}^{\gamma_{2}}$.
Hence, we get
$$
f=Z_{15}^{\beta_{4}-2}f_{12,6}^{(Z)}-Z_{13}^{\gamma_{2}-2}
f_{12,1}^{(Z)}+Z_6^5Z_{15}^{\beta_{4}-2}-Z_6^2Z_{13}^{\gamma_{2}-2}Z_{14}.
$$
We may decrease the degree of $f$.

Let $i=16$.
Since we have
$$
Z_{16}^{\beta_{5}}=Z_{16}^{\beta_{5}-2}f_{12,9}^{(Z)}
+Z_6Z_{13}^2Z_{16}^{\beta_{5}-2},
$$
we may assume that 
$f=Z_{16}^{\beta_{5}}-Z_{14}^{\gamma_{3}}Z_{15}^{\gamma_{4}}$.
In view of $f_{12,5}^{(Z)}\in J$ we may assume that 
$\gamma_{14}\geqq 2$ or $\gamma_{4}\geqq 2$.
By $f_{12,4}^{(Z)}$ and $f_{12,6}^{(Z)}-f_{12,7}^{(Z)}\in J$ 
we may decrease the degree of $f$.
\vskip3mm 
It is sufficient to show the following case about the remainder.
\vskip3mm
\noindent
{\it Case 2.} Let 
$f=Z_{a_i}^{\beta_i}Z_{a_j}^{\beta_j}
-Z_{a_k}^{\gamma_k}Z_{a_\ell}^{\gamma_\ell}Z_{a_m}^{\gamma_m}\in I$
with $\beta_i>0$, $\beta_j>0$, $\gamma_k>0$, $\gamma_\ell>0$ 
and $\gamma_m\geqq 0$.
By the proof of Case 1 we may assume that $0<\beta_i<\alpha_i$, 
$0<\beta_j<\alpha_j$, $0<\gamma_k<\alpha_k$, $0<\gamma_\ell<\alpha_\ell$ and 
$0\leqq\gamma_m<\alpha_m$.
In this case $f$ is one of the four polynomials 
$f_{12,2}^{(Z)}$, $f_{12,4}^{(Z)}-f_{12,3}^{(Z)}$, 
$f_{12,5}^{(Z)}$ and $f_{12,8}^{(Z)}$.
%\hfill$\Box$
\end{proof}
We label 
$\fGmt$ the action of $\mathbb{G}_m$ on 
 $B_{H_{12}}\simeq k[Z] / \ker \varphi_{12}$ so that for $g \in \fGmt$
$Z_a \mapsto g^a Z_a$ and each $a$ agrees with the weight.

 The relations listed in Proposition \ref{prop:Z12}
are given by the $2 \times 2$ minors
of
\begin{equation}
\displaystyle{
\left|\begin{matrix}
 Z_6^2  & Z_{14} & Z_{16} \\
 Z_{14} & Z_{16}  & Z_6^3\\
\end{matrix} \right|, \quad
\left|\begin{matrix}
 Z_6^2  & Z_{13} & Z_{14} & Z_{15}   \\
 Z_{13} & Z_{14} & Z_{15} & Z_{16}  \\
\end{matrix} \right|};
\label{eq:2x2minor12Z}
\end{equation}
although  $Z_{13} Z_{15} - Z_6^2 Z_{16}$ is not listed,
and
$\fZ{12,8}$ is not one of the minors they are compatible: for
example, the given minor follows by combining
 $\fZ{12,8}$ with $\fZ{12,1}$ and $\fZ{12,9}$.
We note that the first matrix is the same as 
(\ref{eq:2x2minor4Z}). Thus, we reprise the notation
used in Section \ref{sec:378}.

We construct a non-singular curve $X_{12}$ by giving an affine patch,
an ideal in the ring
$\CC[x, y_{13}, y_{14}, y_{15}, y_{16}]$.
For any complex numbers $b_6$ and $b_7$  distinct from the previous $b$'s,
we let
\begin{gather*}
\begin{split}
       \hk_{2}(x) &:=  (x-\muc_6)(x-\muc_7)
   = x^2 + \hltwo{1} x + \hltwo{0}, \quad \quad
       \hk_{5}(x) := \hk_{2}(x) k_{3}(x), \\
       k_{13}(x)& :=  k_3(x)k_2(x)^2 \hk_2(x)^3, \quad \quad\quad
       k_{14}(x)  :=  k_7(x)^2 = k_3(x)^2k_2(x)^4 , \quad \quad\quad\\
       k_{15}(x)& :=  \hk_5(x)^3, \quad \quad\quad
      \quad\quad k_{16}(x)  :=  k_8(x)^2 = k_3(x)^4k_2(x)^2 .  \\
\end{split}
\end{gather*}

Let the prime ideal $\cP$  in $\CC[x, y_{13}, y_{14}, y_{15}, y_{16}]$
be defined by
\begin{equation*}
    \cP := (f_{12, 1}, f_{12, 2}, f_{12, 3}, f_{12, 4},
f_{12, 5}, f_{12, 6}, f_{12, 7}, f_{12, 8}, f_{12, 9}),
\end{equation*}
where
\begin{equation*}
\begin{split}
f_{12,1} := &y_{13}^{2} - \hk_{2}(x) y_{14}, \quad
f_{12,2} := y_{13} y_{14} - k_{2}(x) y_{15}, \quad
f_{12,3} :=\hk_2(x) y_{14}^{2} - y_{13} y_{15} k_2(x) \\
f_{12,4} :=& y_{14}^{2} - k_{2}(x) y_{16} , \quad
f_{12,5} :=y_{13}y_{16} - y_{14} y_{15}, \quad
f_{12,6} := y_{15}^{2} - \hk_2(x)k_3(x), \\
f_{12,7} :=& y_{14}y_{16} - k_2(x)k_3(x), \quad
f_{12,8} := y_{15} y_{16} - k_3(x)y_{13}, \quad
f_{12,9} := y_{16}^{2} - k_3(x)y_{14},
\end{split}
\end{equation*}
which are the $2 \times 2$ minors of
\begin{equation}
\displaystyle{
\left|\begin{matrix}
 k_2(x)  & y_{14} & y_{16} \\
 y_{14} & y_{16}  & k_3(x)\\
\end{matrix} \right|, \quad
\left|\begin{matrix}
 \hk_2(x)  & y_{13} & y_{14} \hk_2(x) & y_{15}   \\
 y_{13} & y_{14} & y_{15} k_2(x)  & y_{16}  \\
\end{matrix} \right|};
\label{eq:2x2minor12}
\end{equation}
again, the minor $y_{13} y_{15} - \hk_2(x) y_{16}$ is not
in the list of $f_{i,j}$ and
$f_{12,8}$ is not a minor, but they are compatible--the minor
follows  by combining $f_{12,8}$ with $f_{12,1}$ and $f_{12,9}$.
Here the first matrix is the same as 
(\ref{eq:2x2minor4}) and the latter  is related to the
double covering of $X_4$. 
In other words, we essentially identify $y_a$ and $y_{2a}$ for
$a=7,8$. There is only one point at infinity in this
affine model because, as is clear from the equations, $x$ approaches
$\infty$ if and only if 
each $y_b$ $(b=13,14,15,16)$ approaches
$\infty$.
We define the $\fGmt$ action on $x$ and $y_a$ by
$g_m^{-6} x$ and $g_m^{-a} y_a$ $(a = 13, 14, 15, 16)$
near $\infty \in X_{12}$.

Corresponding to Proposition \ref{prop:Z12},
we will consider a commutative ring, 
$$
R_{12} =\CC[x, y_{13}, y_{14}, y_{15}, y_{16}]/ \cP.
$$

\begin{proposition}
There is a ring homomorphism $R_4 \to R_{12}$.
% or a natural projection $\Spec R_{12} \to \Spec R_{4}$.
\end{proposition}

Following Nagata's Jacobian criterion
in Th.30.10 of \cite[Theorem 30.10]{Mat}, we show that Spec$R_{12}$ is
non-singular.
\begin{proposition}
For every $(x, y_{13}, y_{14}, y_{15}, y_{16})$ which is a
zero of every $(f_{12,a})_{a = 1, \cdots, 9}$,
we have
$$
\rank \cU_{12} = 4, \quad
\cU_{12} := 
\left(
 \begin{pmatrix}
\frac{\partial}{\partial x} f_{12,a} &
\frac{\partial}{\partial y_{13}} f_{12,a} &
\frac{\partial}{\partial y_{14}} f_{12,a} &
\frac{\partial}{\partial y_{15}} f_{12,a} &
\frac{\partial}{\partial y_{16}} f_{12,a} \\
 \end{pmatrix}_{a = 1, \cdots, 9}
\right).
$$
\end{proposition}
\begin{proof}
$\cU_{12}$ is 
$$
\begin{pmatrix}
-\hk_2' y_{14} & 2 y_{13} &-\hk_{2}&      &     & \\
-k_{2}' y_{15} & y_{14}   & y_{13}  & - k_2 &\\
\hk_2'y_{14}^2 - y_{13} y_{15} k_2' &
-y_{15} k_2 & 2\hk_2 y_{14} & -y_{13}  k_2 & \\
- k_{2}' y_{16} & & 2y_{14} & & k_{2}  \\
    &y_{16} & - y_{15} & - y_{14} &y_{13} \\
- (\hk_2k_3)' & & & 2 y_{15} & \\
- (k_2 k_3)' & & y_{16} &  &y_{14} \\
- k_3'y_{13} & - k_3 & & y_{16} & y_{15} \\ 
-k_3' y_{14} & & - k_3 & & 2 y_{16} 
\end{pmatrix}.
$$
When $x = b_1$ as a zero of $k_3(x)$,
 every $y_{a}$ vanishes and thus
$$
\begin{pmatrix}
 &  &-\hk_{2}(b_1)& { }     &  { }   & \\
 &  & &    - k_2(b_1) &\\
 &  & &    &    & \\
 &  & &    &    & \\
- \hk_2(b_1) & & & 2 y_{15} & \\
- k_2(b_1) (b_1 - b_2)(b_1 - b_3) & &  &  & \\
 &  & &    &    & \\
 &  & &    &    & 
\end{pmatrix}
$$
is evidently a matrix of rank  4.
When $x = b_4$ as a zero of $k_2(x)$,
 $y_{13}$, $y_{14}$ and $y_{16}$ vanish and 
$\cU_{12}$ is equal to 
$$
\begin{pmatrix}
 &  &-\hk_{2}(b_4)&      &     & \\
-(b_4 - b_5) y_{15} &   &  &    &  &\\
y_{15} &    &  &    &  &\\
 & &  &  &  & \\
 & &  &  &  & \\
   &  & - y_{15} &  & \\
- (\hk_2k_3)'(b_1) & & & 2 y_{15} & \\
- (b_4-b_5) k_3(b_4) & &  &  & \\
 & - k_3(b_4) & &  & y_{15} \\ 
 & &  &  &  & \\
\end{pmatrix}
$$
whose rank is obviously $4$.
When $x = b_6$ as a zero of $\hk_2(x)$,
 $y_{13}$ and $y_{15}$ vanish and thus
the matrix is 
$$
\begin{pmatrix}
-(b_6 - b_7) y_{14} &  & &      &     & \\
 & y_{14}   &  & - k_2(b_6) &\\
(b_6 - b_7) y_{14}^2   &  &  &  & \\
- k_{2}(b_6)' y_{16} & & 2y_{14} & & k_{2}(b_6)  \\
    &y_{16} &  & y_{14} & \\
- (b_6 - b_7)k_3(b_6) & & &  & \\
- (k_2 k_3)'(b_6) & & y_{16} &  &y_{14} \\
 & - k_3 & & y_{16} & \\ 
-k_3'(b_6) y_{14} &  & -k_3(b_6) & & 2 y_{16}  
\end{pmatrix}.
$$
Since $\displaystyle{
\frac{y_{14}}{k_3(b_6)}=
\frac{y_{16}}{y_{14}}}$ and
$\displaystyle{
\frac{y_{16}}{y_{14}}=
\frac{k_3(b_6)}{y_{16}}=
\frac{y_{14}}{k_2(b_6)}}$,
the rank is $4$.
When $x$ differs from $b_1$, $b_2$, $\cdots$, and $b_7$, we compute
$$
\begin{pmatrix}
1 &  & &      &     & & & & \\
 & 1 & & -y_{15}/y_{16}  &-y_{14}/y_{16}  &  & & &\\
y_{14} & -y_{13} & 1 &      &     &  & & &\\
  &  & & 1     &     &  & & &\\
  &  & &      & 1    &  & & &\\
-k_3/y_{14} &  & &      &  y_{15}/y_{14}   & 1 & -\hk_2/y_{13} & &\\
  & -k_3/y_{15} & &      &     &  &1 & k_2 / y_{13} &\\
  &  & &      &     &  & & 1 &\\
 &  & &      & -y_{16}/y_{13}  &  & & -y_{14}/y_{13} & 1\\
\end{pmatrix}
\cU_{12},
$$
which is equal to a matrix of rank $4$, 
$$
\begin{pmatrix}
-\hk_2' y_{14} & 2 y_{13} &-\hk_{2}&      &     & \\
\\
\\
- k_{2}' y_{16} & & 2y_{14} & & k_{2}  \\
    &y_{16} & - y_{15} & - y_{14} &y_{13} \\
\\
\\
- k_3'y_{13} & - k_3 & & y_{16} & y_{15} \\ 
\\
\end{pmatrix}.
$$
\end{proof}

There is a  Riemann surface, which we denote $X_{12}$, obtained from the
affine smooth curve given by Spec($R_{12}$) by adding one point $\infty$
(see \ref{sec:AppendixB}).

\bigskip
In the local ring of the place   $\infty$, we express
%$S^{-1}_\infty R_{12}$; in $S^{-1}_\infty R_{12}$, 
$x$ and $y$'s as

$$
%	y_a = \frac{1}{t^{12+a}} , \quad (a = 1, 2, 3, 4), \quad
	y_a = \frac{1}{t^{12+a}}(1+\cdots) , \quad (a = 1, 2, 3, 4), \quad
        x = \frac{1}{t^6},
$$

using a local parameter $t$ at $\infty$.

For later use, we introduce the following ring,
$$
\hR_{12}:=\CC[x, w_3, w_2, \hw_2]/(w_3^6 - k_3(x), w_2^6 - k_2(x),
\hw_2^6 -\hk_2(x)),
$$
and then we have a natural ring homomorphism
$i_{12}:R_{12}\to \hR_{12}$.
Since in $\hR_{12}$, we have
\begin{equation}
w_3^6 = k_3(x), \quad
w_2^6 = k_2(x), \quad
\hw_2^6 =\hk_2(x),
\label{eq:w6}
\end{equation}
in $i_{12}(R_{12})$, the following holds
\begin{equation}
 y_{13} = w_3w_2^2 \hw_2^3, \quad \quad
 y_{14} = w_3^2w_2^4 , \quad \quad
 y_{15} = w_3^3 \hw_2^3, \quad \quad
 y_{16} = w_3^4w_2^2.
\label{eq:yw6}
\end{equation}
This gives a  projection
\begin{equation}
\Spec \hR_{12} \to \Spec R_{12}.
\label{eq:projhatR}
\end{equation}
\bigskip
%$$

The curve $X_{12}$ has a cyclic action of a primitive
sixth root of unity $\zeta_6$:
$$
   \hzeta_6(x,y_{13}, y_{14}, y_{15},y_{16})
     = (x, \zeta_6 y_{13}, \zeta_6^2 y_{14},
      \zeta_6^3 y_{15}, \zeta_6^4 y_{16}).
$$
Using the above expression of $w$'s, it is obvious
that the prime ideal $\cP$ is stable for the cyclic action. 

\bigskip

\subsection{The Weierstrass gaps and holomorphic one forms}

The Weierstrass gap sequence of $R_{12}$ at $\infty$
is given by the following table:
in particular, the semigroup $H(X_{12},\infty)$ equals $H_{12}$.
\begin{gather*}
{\tiny{
\centerline{
\vbox{
	\baselineskip =10pt
	\tabskip = 1em
	\halign{&\hfil#\hfil \cr
        \multispan7 \hfil Table 2: The Weierstrass gaps and non-gaps at $\infty$\hfil \cr
	\noalign{\smallskip}
	\noalign{\hrule height0.8pt}
	\noalign{\smallskip}
%\strut\vrule& 
0 &1 & 2 & 3 & 4 & 5 & 6 & 7 & 8 & 9 & 10 & 11 & 12 & 13 & 14 & 15 & 16&
17& 18& 19 & 20 & 21 & 22 & 23 \cr
\noalign{\smallskip}
\noalign{\hrule height0.3pt}
\noalign{\smallskip}
% \strut\vrule & 
 1& - & - & - & - & - & $x$& - & - & - & - & -& $x^2$& $y_{13}$
& $y_{14}$ & $y_{15}$ & $y_{16}$ & -& $x^3$ & $xy_{13}$& $x y_{14}$
& $x y_{15}$ & $x y_{16}$& -  \cr 
% \strut\vrule & 
%  - & $x^4dx/y_{15}y_{16}$ & $xdx/y_{13}$  & $xdx/y_{14}$ &
% $xdx/y_{15}$ & $xdx/y_{16}$ 
% & - & $dx/y_{13}$ & $dx/y_{14}$ & $dx/y_{15}$ 
%& $dx/y_{16}$& $x^2 dx/y_{13} y_{16}$
%& -& -& -& -& -& $x dx/y_{13} y_{16}$
% & -& - &- & - & -& $dx / y_{13} y_{16}$ \cr 
% \strut\vrule & 
\noalign{\smallskip}
	\noalign{\hrule height0.8pt}
}
}
}
}}
\end{gather*}

We introduce a notation for
the monomials in $R_{12}$ whose orders of pole at
$\infty$ correspond to the non-gaps:
$$
	\phit{0} = 1, \quad
	\phit{1} = x, \quad
	\phit{2} = x^2, \quad
	\phit{a+2} = y_a, \quad
	\phit{7} = x^3, \quad
	\phit{a+7} = x y_a, \quad
$$
\begin{gather*}
    \phit{6i - 12} = x^i , \quad 
    \phit{6i - 12 + a} = x^{i-1} y_a , \quad
    \phit{6i - 7} = x^{i - 3} y_{13} y_{16}, \quad
(a=1,2,3,4, i=4,5,6,\cdots).
\end{gather*}
We define the weight $\Nt(n)$ by
$$
	\Nt(n)=- \wt(\phit{n}),
$$
where $\wt()$ is the negative of the order of pole  at $\infty$,
 consistent with the $\fGmt$ action.
% $\infty$ is of course the only singular points of the monomials.

The related Young diagram $\mathcal Y_{12}$ of $M_{12}$ is given by
$$
\Lambda_i = \Nt(g) - \Nt(i-1) -13 + i,
$$
which is equal to $\alpha_{12-i}(L(H_{12}))$ 
as defined in (\ref{eq:alphaL}).
\begin{center}
\begin{equation*}
\yng(12,7,2,2,2,2,2,1,1,1,1,1).
\end{equation*}
\end{center}
\def\theequation{\thesection.\arabic{equation}}

When $\muc$'s are distinct,
a basis of the holomorphic one-forms is given by  $\nuIt{i}$,
 $i=1, 2, \ldots, 12$, 
$$
\nuIt{[13-i]}:=	\nuIt{i} := \frac{\phit{i-1}dx}{6 y_{13} y_{16}}
	     = \frac{\phit{i-1}dx}{6 y_{14} y_{15}},
$$
or
\begin{equation}
\begin{matrix}
\nuIt{1}=\displaystyle{\frac{dx}{6 y_{13}y_{16}}},&
\nuIt{2}=\displaystyle{\frac{xdx}{6 y_{13}y_{16}}},&
\nuIt{3}=\displaystyle{\frac{x^2dx}{6 y_{13}y_{16}}},&
\nuIt{4}=\displaystyle{\frac{dx}{6 y_{16}}},\\
\nuIt{5}=\displaystyle{\frac{dx}{6 y_{15}}},&
\nuIt{6}=\displaystyle{\frac{dx}{6 y_{14}}},&
\nuIt{7}=\displaystyle{\frac{dx}{6 y_{13}y_{16}}},&
\nuIt{8}=\displaystyle{\frac{x^3dx}{6 y_{13}y_{16}}},\\
\nuIt{9}=\displaystyle{\frac{xdx}{6 y_{16}}},&
\nuIt{10}=\displaystyle{\frac{xdx}{6 y_{15}}},&
\nuIt{11}=\displaystyle{\frac{xdx}{6 y_{14}}},&
\nuIt{12}=\displaystyle{\frac{xdx}{6 y_{13}}},\\
\end{matrix}
\label{eq:nuI12}
\end{equation}
%%%%%%%%%%%%%%%%%SM 06302013

%and its Abelian integral for a point $P \in X_{12}$ is defined:
%$$
%    \hu(P) = \int^P_{\infty} \nuIt{} \in \CC^{12},
%$$
%for $k$ points $P_1, P_2, \cdots, P_k \in X_{12}$, the Abelian map
%$\hu(P_1, \cdots, P_k)$ is defined by $\sum_{i=1}^k \hu(P_i)$.
%%%%%%%%%%%%%%%%%
and the corresponding Abel map for a point $P \in X_{12}$ is defined:
$$
    \hu(P) \equiv \hu_o(P) := \int^P_{\infty} \nuIt{} \in \CC^{12},
$$
for $k$ points $P_1, P_2, \ldots, P_k \in X_{12}$, the Abel map
$\hu(P_1, \ldots, P_k)\equiv\hu_o(P_1, \ldots, P_k)$
is defined by $\sum_{i=1}^k \hu(P_i)$.

%%%%%%%%%%%%%%%%%SM 06302013
We also choose a basis
$   \alphat{i}, \betat{j}$  $ (1\leqq i, j\leqq 12)$
of $H_1(X_{12},\ZZ)$ such that the
intersection numbers are
$\alphat{i}\cdot\alphat{j}=\betat{i}\cdot\betat{j}= 0$ and 
$\alphat{i}\cdot\betat{j}=\delta_{ij}$. Let
 the period matrices be
 denoted by
\begin{equation}
\begin{split}
\left[\,\omegatp{}  \ \omegatpp{} \right]&:= 
\frac{1}{2}\left[\int_{\alphat{i}}\nuIt{j} \ \ \int_{\betat{i}}\nuIt{j}
\right]_{i,j=1,2, \cdots, 12}. \label{eq:12:2.4}
\end{split}
\end{equation} 
Then we define the Jacobian $\cJ_{12}$ by
$$
\kappa : \CC^{12} \to \cJ_{12} =\CC^{12}/\Lambda_{12},
$$
where $\Lambda_{12}$ is generated by $\omegatp{}$ and $\omegatpp{}$.

The Abel map of
    for $(P_1, \cdots, P_k)\in S^k X_{12}$
to $\CC^{12}$ is also expressed by
$$
\ta_{23-\Nt(12-i)}^{[k]} :=u_{i}^{[k]} :=\hu(P_1, \ldots, P_k) \in \CC^{12}.
%:= \sum_{j=1}^{k}\int^{P_j}_{\infty} d u_i \in \CC^{12},
$$
Further we define the subvariety $\WWt{k}$ by
$$
   \WWt{k} := \hu(S^k X_{12})/ \Lambda_{12}.
$$
By letting
\begin{equation}
\ta_{23-\Nt(12-i)} :=u_{i} 
:=\ta_{23-\Nt(12-i)}^{[12]} \equiv u_{i}^{[12]},
\label{eq:ta12}
\end{equation}
the weights are given by 
$$
\{\wt(\ta_i)\} = \{1, 2, 3, 4, 5, 7, 8, 9, 10, 11, 17, 23\}.
$$

\bigskip

%%%%%%%%%%%%%%%%%%%%%%%%%%%%%%%%%%%%%%%%%%%%%%%%%%%%%%%%%%%%%%%%%%%%%%%%%%%%%%%
Here we use the convention that for $P_a \in X_{12}$,
$P_a $ is expressed by $(x_a, y_{13,a}, y_{14,a}, y_{15,a}, y_{16,a})$
or 
$(x_{P_a}, y_{13,{P_a}}, y_{14,{P_a}}, y_{15,{P_a}}, y_{16,{P_a}})$
and $w_{b,a}$ is also $w_b$ associated with $P_a$.

By letting $u:=\hu(P_1, \cdots, P_{12})$, we have
\begin{equation*}
\begin{pmatrix}
\partial/\partial{u_1}\\
\partial/\partial{u_2}\\
\vdots\\
\partial/\partial{u_{11}}\\
\partial/\partial{u_{12}}
\end{pmatrix}
=
\Psi_{12}^{(12)-1}
\begin{pmatrix}
6y_{13,1}y_{14,1} \partial/\partial{x_1}\\
6y_{13,2}y_{14,2} \partial/\partial{x_2}\\
\vdots\\
6y_{13,{11}}y_{14,{11}} \partial/\partial{x_{11}}\\
6y_{13,{12}}y_{14,{12}} \partial/\partial{x_{12}}
\end{pmatrix},
\end{equation*}
where 
\begin{equation}
\Psi_{12}^{(12)}:=
\begin{pmatrix}
\phit{0}(P_1) & \phit{1}(P_1) & \phit{2}(P_1) & \cdots &\phit{11}(P_1) \\
\phit{0}(P_2) & \phit{1}(P_2) & \phit{2}(P_2) & \cdots &\phit{11}(P_2) \\
\vdots & \vdots   & \vdots  & \ddots & \vdots \\
\phit{0}(P_{11}) & \phit{1}(P_{11}) & \phit{2}(P_{11}) &\cdots & 
\phit{11}(P_{11})  \\
\phit{0}(P_{12}) & \phit{1}(P_{12}) & \phit{2}(P_{12}) &\cdots &
 \phit{{11}}(P_{12}) \\
\end{pmatrix}.
\label{eq:partialInPsi12}
\end{equation}
In particular, for any 12-tuple of numbers $(\varepsilon_i)_{i=1,\ldots,12}$,
we have
$$
\sum_{i=1} \epsilon_i\frac{\partial}{\partial u_i}
=
|\Psi_{12}^{(12)-1}|
\left|
\begin{matrix}
\phit{0}(P_1) & \phit{1}(P_1) & \phit{2}(P_1) & \cdots &\phit{11}(P_1) &
3y_{13,1}y_{16,1} \partial/\partial{x_1}\\
\phit{0}(P_2) & \phit{1}(P_2) & \phit{2}(P_2) & \cdots &\phit{11}(P_2) &
3y_{13,2}y_{16,2} \partial/\partial{x_2}\\
\vdots & \vdots   & \vdots  & \ddots & \vdots \\
\phit{0}(P_{11}) & \phit{1}(P_{11}) & \phit{2}(P_{11}) &\cdots & 
\phit{11}(P_{11})  &
3y_{13,3}y_{16,3} \partial/\partial{x_3}\\
\phit{0}(P_{12}) & \phit{1}(P_{12}) & \phit{2}(P_{12}) &\cdots &
 \phit{{11}}(P_{12}) &
3y_{13,4}y_{16,4} \partial/\partial{x_4}\\
\epsilon_1 & \epsilon_2 & \epsilon_3 & \cdots & \epsilon_{12} &
\end{matrix}
\right|.
$$
We express the change of variables:
\begin{equation}
          \sum_{i, j=1}^{12} \phit{i-1}(P_1) \phit{j-1}(P_2)
	\frac{\partial^2 }
{\partial \hu_i(P_1)\partial \hu_j(P_2)}
= 9 y_{13,1} y_{16,1} y_{13,2} y_{16,2}
	\frac{\partial^2 }{\partial x_1\partial x_2}.
\label{eq:partial_H12}
\end{equation}
%%%%%%%%%%%%%%%%%%%%%%%%%%%%%%%%%%%%%%%%%%%%%%%%%%%%%%%%%%%%%%%%%%%%%%%%%%%%%%%

\subsection{Differentials of the second and the third kinds}
\label{differential forms 12}

We also give an algebraic representation of
the fundamental normalized differential of the 
second kind  $\Omegat(P_1, P_2)$ on $X_{12}\times X_{12}$,
whose defining properties are the same as those given for
$\Omegaf(P_1, P_2)$ on $X_{4}\times X_{4}$. 
%Let $\Omegat(P_1, P_2)$ be symmetric, 
%\begin{equation}
%\Omegat(P_1, P_2)=\Omegat(P_2, P_1), 
%\label{eq12.1.6}  
%\end{equation}
%and have its only pole (of second order) along the diagonal of 
%$X_{12}\times X_{12}$, 
%so that in the vicinity of each point $(P_1,P_2)$,
%it is expanded in power series as 
%\begin{equation}
%\Omegat(P_1, P_2)=\Big(\frac{1}{(t_{P_1} -t_{P_2} ')^2  } +d_\ge(1)\Big)
%   d t_{P_1} \otimes d t_{P_2}
%\ \ (\hbox{\rm as}\  P_1\rightarrow P_2), 
%\label{expansion12}
%\end{equation}
%where $t_P$ is a local coordinate at the point $P \in X_{12}$.

\begin{proposition} \label{prop:Sigmat}
Let $\Sigmat\big(P_1, P_2\big)$ be the following form,
\begin{gather}
\begin{split}
\Sigmat(P_1, P_2) &:=
\frac{
  y_{13,1} y_{16,1}
+ y_{13,2} y_{16,1}
+ y_{13,1} y_{16,2}
+ y_{14,2} y_{15,1}
+ y_{14,1} y_{15,2}
+ y_{14,2} y_{15,2}
}{6 (x_1-x_2) y_{13,1} y_{16,1}} d x_1.
\label{eq:Sigmat}
\end{split}
\end{gather}
Then $\Sigmat(P, Q)$ has the following properties:
\begin{enumerate}
\item $\Sigmat(P, Q)$ as a function of $P$ is singular at 
$Q=(x_Q, y_{13,Q}, y_{14,Q}, y_{15,Q}, y_{16,Q})$ and
 $\infty$ and is not singular at 
$\hzeta_6(Q)= 
(x_Q, \zeta_6 y_{13,Q}, \zeta_6^2 y_{14,Q}, 
\zeta_6^3 y_{15,Q}, \zeta_6^4 y_{16,Q})$.

\item $\Sigmat(P, Q)$ as a function of $Q$ is singular at $P$ and
 $\infty$.
\end{enumerate}
\end{proposition}

\begin{proof}
The result can be checked more easily  in terms of the functions
$w$'s in the  ring $\hat{R}_{12}$, cf.  (\ref{eq:w6});
we express $\Sigmat(P_1, P_2)$ in term of $w$'s,
\begin{gather}
\begin{split}
\Sigmat(P_1, P_2) 
&=\frac{
  w_{3,1} w_{2,1}^2 \hw_{2,1}^3 w_{3,1}^4 w_{2,1}^3
+ w_{3,2} w_{2,2}^2 \hw_{2,2}^3 w_{3,1}^4 w_{2,1}^3
+ w_{3,1} w_{2,1}^2 \hw_{2,1}^3 w_{3,2}^4 w_{2,2}^3
}{6 (x_1-x_2) y_{13,1} y_{16,1}} d x_1\\
&+\frac{
 w_{3,1}^2 w_{2,1}^4 w_{3,1}^4  \hw_{2,1}^3
+ w_{3,1}^2 w_{2,1}^4 w_{3,2}^4 \hw_{2,2}^3
+ w_{3,2}^2 w_{2,2}^4 w_{3,1}^4 \hw_{2,1}^3
}{6 (x_1-x_2) y_{13,1} y_{16,1}} d x_1.\\
\label{eq:Sigmatw}
\end{split}
\end{gather}
When the $x$-coordinates of $P_1$ and $P_2$ differ, it obvious that
it does not diverge. When the $x$-coordinates coincide but
$P_1$ and $P_2$ differ, some of $w_{3,1}$, $w_{2,1}$, and $\hw_{3,1}$
differ from  $w_{3,2}$, $w_{2,2}$, and $\hw_{3,2}$. Then the numerator
vanishes to order  higher than two when expanded in the local parameter.
$\Sigmat(P_1, P_2)$ diverges only at infinity.
\end{proof}

The following holds for smooth $X_{12}$;
\begin{proposition} \label{prop:dSigmat}
There exist  differentials $\nuIIt{j}=\nuIIt{j}(x,y)$ $(j=1, 2, \cdots, 12)$ 
of the second kind such that
they have
 their only pole at $\infty$ and satisfy the relation,
\begin{equation}
\begin{split}
  &d_{Q} \Sigmat\big(P, Q\big) - 
  d_{P} \Sigmat\big(Q, P\big)\\
   &\quad\quad=
     \sum_{i = 1}^{12} \Bigr(
         \nuIt{i}(Q)\otimes \nuIIt{i}(P)
        - \nuIt{i}(P)\otimes \nuIIt{i}(Q)
     \Bigr)
   \label{eq12:3.4}, 
\end{split}
\end{equation} 
where
\begin{equation}
 d_{Q} \Sigmat\big(P, Q\big)
   :=d x_P \otimes d x_Q\frac{\partial }{ \partial x_Q}
   \Sigmat .
\end{equation} 
The set of differentials $\{\nuIIt{1}$, $\nuIIt{2}$, 
$\cdots$, $\nuIIt{12}\}$
 is
determined modulo the $\mathbb{C}$-linear space spanned by
$\langle\nuIt{j}\rangle_{j=1, \ldots, 12}$.
\end{proposition}

\begin{proposition}
The differentials of the second kind are given by
\begin{gather*}
\begin{split}
\nuIIt{1}&= -
\Bigr(23x^5+(19\ltwo{1}+18\lthree{1}+20\hltwo{1})x^4
+(14\lthree{1}\ltwo{1}+15\lthree{1}\hltwo{1}
  +16\ltwo{1}\hltwo{1}\\
&\qquad
+13\lthree{2}+15\ltwo{2}+17\hltwo{2})x^3\\
&\qquad
+(10\lthree{1}\ltwo{2}+9\lthree{2}\ltwo{1}+10\lthree{2}\hltwo{1}
 +12\ltwo{2}\hltwo{1}+12\lthree{1}\hltwo{2}+13\ltwo{1}\hltwo{2}\\
&\qquad
 +8\lthree{3}+11\lthree{1}\ltwo{1}\hltwo{1})x^2\\
&\qquad
+(4\lthree{3}\ltwo{1}+5\lthree{2}\ltwo{2}+5\lthree{3}\hltwo{1}
  +7\lthree{2}\hltwo{2}+9\ltwo{2}\hltwo{2}
  +6\lthree{2}\ltwo{1}\hltwo{1}\\
&\qquad
+8\lthree{1}\ltwo{1}\hltwo{2}
  +7\lthree{1}\ltwo{2}\hltwo{1})x\\
&\qquad
+3\lthree{2}\ltwo{1}\hltwo{2}+4\lthree{1}\ltwo{2}\hltwo{2}
 +2\lthree{2}\ltwo{2}\hltwo{1}+\lthree{3}\ltwo{1}\hltwo{1}
+2\lthree{3}\hltwo{2} \Bigr)d x/6 y_{13},\\
\end{split}
\end{gather*}
\begin{gather*}
\begin{split}
\nuIIt{2}&=-\Bigr(17x^4+(13\ltwo{1}+12\lthree{1}+14\hltwo{1})x^3\\
&\qquad
 +(10\ltwo{1}\hltwo{1}+9\lthree{1}\hltwo{1}+8\lthree{1}\ltwo{1}
+11\hltwo{2}+9\ltwo{2}+7\lthree{2})x^2\\
&\qquad
 +(7\ltwo{1}\hltwo{2}+6\lthree{1}\hltwo{2}
  +4\lthree{1}\ltwo{2}+6\ltwo{2}\hltwo{1}+4\lthree{2}\hltwo{1}
  +3\lthree{2}\ltwo{1}\\
&\qquad\quad
+2\lthree{3}
+5\lthree{1}\ltwo{1}\hltwo{1})x\\
&\qquad
+\lthree{1}\ltwo{2}\hltwo{1}+2\lthree{1}\ltwo{1}\hltwo{2}
  +3\ltwo{2}\hltwo{2}+\lthree{2}\hltwo{2}\Bigr)d x/6 y_{13},\\
\end{split}
\end{gather*}
\begin{gather*}
\begin{split}
\nuIIt{3}&=\Bigr(11x^3+(6\lthree{1}+7\ltwo{1}+8\hltwo{1})x^2
 +(4\ltwo{1}\hltwo{1}+3\lthree{1}\hltwo{1} +2\lthree{1}\ltwo{1}\\
&\qquad
 +\lthree{2}+3\ltwo{2}+5\hltwo{2})x
 -3\lthree{2}\ltwo{1}+\ltwo{1}\hltwo{2}\Bigr)d x/6 y_{13},\\
\nuIIt{4}&=-\Bigr((8x^3+(4\ltwo{1}+6\lthree{1})x^2
+(2\lthree{1}\ltwo{1}4\lthree{2})x+2\lthree{3})\Bigr)
d x/6 y_{14},\\
\end{split}
\end{gather*}
\begin{gather*}
\begin{split}
\nuIIt{5}&=-\Bigr(9x^3+(6\hltwo{1}+6\lthree{1})x^2
+ (3\lthree{1}\hltwo{1}+3\hltwo{2}+3\lthree{2})x  \Bigr)d x/6 y_{15},\\
\nuIIt{6}&=-\Bigr(10x^3+(8\ltwo{1}+6\lthree{1})x^2+(4\lthree{1}\ltwo{1}
+2\lthree{2}+6\ltwo{2})x+2\lthree{1}\ltwo{2}\Bigr)
d x/6 y_{16},\\
\end{split}
\end{gather*}
\begin{gather*}
\begin{split}
\nuIIt{7}&=-\Bigr(7x^5+(6\lthree{1}+5\ltwo{1}+4\hltwo{1})x^4\\
&+(3\ltwo{2}+3\lthree{1}\hltwo{1}
-\hltwo{2}+5\lthree{2}+2\ltwo{1}\hltwo{1}+4\lthree{1}\ltwo{1})x^3\\
&+(\lthree{1}\ltwo{1}\hltwo{1}+2\lthree{2}\hltwo{1}
 +3\lthree{2}\ltwo{1}+2\lthree{1}\ltwo{2}
  +4\lthree{3})x^2\\
&+(-2\lthree{3}\ltwo{1}-\lthree{2}\ltwo{2} -\lthree{3}\hltwo{1})
\Bigr)d x/6 y_{13}y_{16},\\
\nuIIt{8}&=-\Bigr(5x^2+(2\hltwo{1}+\ltwo{1})x
  \Bigr)/6 y_{13},\\
\nuIIt{9}&=-2x^2d x/6 y_{14},\qquad
\nuIIt{10}= -3x^2 d x/6 y_{15},\qquad
\nuIIt{11}=-(2\ltwo{1}x+4x^2)d x/6 y_{16},\quad\\
\nuIIt{12}&=-x^4/6 y_{13}y_{16}.\\
\end{split}
\end{gather*}
\end{proposition}

\begin{proof}
By letting the numerator in (\ref{eq:Sigmat}) and (\ref{eq:Sigmatw}) be
denoted by
\begin{gather*}
\begin{split}
A(P_1,P_2)&:=
  y_{13,1} y_{16,1}
+ y_{13,2} y_{16,1}
+ y_{13,1} y_{16,2}
+ y_{14,2} y_{15,1}
+ y_{14,1} y_{15,2}
+ y_{14,2} y_{15,2}\\
&=
  w_{3,1} w_{2,1}^2 \hw_{2,1}^3 w_{3,1}^4 w_{2,1}^2
+ w_{3,2} w_{2,2}^2 \hw_{2,2}^3 w_{3,1}^4 w_{2,1}^2
+ w_{3,1} w_{2,1}^2 \hw_{2,1}^3 w_{3,2}^4 w_{2,2}^2\\
& \quad
+ w_{3,1} w_{2,1}^4 w_{3,1}^4 \hw_{2,1}^3
+ w_{3,1} w_{2,1}^4 w_{3,2}^4 \hw_{2,2}^3
+ w_{3,2} w_{2,2}^4 w_{3,1}^4 \hw_{2,1}^3\\
&=
  w_{3,1}^5 w_{2,1}^4 \hw_{2,1}^3 
+ w_{3,2}   w_{2,2}^2 \hw_{2,2}^3 w_{3,1}^4 w_{2,1}^2
+ w_{3,1}   w_{2,1}^2 \hw_{2,1}^3 w_{3,2}^4 w_{2,2}^2\\
& \quad
+ w_{3,2}^2 w_{2,2}^3 w_{3,1}^3 \hw_{2,1}^3
+ w_{3,1}^2 w_{2,1}^3 w_{3,2}^3 \hw_{2,2}^3
+ w_{3,2}^5 w_{2,2}^4 \hw_{2,1}^3,
\end{split}
\end{gather*}
we have the equality,
\begin{gather}
\begin{split}
\frac{\partial\Sigmat(P_1, P_2)}{\partial x_2}
&-\frac{\partial \Sigmat(P_2, P_1)}{\partial x_1}\\
&=
\frac{
(y_{13,2}y_{16,2} A(P_1, P_2)- y_{13,1}y_{16,1} A(P_2, P_1))/(x_1-x_2)
}{(x_1-x_2) 
y_{13,1} y_{16,1} y_{13,2} y_{16,2}} \\
&+\frac{
y_{13,1}y_{16,1} 
\frac{\partial A(P_1, P_2)}{\partial x_2}
+ y_{13,2}y_{16,2} 
\frac{\partial A(P_2, P_1)}{\partial x_1}
}{(x_1-x_2) y_{13,1} y_{16,1} y_{13,2} y_{16,2}}.
\label{eq12:AA}
\end{split}
\end{gather}
We must evaluate the numerator  
 of (\ref{eq12:AA}); we denote it by  $B(P_1, P_2)$.
The former terms are written as
\begin{gather*}
\begin{split}
&y_{13,2}y_{16,2} A(P_1, P_2)- y_{13,1}y_{16,1} A(P_2, P_1)\\
&=
 w_{3,2}^6 w_{2,2}^6  \hw_{2,2}^6 w_{3,1}^4 w_{2,1}^2
+ w_{3,1}   w_{2,1}^2 \hw_{2,1}^3 w_{3,2}^9 w_{2,2}^6 {w_{2,2}'}^3 
+ w_{3,2}^7 w_{2,2}^8 \hw_{2,2}^3 w_{3,1}^4 \hw_{2,1}^3 \\
& + w_{3,1}   w_{2,1}^4 w_{3,2}^8 w_{2,2}^4 \hw_{2,2}^6  
+ w_{3,2}^{10} w_{2,2}^8 \hw_{2,1}^6 
- (P_1 \leftrightarrow P_2)\\
%&=
% k_{3,2} k_{2,2} \hk_{2,2} y_{16,1}
%+ y_{13,1} k_{3,2} k_{2,2}  y_{15,2}
%+ k_{3,2} k_{2,2} y_{13,2} y_{15,1} 
%+ y_{14,1}   y_{14,2} k_{3,2} \hk_{2,2}  
%+ k_{3,2} k_{2,2}\hk_{2,1} y_{16,2} 
%- (P_1 \leftrightarrow P_2)\\
&=
y_{16,1} k_{3,2} k_{2,2} \hk_{2,2} 
+ y_{13,1} y_{15,2} k_{3,2} k_{2,2} 
+ y_{13,2} y_{15,1} k_{3,2} k_{2,2} 
+ y_{14,1}   y_{14,2} k_{3,2} \hk_{2,2}\\  
&+ y_{16,2} k_{3,2} k_{2,2}\hk_{2,1} 
- (P_1 \leftrightarrow P_2) \\
&=
y_{16,1} (k_{3,2} k_{2,2} \hk_{2,2} - k_{3,1} k_{2,1} \hk_{2,1})
+y_{16,2} (k_{3,2} k_{2,2} \hk_{2,2} - k_{3,1} k_{2,1} \hk_{2,1})\\
&+ (y_{13,1} y_{15,2} + y_{13,2} y_{15,1}) (
k_{3,2} k_{2,2} -k_{3,1} k_{2,1} )
+ y_{14,1}   y_{14,2} ( k_{3,2} \hk_{2,2}  -k_{3,1} \hk_{2,1} ). 
\end{split}
\end{gather*}
Here $k_{a, b} := k_a(x_b)$, and $\hk_{2, b} := \hk_2(x_b)$
for $a = 2,3$ and $b=1,2$ and $k' := d k / d x$.
We reinforce that these are expressions in the affine coordinates $x$'s
and $y$'s, in other words the $w$'s in these expressions,
which are algebraic over the ring $R_{12}$,
appear only within algebraic combinations which belong to $R_{12}$.

The differential terms of (\ref{eq12:AA}) are given as follows:
\begin{gather*}
\begin{split}
&6 y_{13,2}y_{16,2} 
\frac{\partial A(P_2, P_1)}{\partial x_2}
= w_{3,2}^5 w_{2,2}^4 {w_{2,2}'}^3 \Bigr(
y_{16,1} \Bigr(
\frac{1}{w_{3,2}^5} k_{3,2}' w_{2,2}^2 \hw_{2,2}^3
+w_{3,2} \frac{2}{w_{2,2}^4} k_{2,2}' \hw_{2,2}^3
+w_{3,2} w_{2,2}^2 \frac{3}{\hw_{2,2}^3} \hk_{2,2}'\Bigr)\\
&+y_{13,1} \Bigr(
\frac{4}{w_{3,2}^2} k_{3,2}' w_{2,2}^2
+w_{3,2}^4 \frac{2}{w_{2,2}^4} k_{2,2}' \Bigr)
+y_{15,1} \Bigr(
\frac{2}{w_{3,2}^4} k_{3,2}' w_{2,2}^4
+w_{3,2}^2 \frac{4}{w_{2,2}^2} k_{2,2}'\Bigr)\\
&+y_{14,1} \Bigr(
\frac{3}{w_{3,2}^3} k_{3,2}' \hw_{2,2}^3
+w_{3,2}^3 \frac{3}{\hw_{2,2}^3} \hk_{2,2}'\Bigr)
+\Bigr(
\frac{5}{w_{3,2}} k_{3,2}' w_{2,2}^4 \hw_{2,2}^3
+w_{3,2}^5 \frac{4}{w_{2,2}^2} k_{2,2}' \hw_{2,2}^3
+w_{3,2}^5 w_{2,2}^4 \frac{3}{\hw_{2,2}^3} \hk_{2,2}'\Bigr)\Bigr)\\
\end{split}
\end{gather*}
\begin{gather*}
\begin{split}
& =
y_{16,1} \Bigr(
k_{3,2}'k_{2,2}\hk_{2,2} +2k_{3,2}k_{2,2}'\hk_{2,2} +3k_{3,2}k_{2,2}\hk_{2,2}'
\Bigr) \\
&+
y_{13,1} y_{15,2} \Bigr(
4k_{3,2}'k_{2,2} +2k_{3,2}k_{2,2}'\Bigr)
+ y_{15,1} y_{13,2} \Bigr(
2k_{3,2}'k_{2,2} +4k_{3,2}k_{2,2}'\Bigr)
+ y_{14,1} y_{14,2} \Bigr(
3k_{3,2}'\hk_{2,2} +3k_{3,2}\hk_{2,2}'\Bigr)\\
&+y_{16,2} \Bigr(
5k_{3,2}'k_{2,2}\hk_{2,2} +4k_{3,2}k_{2,2}'\hk_{2,2} +3k_{3,2}k_{2,2}\hk_{2,2}'
\Bigr). \\
\end{split}
\end{gather*}

\begin{gather*}
\begin{split}
&6 y_{13,2}y_{16,2} 
\frac{\partial A(P_2, P_1)}{\partial x_2}
+6 y_{13,1}y_{16,1} 
\frac{\partial A(P_1, P_2)}{\partial x_1}\\
& =
y_{16,1} \Bigr(
k_{3,2}'k_{2,2}\hk_{2,2} +2k_{3,2}k_{2,2}'\hk_{2,2} +3k_{3,2}k_{2,2}\hk_{2,2}'
+
5k_{3,1}'k_{2,1}\hk_{2,1} +4k_{3,1}k_{2,1}'\hk_{2,1} +3k_{3,1}k_{2,1}\hk_{2,1}'
\Bigr) \\
&+
y_{13,1} y_{15,2} \Bigr(
4k_{3,2}'k_{2,2} +2k_{3,2}k_{2,2}'
+2k_{3,1}'k_{2,1} +4k_{3,1}k_{2,1}'\Bigr)\\
&+ y_{15,1} y_{13,2} \Bigr(
2k_{3,2}'k_{2,2} +4k_{3,2}k_{2,2}'
+4k_{3,1}'k_{2,1} +2k_{3,1}k_{2,1}' \Bigr)\\
&+ y_{14,1} y_{14,2} \Bigr(
3k_{3,2}'\hk_{2,2} +3k_{3,2}\hk_{2,2}'
+3k_{3,1}'\hk_{2,1} +3k_{3,1}\hk_{2,1}' \Bigr)\\
&+y_{16,2} \Bigr(
5k_{3,2}'k_{2,2}\hk_{2,2} +4k_{3,2}k_{2,2}'\hk_{2,2} +3k_{3,2}k_{2,2}\hk_{2,2}'
+
k_{3,1}'k_{2,1}\hk_{2,1} +2k_{3,1}k_{2,1}'\hk_{2,1} +3k_{3,1}k_{2,1}\hk_{2,1}'
\Bigr) .\\
\end{split}
\end{gather*}
The numerator $B(P_1, P_2)$ in (\ref{eq12:AA}) is
\begin{gather*}
\begin{split}
B(P_1, P_2)&=y_{16,1} B_{16}(P_1,P_2)+y_{13,1} y_{15,2}
B_{13,15}(P_1,P_2)+y_{14,1}y_{14,2} B_{14}(P_1,P_2)\\
&+ y_{13,2} y_{15,1} B_{13,15}(P_2,P_1) +y_{16,2} B_{16}(P_2,P_1),\\
\end{split}
\end{gather*}
where $B_{16}$, $B_{13,14}$ and $B_{15}$ are rational functions
of $x_1$ and $x_2$,
we evaluate:
\begin{gather*}
\begin{split}
B_{16}(P_1,P_2)&=\frac{
(k_{3,2} k_{2,2} \hk_{2,2} - k_{3,1} k_{2,1} \hk_{2,1})
}{x_1-x_2}
+k_{3,2}'k_{2,2}\hk_{2,2} +2k_{3,2}k_{2,2}'\hk_{2,2}+3k_{3,2}k_{2,2}\hk_{2,2}'
\\
&\quad\qquad\qquad +
5k_{3,1}'k_{2,1}\hk_{2,1} +4k_{3,1}k_{2,1}'\hk_{2,1}+3k_{3,1}k_{2,1}\hk_{2,1}',
\\
\end{split}
\end{gather*}
\begin{gather*}
\begin{split}
B_{13,15}(P_1,P_2)&=\frac{
k_{3,2} k_{2,2} -k_{3,1} k_{2,1} 
}{x_1-x_2} +
2k_{3,2}'k_{2,2} +4k_{3,2}k_{2,2}'
-4k_{3,1}'k_{2,1} -2k_{3,1}k_{2,1}', \\
\end{split}
\end{gather*}
and
\begin{gather*}
\begin{split}
B_{14}(P_1,P_2)&=\frac{
 k_{3,2} \hk_{2,2}  -k_{3,1} \hk_{2,1}  
}{x_1-x_2}
+3k_{3,2}'\hk_{2,2} +3k_{3,2}\hk_{2,2}'
-3k_{3,1}'\hk_{2,1} -3k_{3,1}\hk_{2,1}'. \\
\end{split}
\end{gather*}
We are concerned with $B(P_1, P_2)/(x_1-x_2)$ and thus compute
\begin{gather*}
\begin{split}
&\frac{B_{16}(P_1,P_2)}{(x_1-x_2)}=
-7x_{2}^5 
+\Bigr(-6\lthree{1}-x_{1}-5\ltwo{1}-4\hltwo{1}\Bigr)x_{2}^4\\
&\quad
+\Bigr(5x_{1}^2+(2\hltwo{1}+\ltwo{1})x_{1}-3\ltwo{2}-3\lthree{1}\hltwo{1}
-\hltwo{2}-5\lthree{2}-2\ltwo{1}\hltwo{1}-4\lthree{1}\ltwo{1}
  \Bigr)x_{2}^3\\
&\quad+\Bigr(11x_{1}^3+(6\lthree{1}+7\ltwo{1}+8\hltwo{1})x_{1}^2
 +(4\ltwo{1}\hltwo{1}+3\lthree{1}\hltwo{1} +2\lthree{1}\ltwo{1}\\
&\qquad
 +\lthree{2}+3\ltwo{2}+5\hltwo{2})x_{1}
 -\lthree{1}\ltwo{1}\hltwo{1}-2\lthree{2}\hltwo{1}
 -3\lthree{2}\ltwo{1}-2\lthree{1}\ltwo{2}\\
&\qquad
 +\ltwo{1}\hltwo{2}-4\lthree{3}\Bigr)x_{2}^2\\
&\quad
\end{split}
\end{gather*}
\begin{gather*}
\begin{split}
&+\Bigr(17x_{1}^4+(13\ltwo{1}+12\lthree{1}+14\hltwo{1})x_{1}^3\\
&\qquad
 +(10\ltwo{1}\hltwo{1}+9\lthree{1}\hltwo{1}+8\lthree{1}\ltwo{1}
+11\hltwo{2}+9\ltwo{2}+7\lthree{2})x_{1}^2\\
&\qquad
 +(7\ltwo{1}\hltwo{2}+6\lthree{1}\hltwo{2}
  +4\lthree{1}\ltwo{2}+6\ltwo{2}\hltwo{1}+4\lthree{2}\hltwo{1}
  +3\lthree{2}\ltwo{1}\\
&\qquad\quad
+2\lthree{3}
+5\lthree{1}\ltwo{1}\hltwo{1})x_{1}\\
&\qquad
+\lthree{1}\ltwo{2}\hltwo{1}+2\lthree{1}\ltwo{1}\hltwo{2}
  +3\ltwo{2}\hltwo{2}-2\lthree{3}\ltwo{1}-\lthree{2}\ltwo{2}\\
&\qquad
   +\lthree{2}\hltwo{2}-\lthree{3}\hltwo{1}\Bigr)x_{2}\\
&\quad
\end{split}
\end{gather*}
\begin{gather*}
\begin{split}
&+23x_{1}^5+(19\ltwo{1}+18\lthree{1}+20\hltwo{1})x_{1}^4
+(14\lthree{1}\ltwo{1}+15\lthree{1}\hltwo{1}
  +16\ltwo{1}\hltwo{1}\\
&\qquad
+13\lthree{2}+15\ltwo{2}+17\hltwo{2})x_{1}^3\\
&\qquad
+(10\lthree{1}\ltwo{2}+9\lthree{2}\ltwo{1}+10\lthree{2}\hltwo{1}
 +12\ltwo{2}\hltwo{1}+12\lthree{1}\hltwo{2}+13\ltwo{1}\hltwo{2}\\
&\qquad
 +8\lthree{3}+11\lthree{1}\ltwo{1}\hltwo{1})x_{1}^2\\
&\qquad
+(4\lthree{3}\ltwo{1}+5\lthree{2}\ltwo{2}+5\lthree{3}\hltwo{1}
  +7\lthree{2}\hltwo{2}+9\ltwo{2}\hltwo{2}
  +6\lthree{2}\ltwo{1}\hltwo{1}\\
&\qquad
+8\lthree{1}\ltwo{1}\hltwo{2}
  +7\lthree{1}\ltwo{2}\hltwo{1})x_{1}\\
&\qquad
+3\lthree{2}\ltwo{1}\hltwo{2}+4\lthree{1}\ltwo{2}\hltwo{2}
 +2\lthree{2}\ltwo{2}\hltwo{1}+\lthree{3}\ltwo{1}\hltwo{1}
+2\lthree{3}\hltwo{2}.\\
\end{split}
\end{gather*}
\begin{gather*}
\begin{split}
&\frac{B_{13,15}(P_1,P_2)}{(x_1-x_2)}=
-8x_{2}^3+(-4\ltwo{1}-6\lthree{1}-2x_{1})x_{2}^2
+(-2\lthree{1}\ltwo{1}+2\ltwo{1}x_{1}-4\lthree{2}+4x_{1}^2)x_{2}\\
&\qquad+10x_{1}^3+(8\ltwo{1}+6\lthree{1})x_{1}^2+(4\lthree{1}\ltwo{1}
+2\lthree{2}+6\ltwo{2})x_{1}+2\lthree{1}\ltwo{2}-2\lthree{3},\\
&\frac{B_{14}(P_1,P_2)}{(x_1-x_2)}=
-9x_{2}^3+(-6\hltwo{1}-3x_{1}-6\lthree{1})x_{2}^2
+(3x_{1}^2-3\lthree{1}\hltwo{1}\\
&-3\hltwo{2}-3\lthree{2})x_{2}+9x_{1}^3+(6\lthree{1}+6\hltwo{1})x_{1}^2
+(3\lthree{2}+3\lthree{1}\hltwo{1}+3\hltwo{2})x_{1}.\\
\end{split}
\end{gather*}
By these results we obtain the $\nuIIt{}$'s in the statement.
\end{proof}

\begin{corollary}
\label{cor:Sigmat}
\begin{enumerate}
\item
The one-form 
$$
\Pit_{P_1}^{P_2}(P):= \Sigmat(P, P_1)dx -  \Sigmat(P, P_2)dx  
$$
is a differential of the third kind,  whose only 
(first-order) poles are
$P=P_1$ and $P=P_2$ with residues $+1$ and $-1$ 
respectively.  

\item
The fundamental differential of the second kind
 $\Omegat(P_1, P_2)$ is given by
\begin{equation} 
\begin{split} 
\Omegat(P_1, P_2) &= d_{P_2} \Sigmat(P_1, P_2) 
     +\sum_{i = 1}^{12} \nuIt{i}(P_1)\otimes \nuIIt{i}(P_2)\\
  &=\frac{\Ft(P_1, P_2)dx_1 \otimes dx_2}
{(x_1 - x_2)^2 6^2 
y_{13,P_1}
y_{16,P_1}
y_{13,P_2}
y_{16,P_2}},
\label{eq:realization12}   
\end{split} 
\end{equation}
where $\Ff$ is an element of $R_{12} \otimes R_{12}$.
\end{enumerate}
\end{corollary}

The following result is easily obtained by 
observing the coefficient of the largest degree term
with respect to $P_1$,

\begin{lemma} 
\label{lemma:limFphi12}
We have
\begin{equation} 
\lim_{P_1 \to \infty} 
\frac{\Ft(P_1, P_2)}{\phit{11}(P_1)(x_1 - x_2)^2}
 = \phit{12}(P_2) =x_{P_2}^4 .
\label{eq:limF12}
\end{equation}
\end{lemma}

\medskip

%Hereafter we expand $h_i(P)= r y^{r-1} \nuIIt{i}(P) / dx $,
%$1\le i\le g,$ in the monomial  basis of $R$,
%  to avoid carrying holomorphic
%one-forms.

For later convenience we introduce the notation:
\begin{equation}
\label{eq:Omega_def12}
\begin{split}
\Omegat^{P_1, P_2}_{Q_1, Q_2}
 &:= \int^{P_1}_{P_2} \int^{Q_1}_{Q_2} \Omegat(P, Q) \\
 &= \int^{P_1}_{P_2} (\Sigmat(P, Q_1) - \Sigmat(P, Q_2)) 
 +\sum_{i = 1}^{12} \int^{P_1}_{P_2} \nuIt{i}(P) \int^{Q_1}_{Q_2} \nuIIt{i}(P).
\end{split}
\end{equation}

\bigskip
\bigskip

\section{The sigma function for $(3,7,8)$ and $(6,13,14,15,16)$ curves}
\label{sec:sigma}

\subsection{Legendre relation}
In this section,  $g$ will be $4$ or $12$.
%, $X$ as $X_4$ and $X_{12}$.
%Let $\alphag{i}, \betag{j}$  $ (1\leqq i, j\leqq g)$ be 
%$\alphaf{i}, \betaf{j}$  $ (1\leqq i, j\leqq 4)$ over $X_4$ or
%$\alphat{i}, \betat{j}$  $ (1\leqq i, j\leqq 12)$ over $X_{12}$.
%All other quantities should be regarded with the accessory $(4)$ or
%$(12)$, {\it{e.g.}},
%$\etafp{}$,
%%$\etafpp{}$,
%$\etatp{}$,
%$\etafpp{}$, and so on.
 In the monomial notation, as there were two conventions
 for $X_4$ and $X_{12}$,
 $\phig{}$ will be $\phiHf{}$ for $g = 4$
and $\phit{}$ for $g = 12$.

We write the periods:
\begin{equation}
\begin{split}
   \left[\,\etagp{}  \ \etagpp{}  \right]&:= 
\frac{1}{2}\left[\int_{\alphag{i}}\nuIIg{j} \ \ 
                 \int_{\betag{i}}\nuIIg{j}
\right]_{i,j=1,2,\cdots, g}.
   \label{eq2.5}
\end{split}
  \end{equation} 

Let $\taug{Q_1, Q_2}$ be the normalized 
differential of the third kind that
has residues $+1$, $-1$ at ${Q_1, Q_2}$,
is regular everywhere else,  and is  normalized,
 $\int_{\alpha_i} \taug{P, Q} = 0$ for every $i$ 
as in III.3.5 of \cite[III.3.5]{FK}.
The following Lemma corresponding 
 to  Corollary 2.6 (ii) in \cite{F1} holds:
\begin{lemma} \label{lemma:4.1}
$$
{\Omegag}^{P_1, P_2}_{Q_1, Q_2} = 
\int^{P_1}_{P_2}
\taug{Q_1, Q_2} + \sum_{i, j = 1}^g \gammag{ij} 
\int^{P_1}_{P_2} \nuIg{i} \int^{Q_1}_{Q_2} \nuIg{j},
$$
where $\gammag{} = {\omegagp{}}^{-1} \etagp{}$.
\end{lemma}

\begin{proof}
$
\partial_{P_1}{\Omegag}^{P_1, P_2}_{Q_1, Q_2} - 
\taug{Q_1, Q_2}$
must be expressed by the linear combination of the holomorphic
one forms.
Noting
$
\int^{P_2}_{P_1} \taug{Q_1, Q_2} =
\int^{Q_2}_{Q_1} \taug{P_1, P_2}
$ 
in III.3.5 of \cite[III.3.5]{FK},
$\gammag{ij}$ is symmetric.
Due to 
(\ref{eq:Omega_def4}) and (\ref{eq:Omega_def12}),
\begin{equation*}
 (\Sigma(P_1, Q_1) - \Sigma(P_1, Q_2)) 
 +\sum_{i = 1}^{g}  \nuIg{i}(P_1) \int^{Q_1}_{Q_2} \nuIIg{i}(P)
=\taug{Q_1, Q_2}(P_1) + \sum_{i, j = 1}^g \gammag{ij} 
 \nuIg{i}(P_1) \int^{Q_1}_{Q_2} \nuIg{j}.
%\label{eq:SigmaOmega}
\end{equation*}
By choosing  an appropriate path $\Gamma$ from $Q_1$ to $Q_2$,
homotopic to  $\alphag{k}$, we have
$(\etagp{} \nuIg{})^t =\gammag{} \cdot (\omegagp{} \nuIg{})^t$.
\end{proof}

The following Proposition provides  {\it generalized Legendre relation}
\cite{B1,BLE1,BLE2}, which determines a symplectic structure 
in the Jacobian \cite[Ch.3,4]{Po}(Ch.3,4).
\begin{proposition}{\rm{(generalized Legendre relation)}}
The matrix,
\begin{equation}
   M := \left[\begin{array}{cc}2\omegagp{} & 2\omegagpp{} \\
               2\etagp{} & 2\etagpp{}
     \end{array}\right],
\end{equation} 
 satisfies 
\begin{equation}
   M\left[\begin{array}{cc} & -1 \\ 1 & \end{array}\right]{}^t {M}
   =2\pi\sqrt{-1}\left[\begin{array}{cc} & -1 \\ 1 &
     \end{array}\right].
   \label{eq2.7}
\end{equation} 
\end{proposition}

\begin{proof}
By comparing Lemma \ref{lemma:4.1} and, 
(\ref{eq:realization4})  in Corollary \ref{cor:Sigmaf} and
(\ref{eq:realization12}) in Corollary \ref{cor:Sigmat},
we choose appropriately $(2g)^2$ paths  and take the integrals 
along these paths. For example, for
$$
 \int^{P_1}_{P_2} (\Sigma(P, Q_1) - \Sigma(P, Q_2)) 
 +\sum_{i = 1}^{g} \int^{P_1}_{P_2} \nuIg{i}(P)
 \int^{Q_1}_{Q_2} \nuIIg{i}(P)
$$ $$
=\int^{Q_1}_{Q_2}
\taug{P_1, P_2} + \sum_{i, j = 1}^g \gammag{ij} 
\int^{P_1}_{P_2} \nuIg{i} \int^{Q_1}_{Q_2} \nuIg{j},
$$
the contour of $Q_1$ to $Q_2$ along $\alpha_k$ gives
$$
\sum_{i = 1}^{g} \int^{P_1}_{P_2} \nuIg{i}(P) \etagp{k i}(P)
= \sum_{i, j = 1}^g \gammag{ij} 
\int^{P_1}_{P_2} \nuIg{i} \omegagp{k j},
$$
and further contour integration provides 
${}^t(^t\omegagp{} \etagp{})= (^t\omegagp{} \etagp{})$.
The contour of $Q_1$ to $Q_2$ along $\betag{k}$ gives
$$
 \sum_{i = 1}^{g} \int^{P_1}_{P_2} \nuIg{i}(P)
 \etagpp{k i}(P)
= 2\pi\sqrt{-1}\int^{P_1}_{P_2} ({\omegagp{}}^{-1} \nuIg{})_k +
\sum_{i, j = 1}^g \gammag{ij} 
\int^{P_1}_{P_2} \nuIg{i} \omegagp{k j},
$$
which, in view of the contour integral, turns out to be
$$
{}^t(^t\omegagpp{} \etagpp{})= 2\pi \sqrt{-1} {\mathbb{T}}
+ (^t\omegagpp{} \gammag{}\omegagpp{}),
$$
$$
(^t\omegagp{} \etagpp{})= 2\pi \sqrt{-1} 1_g 
          +(^t\omegagp{} \gammag{}\omegagpp{}).
$$
Here we use 
$
\int_{\beta_k} \taug{Q_1, Q_2} = 2\pi \sqrt{-1}
\int^{Q_1}_{Q_2} (\omegagp{}{}^{-1} \nuIg{})_k
$ in III.3.5 of \cite[III.3.5]{FK}.
\end{proof}

\subsection{$\sigma$ function}

%By the  Riemann relations \cite{F1}, it is known that
% $\text{Im}\,({\omegagp{}}^{-1}\omegagpp{}) $ is positive definite.
%As in  Theorem 1.1 in \cite{F1}, let
%\begin{equation}
%   \delta:=\left[\begin{array}{cc}\delta'\ \\
%       \delta''\end{array}\right]\in \left(\tfrac12\ZZ\right)^{2g}
%   \label{eq2.9} %3.15
%\end{equation} 
%be the theta characteristic which is equal to the Riemann constant 
%$\xi_R$ with respect to the base point $\infty$ and the period matrix 
%$[\,2\omegagp{}\ 2\omegagpp{}]$. 
%%%%%%%%%%%%%%%%
By the Riemann relations \cite{F1}, it is known that
 $\text{Im}\,({\omegagp{}}^{-1}\omegagpp{}) $ is positive definite.

We let
\begin{equation}
   \delta:=\left[\begin{array}{cc}\delta''\ \\
       \delta'\end{array}\right]\in \left(\tfrac12\ZZ\right)^{2g}
   \label{eq2.9} %3.15
\end{equation}
%{\bf{\noindent SM 07132013 end:\\ \noindent}}
be the theta characteristic which is equal to the Riemann constant
$\xi_R$ with respect to the base point $\infty$ and the period matrix
$[\,2\omegagp{}\ 2\omegagpp{}]$. This will be given explicitly in terms of
divisors by Proposition \ref{prop:thetadivisor} and
Corollary \ref{cor:thetadivisor}.

 We define an entire function of (a column-vector)
$u={}^t\negthinspace (u_1, u_2, \ldots, u_g)\in \mathbb{C}^g$,

\begin{equation}
\begin{aligned}
   \sigmag{}(u)&=\sigmag{}(u;M)=\sigmag{}(u_1, u_2, \ldots, u_g;M) \\
   &=c\,\text{exp}(-\tfrac{1}{2}{}\ ^t\negthinspace  
u\etagp{}{\omegagp{}}^{-1}u)
   \vartheta\negthinspace
   \left[\delta\right](\frac{1}{2}{\omegagp{}}^{-1} u;\ 
{\omegagp{}}^{-1}\omegagpp{}) \\
   &=c\,\text{exp}(-\tfrac{1}{2}\ ^t\negthinspace 
u\etagp{}{\omegagp{}}^{-1}\  u) \\
   &\hskip 20pt\times
   \sum_{n \in \ZZ^g} \exp \big[\pi \sqrt{-1}\big\{
    \ ^t\negthinspace (n+\delta'')
      {\omegagp{}}^{-1}\omegagpp{}(n+\delta'')
   + \ ^t\negthinspace (n+\delta'')
      ({\omegagp{}}^{-1} u+2\delta')\big\}\big], 
\end{aligned}
   \label{de_sigma}
\end{equation}
where  $c$ is a certain constant, in fact a rational 
function of the $b$'s.
%In this article, the constant $c$ is not  important because we deal 
%only with ratios of $\sigmag{}$ functions,
%so we say nothing more about it.

%On the definition of $\sigmag{}$, we use the symplectic structure associated
%with the generalized Legendre relation but it agrees with the Hodge
%structure of the theta function \cite{L}.

For a given $u\in\CC^g$, we  introduce
$u'$ and $u''$ in $\RR^g$ so that
\begin{equation*}
   u=2\omegagp{} u'+2\omegagpp{} u''.
\end{equation*}

\begin{proposition} \label{prop:pperiod}
For $u$, $v\in\CC^g$, and $\ell$
$(=2\omegagp{}\ell'+2\omegagpp{}\ell'')$ $\in\Lambda$, we define
\begin{align*}
  L(u,v)    &:=2\ {}^t{u}(\etagp{}v'+\etagpp{}v''),\nonumber \\
  \chi(\ell)&:=\exp[\pi\sqrt{-1}\big(2({}^t {\ell'}\delta''-{}^t
  {\ell''}\delta') +{}^t {\ell'}\ell''\big)] \ (\in \{1,\,-1\}).
\end{align*}
The following holds
\begin{equation}
	\sigmag{}(u + \ell) = 
\sigmag{}(u) \exp(L(u+\frac{1}{2}\ell, \ell)) \chi(\ell).
        \label{eq:4.11}
\end{equation}
\end{proposition}

\begin{proof}
Direct computations using the definition of $\sigma^{(g)}$ give the result
as in Chapter VI of \cite{L}.
%The proof is standard, noting that $\sigmag{}$ is associated to
%the theta function   in Chapter VI of \cite{L}.
\end{proof}

\begin{remark} \label{rmk:Schur}
{\rm{
  Following a formula proven for the  rational/polynomial case by
Buchstaber,  Leykin and  Enolskii \cite{BEL2},
in \cite{Na} Nakayashiki showed that
the leading term in the Taylor expansion
of the $\sigma$ function associated with $(r,s)$ curve 
 is expressed by a Schur function by normalizing the constant factor $c$.
We also expect that
$$
   \sigmag{}(u) = S_{\Lambda}(T)|_{T_{\Lambda_i + g - i} = u_i} 
            + \sum_{\alpha} a_\alpha u^\alpha,
$$
where $a_\alpha\in \QQ[b_1, \cdots, b_\ell]$
for $\ell = 5$ for $X_4$ and $\ell = 7$ for $X_{12}$,
$\alpha = \alpha_1 \cdots \alpha_g$
and  $u^\alpha = u_1^{\alpha_1} \cdots u_g^{\alpha_g}$.
Here for a Young diagram $\Lambda$,
$S_\Lambda$ and $s_\Lambda$ are the Schur functions defined by
$$
	S_{\Lambda}(T) = s_{\Lambda}(t), \quad
          T_k :=\frac{1}{k}\sum_{i=1}^g t_j^k.
$$
}}
\end{remark}

Let us simply write $\WW^k$ instead of $\WW^{(g) k}$.
The vanishing locus of $\sigmag{}$ is:
\begin{equation}
	\Theta^{g-1} =( \WW^{g-1} \cup [-1] \WW^{g-1}) = \WW^{g-1}.
\label{eq:Theta:g-1}
\end{equation}
The last equality is due to Proposition \ref{prop:pperiod},
which shows that $\sigmag{}$ is an even or odd
function under the action of $[-1]$; the reason for
introducing $\WW^{g-1} \cup [-1] \WW^{g-1}$ is that the analogous
loci when $g-1$ is replaced by $k$ play an important role
and $\WW^{k}$ is not $[-1]$-invariant in general.

\subsection{The Riemann fundamental relation} 
%%%%%%%%%%%%%%%%%SM 06302013

%%%%%%%%%%%%%%%%%SM 06302013
Using (\ref{eq:-1u}) to detect the divisors corresponding to
the minus-sign operation on $\JJ_4$,
we review a relation which we  call the Riemann fundamental
relation (\cite{R}, \S195 in
\cite{B1}):
\begin{proposition}
For $(P, Q, P_i, P'_i) \in X^2 \times (S^g(X)\setminus S^g_1(X)) \times
(S^g(X)\setminus S^g_1(X))$,
%if
$$
        u  := \hu(P_1, \ldots, P_g),   \quad
        v  := \hu(P'_1, \ldots, P'_g),   \quad
$$
\begin{align*}
\exp\left(
\sum_{i, j = 1}^g
   {\Omegag}_{P_i, P'_j}^{P, Q} \right)
&=
\frac{\sigmag{}(\hu_o(P) - u) \sigmag{}(\hu_o(Q) - v)}
     {\sigmag{}(\hu_o(Q) - u) \sigmag{}(\hu_o(P) - v)}\\
&=\frac{\sigmag{}(\hu_o(P) - \hu(P_1, \ldots, P_g))
        \sigmag{}(\hu_o(Q) - \hu(P'_1, \ldots, P'_g))}
     {\sigmag{}(\hu_o(Q) - \hu(P_1, \ldots, P_g))
      \sigmag{}(\hu_o(P) - \hu(P'_1, \ldots, P'_g))}.
\end{align*}
\end{proposition}

\begin{proof}
%\mathbf{Maybe the $w$'s below should be $u$'s}
The right-hand side can be expressed as
$$
\exp(\mbox{bilinear term in $u$'s})
\frac{\theta(\omegagp{}{}^{-1}(\hu_o(P) - u)+\xi_R )
      \theta(\omegagp{}{}^{-1}(\hu_o(Q) - v)+\xi_R )}
     {\theta(\omegagp{}{}^{-1}(\hu_o(Q) - u)+\xi_R )
       \theta(\omegagp{}{}^{-1}(\hu_o(P) - v))+\xi_R )},
$$
where $\xi_R$ is the Riemann constant.
By Riemann's theorem for theta functions in 
\cite{F1} (p. 23),
the above becomes
$$
\exp(\mbox{bilinear term in $u$'s})
\exp\left(\sum_{j=1}^g \int^P_Q \taug{P_j, P'_j}\right).
$$
The exponential part of the bilinear term is given by
$$
 (u - v)^t \gammag{} (\hu_o(P) - \hu_o(Q)) =
  \sum_{i, j, k}
 \gammag{ij} \int^{P_k}_{P'_k} \nuIg{i} \int^P_Q \nuIg{j},
$$
which equals (Lemma (\ref{lemma:4.1}))
$$
\sum_{i=1}^g {\Omegag}^{P, Q}_{P_i, P'_i}
-
\sum_{i=1}^g \int^P_Q \taug{P_i, P'_i} .
$$
The above integrals depend upon the paths we choose, but
(\ref{eq:4.11}) shows that such dependence cancels.
Thus the right-hand side coincides with the left-hand side.
\end{proof}

%\begin{proposition} \label{prop:wpxx=F}
%For $(P, P_1, \ldots, P_g) \in X \times S^g(X) \setminus S^g_1(X)$
%and $u := \uab(P_1, \ldots, P_g)$,
%the equality
%$$
%	\sum_{i, j = 1}^g \wpg{i, j}
%          \left( \uab(P) - u\right)
%         \phig{i-1}(P)
%         \phig{j-1}(P_a) =\frac{\Fg(P, P_a)}{(x-x_a)^2},
%$$
%holds for every $a = 1, 2, \ldots, g$,
%where we set
%$$
%	\wpg{ij}(u) := -\frac{\sigmag{i}(u) \sigmag{j}(u) 
%              -  \sigmag{}(u) \sigmag{ij}(u)}
%               {\sigmag{}(u)^2}\equiv
%        -\frac{\partial^2}{\partial u_i\partial u_j}\log\sigmag{}(u).
%$$
%Here for $g=4$ case, $\phig{i}$ is interpreted as $ \phiHf{i}$.
%\end{proposition}

\begin{proposition} \label{prop:wpxx=F}
For $(P, P_1, \ldots, P_g) \in X \times S^g(X) \setminus S^g_1(X)$
and $u := \hu(P_1, \ldots, P_g)$,
the equality,
$$
        \sum_{i, j = 1}^g \wpg{i, j}
          \left(\hu_o(P) - u\right)
         \phig{i-1}(P)
         \phig{j-1}(P_a) =\frac{\Fg(P, P_a)}{(x-x_a)^2},
$$
holds for every $a = 1, 2, \ldots, g$,
where we set
$$
        \wpg{ij}(u) := -\frac{\sigmag{i}(u) \sigmag{j}(u)
              -  \sigmag{}(u) \sigmag{ij}(u)}
               {\sigmag{}(u)^2}\equiv
        -\frac{\partial^2}{\partial u_i\partial u_j}\log\sigmag{}(u).
$$
Here for $g=4$ case, $\phig{i}$ is interpreted as $ \phiHf{i}$.
\end{proposition}

\begin{proof}
For the case $X_4$, using the relation (\ref{eq:partial_H4}) and
taking logarithm of both sides in the Riemann
fundamental relation and differentiating
along $P_1=P$ and $P_2=P_a$, we obtain the claim.
Similarly we have the result for the case $X_{12}$ using
(\ref{eq:partial_H12}).
\end{proof}

\bigskip
\bigskip

\section{Jacobi inversion formulae
for $(3,7,8)$ and $(6,13,14,15,16)$ curves}
\label{sec:JacobiInv}

In this section, we will consider a Jacobi-inversion type formula
associated with $X_4$ and $X_{12}$.

\subsection{The $\mug{n}$ functions over 
$(3,7,8)$ and $(6,13,14,15,16)$ curves}\label{the mu function}

In this subsection,  $g$ will be $4$ or $12$, and
$\phig{}$ accordingly,  $\phiHf{}$, $\phiMf{}$, or $\phit{}$. 
(We refer to $\WW^k$ instead of $\WW^{(g) k}$.)
In \cite{MP08},
we introduced meromorphic functions on the curve,
reviewed here in
Definition \ref{def:mul}, 
which  generalize the polynomial  
$U$ in Mumford's $(U,V,W)$ parameterization of
a hyperelliptic Jacobian (which he attributes to Jacobi)
\cite{Mu}.

\medskip

For the definition of the function $\mu$'s,
 we introduce the Frobenius-Stickelberger (FS) matrix
and its determinant following \cite{MP08}.
Let $n$ be a positive integer  and 
$P_1, \ldots, P_n$ be in $X_g\backslash\infty$.
As in (\ref{eq:partialInPsi4}) and (\ref{eq:partialInPsi12}),
we define the \textit{$\ell$-reduced  
Frobenius-Stickelberger} (FS) \textit{matrix} by:
$$
\Psi_{n}^{(g;\check\ell)}(P_1, P_2, \ldots, P_n) := 
\begin{pmatrix}
\phig0(P_1) &\phig1(P_1) & \phig2(P_1)  &\cdots & \check{\phig{\ell}}(P_1) 
& \cdots & \phig{n}(P_1) \\
\phig0(P_2) & \phig1(P_2) & \phig2(P_2) &\cdots & \check{\phig{\ell}}(P_2) 
 & \cdots & \phig{n}(P_2) \\
\vdots & \vdots & \vdots & \ddots& \vdots & \ddots& \vdots\\
\phig0(P_n) & \phig1(P_{n}) & \phig2(P_{n}) &\cdots & \check{\phig{\ell}}(P_n) 
 & \cdots&  \phig{n}(P_{n})
\end{pmatrix},
$$
and $\psi_{n}^{(g;\check\ell)}(P_1, P_2, \ldots, P_n) := 
\det \Psi_{n}^{(g;\check\ell)}(P_1, P_2, \ldots, P_n)$
(a check on top of a letter signifies deletion).
It is also convenient to introduce the simpler notation:
\begin{gather*}
\psig{n}(P_1, \ldots, P_n)
 := \det(\Psi_{n}^{(g;\check{n})}(P_1, \ldots, P_n)), \quad
\Psig{n}(P_1, \ldots, P_n):= \Psi_{n}^{(g;\check{n})}(P_1, \ldots, P_n),
\end{gather*}
for the un-bordered matrix.
%Here $\phi^{(g)}$ reads $\phiHf{}$, $\phiMf{}$ or $\phit{}$.
We call this matrix {\it{Frobenius-Stickelberger (FS) matrix}}
and its determinant {\it{Frobenius-Stickelberger (FS) determinant}}.
These become singular for some tuples in $(X_g \backslash \infty)^n$.

\begin{definition} \label{def:mul}
For $P, P_1, \ldots, P_n$ $\in (X_g\backslash\infty) \times \SS^n(X_g\backslash\infty)$,
we define $\mug{n}(P)$ by
$$
\mug{n}(P): = 
\mug{n}(P; P_1,  \ldots, P_n): = 
\lim_{P_i' \to P_i}\frac{1}{\psig{n}(P_1' , \ldots, P_n' )}
\psig{n+1}(P_1' , \ldots, P_n' , P),
$$
where the $P_i^\prime$ are generic,
the limit is taken (irrespective of the order) for each $i$;
and $\mug{n, k}(P_1, \ldots, P_n)$ by
$$
\mug{n}(P)
 = \phig{n}(P) + 
\sum_{k=0}^{n-1} (-1)^{n-k}\mug{n, k}(P_1, \ldots, P_n) \phig{k}(P),
$$
with the convention $\mug{n, n}(P_1, \ldots, P_n) \equiv 1$.
Let $\mu_{H^a, n}^{(4)}$ be $\muf_{n}$ for $\phi_{H^a, n}^{(4)}$ for $a = 0, 1$.

\end{definition}

Meromorphic functions, viewed as divisors
on the curve, allow us to express the addition structure
of $\mathrm{Pic}X_g$ in terms of  FS-matrices.
For $n$ points $(P_i)_{i=1, \ldots, n}$ $\in X_g\backslash\infty$,
we find an element of $R$ associated with 
 any point $P= (x,y)$ in $(X_g\backslash\infty )$,
$\alpha_n(P) :=\alpha_n(P; P_1, \ldots, P_n)
= \sum_{i=0}^{n} a_i \phi_i(P)$, $a_i \in \CC$ and $a_n = 1$,
which has a  zero at each point $P_i$ (with multiplicity, if
the $P_i$ are repeated)
and has smallest possible order of pole at $\infty$ with this property;
$\alpha_n$ can be identified with $\muMf{n}$ for $X_4$ and
$\mut{n}$ for $X_{12}$.

%%%%%%%%%%%%%%%%%%%%%%%% SM 06302013

%On $X_4$, instead of the pole condition on elements of
%$H^0(X_4\setminus \infty, \cO_{X_4})$, we impose
%an alternative condition to determine a polynomial $\alphaH_n
%= \sum_{i=0}^{n} a_i \phiHf{i}(P)$, $a_i \in \CC$ and $a_n = 1$,
% namely to have a (possibly multiple)  zero at each point $P_i$
%and  smallest possible degree to match the
%order of pole of $d x/3 y_7 y_8$ 
%at $\infty$;
%$\alphaH_n$ can be identified with $\muHf{n}$, in fact
%$\displaystyle{\frac{\muHf{n}dx}{3y_7y_8}}$ does not have a singularity over
%$X\setminus \infty$ whereas
%$\displaystyle{\frac{\mu_{H^0,n}^{(4)}dx}{3y_7y_8}}$ does.

%For $\alpha_n$ for $X_g$ $(g=4,12)$ and
%$\alphaH_n$ for $X_4$ and $\Ng(n)$ 
%$(\NHf(n), \NMf(n), \Nt(n))$, we have the following lemma.
%\begin{lemma}\label{prop:2theta2}
%Let $n$ be a positive integer.
%For $(P_i)_{i=1,\ldots, n}\in \SS^n(X \backslash\infty) $,
%the  function  $\alpha_n$
%over $X_g$ induces the map (which we call by the same name): %injection 
%$$
%\alpha_n: 
%\SS^n(X_g \backslash\infty)  \to \SS^{\Ng(n) - n}(X_g), 
%$$
%{{i.e.}}, to 
%$(P_i)_{i=1,\ldots, n}\in \SS^n(X_g \backslash\infty) $ there corresponds an % unique
% element $(Q_i)_{i=1,\ldots, \Ng(n)-n}\in \SS^{\Ng(n)-n}(X_g)$,
%such that
%$$
%\sum_{i=1}^{n} P_i - n \infty 
%\sim - \sum_{i=1}^{\Ng(n)-n} Q_i  + (\Ng(n)-n) \infty .
%$$
%\end{lemma}
%%%%%%%%%%%%%%%%%%%%%%%% SM 06302013
For $\alpha_n$ for $X_g$ $(g=4,12)$ and
$\Ng(n)$ $(\NMf(n), \Nt(n))$, we have the following lemma:
\begin{lemma}\label{lmm:2theta2}
Let $n$ be a positive integer.
For $(P_i)_{i=1,\ldots, n}\in \SS^n(X \backslash\infty) $,
the  function  $\alpha_n$
over $X_g$ induces the map (which we call by the same name): %injection
$$
\alpha_n:
\SS^n(X_g \backslash\infty)  \to \SS^{\Ng(n) - n}(X_g),
$$
{{i.e.}}, to
$(P_i)_{i=1,\ldots, n}$ $\in \SS^n(X_g \backslash\infty) $
there corresponds an % unique
 element $(Q_i)_{i=1,\ldots, \Ng(n)-n}\in \SS^{\Ng(n)-n}(X_g)$,
such that
$$
\sum_{i=1}^{n} P_i - n \infty
\sim - \sum_{i=1}^{\Ng(n)-n} Q_i  + (\Ng(n)-n) \infty .
$$
\end{lemma}

On the other hand,
$\displaystyle{\frac{\muHf{n}dx}{3y_7y_8}}$ does not have a singularity over
$X\setminus \infty$ whereas
$\displaystyle{\frac{\mu_{H^0,n}^{(4)}dx}{3y_7y_8}}$ does.
The divisor of $\muHf{n}(P)$ contains $\sum_{i=1}^5 B_a - 5 \infty\sim$
$2(B_4+B_5) - 4 \infty$.
For $\alphaH_n:=\muHf{n}$ for $X_4$ and $\NHf(n)$, we have the following lemma:

\begin{lemma}\label{lem:2theta4}
Let $n$ be a positive integer.
For $(P_i)_{i=1,\ldots, n}\in \SS^n(X_4 \backslash\infty) $,
the  function  $\alphaH_n$
over $X_4$ induces the map (which we call by the same name): %injection
$$
\alphaH_n:
\SS^n(X_4 \backslash\infty)  \to \SS^{\NHf(n) - n-5}(X_4),
$$
{{i.e.}}, to $(P_i)_{i=1,\ldots, n}$ $\in \SS^n(X_4 \backslash\infty)$
there corresponds an element \break
 $(Q_i)_{i=1,\ldots,\NHf(n)-n-5}$
%\mathbf{maybe the following should be instead 
%$\SS^{\NHf(n) - n-5}(X_4)$}
 $\in \SS^{\NHf(n) - n-5}(X_4)$, such that
$$
\sum_{i=1}^{n} P_i +B_4+B_5 - (n+2) \infty
\sim -\left( \sum_{i=1}^{\NHf(n)-n-5} Q_i
+B_4+B_5 - (\NHf(n)-n-3) \infty \right).
$$
\end{lemma}

\begin{proof}
The linear equivalence holds because:
\begin{gather*}
\begin{split}
&\sum_{i=1}^{n} P_i +\sum_{a=1}^5 B_a
+ \sum_{i=1}^{\NHf(n)-n-5} Q_i  -\NHf(n) \infty\\
&\sim\sum_{i=1}^{n} P_i  + \sum_{i=1}^{\NHf(n)-n-5} Q_i
+2(B_4 + B_5) -(\NHf(n) -1)\infty \sim 0.
\end{split}
\end{gather*}
\end{proof}

%%%%%%%%%%%%%%%%%SM 06302013

\bigskip
We want the preimage of $\alpha_n$
 to include the base point $\infty$.
For an effective divisor $D$ in $\SS^n(X_g)$ of degree $n$,
let $D'$  be the maximal subdivisor  of $D$ which does
not contain $\infty$,
$D = D' + (n-m) \infty $
where $\deg D'=m (\le n)$ and
$D' \in \SS^m(X_g\backslash\infty)$.
Then we extend the map to $\overline{\alpha}_n$
by defining $\overline{\alpha}_n(D)={\alpha}_m(D^\prime )+[\Ng(n)-n-
(\Ng(m)-m)]\infty.$
In \cite{MP09}, we use the following relation as
the Serre duality but  due to Remark \ref{rmk:CanonDiv},
 the expression should be modified for $X_4$.

%%%%%%%%%%%%%%%%%SM 06302013

We see from the linear equivalence of Lemmas \ref{lmm:2theta2}
 and \ref{lem:2theta4}:
\begin{proposition} \label{prop:addition}
For a positive integer, the
Abel map composed with $\alpha_n$
of $\mut{}$ for $X_{12}$ and
 $\alphaH_n$ of $\muHf{n}$ for $X_4$ induce
$$
\iota_n :  \WW^n \to  \WW^{\Nt(n) - n} \quad \mbox{for } X_{12},
$$
and
$$
\iota_n :  \WW^n \to  \WW^{\NHf(n) - n-5} \quad \mbox{for } X_{4}.
$$
\end{proposition}

\begin{remark} \label{rmk:addition}
Due to Proposition \ref{prop:addition}, we have introduced the shifted Abelian
map $\hu$ rather than $\hu_o$ in (\ref{eq:shiftedAmap}) especially for
$X_4$.
In other words, for any $P_1, P_2, P_3, P_4$ and appropriate $Q$'s in
$X_4$,
we have 
%\mathbf{Maybe this should be $-\hu(P_1, P_2) =\hu(Q_1'', Q_2'', Q_3'')$} 
\begin{equation}
\begin{array}{ll}
-\hu(P_1) =\hu(Q_1''', Q_2''') , 
&-\hu(P_1, P_2) =\hu(Q_1'', Q_2'', Q_3'') ,\\
-\hu(P_1, P_2, P_3) =\hu(Q_1', Q_2', Q_3') ,
&-\hu(P_1, P_2, P_3, P_4) =\hu(Q_1, Q_2, Q_3, Q_4) .
\end{array}
\end{equation}
%\mathbf{I am not able to achieve consistency. 
%The $N$'s are not given but from subsection 3.2 I get:
%${\NHf(n) - n-5}$ when $n=1$ is $1-1-5$ (etc. $n=2,3,4$
%so from Lemma 6.2 there should be $-5$ $Q$'s.}  
The third relation,
\begin{equation}
\begin{split}
%-\hu(P_1, P_2, P_3) &=\hu(Q_1, Q_2, Q_3) ,\quad \mbox{i.e.},\\
-\hu_o(P_1, P_2, P_3) &=\hu_o(Q_1, Q_2, Q_3) + 2\hu_o(B_4, B_5),
\label{eq:-1u}
\end{split}
\end{equation}
is quite important since it  shows (see \cite{Mu1}  p.166)
that Serre duality on $X_4$ is given as
 $\iota_{3} : \WW^3 \to  \WW^{3}$ by
$$
P_1 + P_2 + P_3 +B_4+B_5 - 5 \infty
\sim -\left( Q_1 + Q_2 + Q_3 +B_4+B_5 - 5\infty \right).
$$
\end{remark}

\begin{proposition} \label{prop:W3W3}
%We have the Serre duality for $X_4$ case,
%Further
For $X_4$ we have the relation
$$
\WW^3 \subset \hu_o(S^3(X_4)).
$$
\end{proposition}

\begin{proof}
Noting $(y_7/y_8)=B_4 + B_5 +\infty - B_1 - B_2 - B_3$,
there are $P_a' \in X_4$ $(a=1,2,3)$ such that
\begin{gather*}
\begin{split}
P_1 + P_2 + P_3 +B_4+B_5 - 5 \infty
&\sim -\left( Q_1 + Q_2 + Q_3 +B_4+B_5 - 5\infty \right)\\
&\sim -\left( Q_1 + Q_2 + Q_3 +B_1+B_2 + B_3 - 6\infty \right)\\
&\sim P_1' + P_2' + P_3' - 3\infty.
\end{split}
\end{gather*}
In the last relation,
 we consider the divisor of $\muMf{5}(P; Q_1, Q_2, Q_3, B_1, B_2, B_3)$
noting $\NMf_{5} = 9$.
\end{proof}

Let $\mathrm{image}(\iota_n)$ be denoted by $[-1]\WW^n$,
especially $\iota_{g}: \WW^{g} \to [-1] \WW^{g}$.
Noting for $n\ge g$,
$\iota_g \circ\iota_n$ gives the Abel sum
$$
 \WW^n \stackrel{\iota_n}{\longrightarrow}
  \WW^{g} \stackrel{\iota_g}{\longrightarrow} \WW^{g},
\quad
(\uab(P_1,\ldots ,P_n) \equiv \uab( Q_1,\ldots ,Q_g)\ \mathrm{mod}\Lambda).
$$
In particular,
the addition law on the Jacobian is given by
 $\iota_g \circ\iota_{2g}$
$$
 \WW^{2g} \stackrel{\iota_{2g}}{\longrightarrow}
  \WW^{g}
  \stackrel{\iota_g}{\longrightarrow} \WW^{g},
\quad
(\uab(P_1, \ldots ,P_g , P'_1, \ldots , P'_g)\equiv \uab(Q_1, \ldots , Q_g)
\ \mathrm{mod}\Lambda).
$$

\begin{proposition} \label{prop:thetadivisor}
There are points
$P_{R,1}$, $P_{R,2}$ and $P_{R,3}$ of $X_4 \setminus \infty$
such that $2\hu(P_{R,1}, P_{R,2}, P_{R,3}) =0$ and
 the Riemann constant associated with $X_4$
is given by
$$
\xi_R=\hu(P_{R,1}, P_{R,2}, P_{R,3}).
$$
\end{proposition}

\begin{proof}
Let us consider  the points 
$P_{0,1}$, $P_{0,2}$ and $P_{0,3}$ of $X_4 \setminus \infty$
satisfying
\begin{equation}
 2(P_{0,1}+P_{0,2}+P_{0,3}) +2 (B_4+B_5) -10\infty \sim 0,
\label{eq:proofP6.3}
\end{equation}
and
$$
 (P_{0,1}+P_{0,2}+P_{0,3}) + (B_4+B_5) -5\infty \not\sim 0.
$$

 From Lemma \ref{lem:detbetas},
the condition means that every permutation
$(P_1, P_2, P_3)$ of
$(P_{0,1}, P_{0,2},P_{0,3})$ is a non-trivial solution of 
$\beta(P_1, P_2, P_3)=0$ for
\begin{gather*}
\begin{split}
\beta(P_1, P_2, P_3)
=\left|\begin{matrix} 
y_7(P_{1}) & y_8(P_{1}) & (xy_7)(P_{1}) & (xy_8)(P_{1}) \\
\frac{d y_7}{d x}(P_{1}) & \frac{d y_8}{d x}(P_{1}) 
& \frac{d x y_7}{d x}(P_{1}) & \frac{d x y_8}{d x}(P_{1}) \\
y_7(P_{2}) & y_8(P_{2}) & (xy_7)(P_{2}) & (xy_8)(P_{2}) \\
y_7(P_{3}) & y_8(P_{3}) & (xy_7)(P_{3}) & (xy_8)(P_{3}) 
\end{matrix}\right|.\\
\end{split}
\end{gather*}
The non-triviality of the solution means %that the linear equivalence:
\begin{equation}
\begin{split}
\mbox{1) }& (P_{0,1}+P_{0,2}+P_{0,3}) -3\infty \sim 0
\mbox{ does not hold,}\\
\mbox{2) }& P_{0,a} \mbox{ is not equal to } B_b, \ \  a=1,2,3, \ b=1,\ldots,5. 
\label{eq:avoid}
\end{split}
\end{equation}
%Due to the symmetry,
%the combination actually means that
%$(P_1, P_2, P_3)=(P_{0,1}, P_{0,2},P_{0,3})$,
%$(P_1, P_2, P_3)=(P_{0,2}, P_{0,3},P_{0,1})$ and
%$(P_1, P_2, P_3)=(P_{0,3}, P_{0,1},P_{0,2})$.
 By employing the notation $[A,B]:=A_7B_8-A_8B_7$ in this proof,
we have the relation,
\begin{gather*}
\begin{split}
\beta(P_1, P_2, P_3)&=
[y(P_1), \frac{dy}{dx}(P_1)] [y(P_2), y(P_3)]
(x(P_1)-x(P_2)) (x(P_1)-x(P_3))\\
&+[y(P_1), y(P_2)] [y(P_1), y(P_3)]
(x(P_2)-x(P_3)).\\
\end{split}
\end{gather*}
Further the relations
 $\beta(P_1, P_2, P_3)=0$,
 $\beta(P_2, P_3, P_1)=0$ and $\beta(P_3, P_1, P_2)=0$
 mean
\begin{gather*}
\begin{split}
&[y(P_1), \frac{dy}{dx}(P_1)] 
\frac{[y(P_2), y(P_3)]}{x(P_2)-x(P_3)}+
\frac{[y(P_1), y(P_2)]}{x(P_1)-x(P_2)}
\frac{[y(P_1), y(P_3)]}{x(P_1)-x(P_3)}  =0,\\
&[y(P_2), \frac{dy}{dx}(P_2)] 
\frac{[y(P_3), y(P_1)]}{x(P_3)-x(P_1)}+
\frac{[y(P_2), y(P_3)]}{x(P_2)-x(P_3)}
\frac{[y(P_2), y(P_1)]}{x(P_2)-x(P_1)}  =0,\\
&[y(P_3), \frac{dy}{dx}(P_3)] 
\frac{[y(P_1), y(P_2)]}{x(P_1)-x(P_2)}+
\frac{[y(P_3), y(P_1)]}{x(P_3)-x(P_1)}
\frac{[y(P_3), y(P_2)]}{x(P_3)-x(P_2)}  =0.\\
\end{split}
\end{gather*}
%By arranging them, these relations \textbf{imply? ``mean''=''they are
%equivalent to''} mean that
By arranging them, these relations imply that
$\hat \beta(P_1, P_2) = 0$,
$\hat \beta(P_2, P_3) = 0$, and
$\hat \beta(P_3, P_1) = 0$, where
\begin{gather*}
\begin{split}
\hat \beta(P_1, P_2) :=
&
\left(\frac{[y(P_1), y(P_2)]}{x(P_1)-x(P_2)}\right)^2-
[y(P_1), \frac{dy}{dx}(P_1)] 
[y(P_2), \frac{dy}{dx}(P_2)].\\
%&[y(P_1), \frac{dy}{dx}(P_1)] 
%[y(P_3), \frac{dy}{dx}(P_3)] -
%\left(\frac{[y(P_1), y(P_3)]}{x(P_1)-x(P_3)}\right)^2,\\
%&[y(P_2), \frac{dy}{dx}(P_2)] 
%[y(P_3), \frac{dy}{dx}(P_3)] -
%\left(\frac{[y(P_2), y(P_3)]}{x(P_2)-x(P_3)}\right)^2.\\
\end{split}
\end{gather*}
It should be noted that $\hat \beta(P_1, P_2) = 0$ is a relation
in $X_4^2$ rather than $X_4^3$. Then it is obvious that
$\hat \beta(P_1, P_2)=\hat \beta(P_2, P_1)$, and
$\hat \beta(B_a, P_2)=0$ for $a=1, \ldots, 5$.
By expanding the first term, we have
\begin{gather}
\begin{split}
\hat \beta(P_1, P_2) =
&\Bigr(
\sum_{\ell=1}\frac{1}{\ell!} 
[y(P_2),  \frac{d^\ell y}{dx^\ell}(P_2)]
(x(P_1)-x(P_2))^{\ell-1}
\Bigr)\\
\times&
\Bigr(
\sum_{\ell=1}\frac{1}{\ell!} 
[y(P_1),  \frac{d^\ell y}{dx^\ell}(P_1)]
(x(P_1)-x(P_2))^{\ell-1}
\Bigr)\\
&-[y(P_1), \frac{dy}{dx}(P_1)] 
[y(P_2), \frac{dy}{dx}(P_2)].
\label{eq:hatbeta}
\end{split}
\end{gather}
%\mathbf{I cannot understand the English
%in the following sentence--this does not seem to be a main statement:}
Noting the order of $\hat\beta$ with respect to $(x(P_1)-x(P_2))$,
$\hat\beta$  obviously has a single zero at $(x(P_1)-x(P_2))$.
However we are not concerned with the case $x(P_1)=x(P_2)$
due to (\ref{eq:avoid}).
Thus, we consider 
$$
{\breve\beta}(P_1, P_2):=\frac{\hat\beta(P_1, P_2)}{x(P_1)-x(P_2)},
$$
which does not vanish at $x(P_1)=x(P_2)$.
Noting that $d y_a/dx $ belongs to $\frac{1}{y_a^2}\CC[x]$ and 
$\hat \beta(B_b, P_2)=0$ for $b=1, \ldots, 5$ ($B_b$ is a zero of $y_a$),
we find that  
${\breve\beta}(P_1, P_2)$ is a meromorphic function belonging to
$H^0((X_4\setminus \infty)^2, \cO_{X_4^2})$.  %\mathbf{?}
Thus ${\breve\beta}(P_1, P_2)=0$ defines a hypersurface 
$S$ in $(X_4\setminus \infty)^2$,
namely a curve. By letting $p_a: X_4^2 \to X_4$ be the
 natural projections ($a=1,2$), 
there is a continuous 
surjective map $\varpi_a:S \to p_a((X_4\setminus \infty)^2)$, which 
therefore is finite-to-one.
%Since ${\breve\beta}(P_1, P_2)$ is a meromorphic function,
%for each point $P_a$ in $p_a((X_4\setminus \infty)^2)$, the number of
%$\varpi_a^{-1}(P_a)$ is finite. 

%\textbf{wrong: I recommend deleting this}In other words,
%$S$ covers $X_4\setminus \infty$ topologically.

Finally, if we consider inside
 $(X_4\setminus \infty)^3$ 
 the three hypersurfaces$S_{12}$, $S_{23}$ and $S_{31}$,
 defined by  ${\breve\beta}(P_1, P_2)=0$,
${\breve\beta}(P_2, P_3)=0$ and
${\breve\beta}(P_3, P_1)=0$,
we can say that they have non-empty intersection by B\'ezout's theorem,
since the equations can be extended to the compactification. % and
If one of the $P_a$ is $B_b$ ($a=1,2,3, b=1,\ldots,5$),
it is obvious that they don't satisfy (\ref{eq:proofP6.3}).
Hence any common zero of   ${\breve\beta}$'s must be a non-trivial solution.

%Due to the relation between $\hat\beta=0$ and $\beta=0$,
In other words, we find non-trivial points in $X_4^3$ satisfying $\beta=0$'s
and we let them be $\{(P_{0,1}, P_{0,2},P_{0,3})\}
\subset (X_4\setminus \infty)^3$. 

\bigskip
Hence letting one of  
$(P_{0,1}, P_{0,2},P_{0,3})$ be
$(P_{R,1}, P_{R,2},P_{R,3})$,
we have
$$
\left(\frac{\muHf{3}(P; P_{R,1}, P_{R,2}, P_{R,3})dx}
 {3y_8^2}\right)
\sim 2(P_{R,1}+P_{R,2}+P_{R,3}) +2 (B_4+B_5) -4\infty
$$
by noting that the function $\gamma$ in 
Lemma \ref{lem:detbetas} is essentially the same as
$\Psi_4^{(4)}(P_1, \ldots, P_4)$.
Due to Theorem 11 in [Lew], 
$$
\hu_o(2(P_{R,1}+P_{R,2}+P_{R,3}) +2 (B_4+B_5))=-2 \xi_R,
$$
where $\xi_R$ is the Riemann constant with base point $\infty$.
\end{proof}

\begin{lemma} \label{lem:detbetas}
For a point $(P_{0,1}, P_{0,2},P_{0,3}) \in (X_4\setminus \infty)^3$ 
such that every permutation $(P_1, P_2, P_3)$ of
$(P_{0,1}, P_{0,2},P_{0,3})$ satisfies $\beta(P_1,P_2,P_3)=0$,
$$
2(P_{0,1}+P_{0,2}+P_{0,3}) 
+2(B_4+B_5) - 10 \infty \sim 0.
$$
\end{lemma}

\begin{proof} 
Let us consider the divisor of
\begin{gather*}
\begin{split}
\gamma(P; P_{1}, P_{2}, P_{3}) =
%\frac{y_7(P)}{y_8(P)}
\left|\begin{matrix} 
y_7(P) & y_8(P) & (x y_7)(P) & (x y_8)(P) \\
y_7(P_1) & y_8(P_1) & (xy_7)(P_1) & (x y_8)(P_1) \\
y_7(P_2) & y_8(P_2) & (xy_7)(P_2) & (x y_8)(P_2) \\
y_7(P_3) & y_8(P_3) & (xy_7)(P_3) & (x y_8)(P_3) 
\end{matrix}\right|.\\
\end{split}
\end{gather*}
It is trivial that its divisor contains 
$P_1+ P_2+P_2$ and $(B_1 + B_2 + B_3 + B_4+B_5)$. 
However there might be other three points as its zeros.
By considering a local parameter $t_{c}:=x-x_c$ 
around $P_{c}$, we have
\begin{gather*}
\begin{split}
&(y_7(P), y_8(P),  x(P) y_7(P), x(P) y_8(P))\\
&=
(y_7(P_{c})
+\frac{d}{dx} y_7(P_{c}) t_{c}
+\frac{1}{2!}\frac{d^2}{dx^2} y_7(P_{c}) t_{c}^2+
\cdots,\\
&\ldots,
x y_8(P_{c})
+\frac{d}{dx} (xy_8(P_{c})) t_{c}
+\frac{1}{2!}\frac{d^2}{dx^2} (xy_8(P_{c})) t_{c}^2+
\cdots).
\end{split}
\end{gather*}
In other words, we have
\begin{gather*}
\begin{split}
&\gamma(P; P_{1}, P_{2}, P_{2}) \\
&=\left|\begin{matrix} 
y_7(P_c) +\frac{d y_7(P_{c}) }{dx} t_{c}&
y_8(P_c) +\frac{d y_8(P_{c}) }{dx} t_{c}\\
y_7(P_1) & y_8(P_1) \\
y_7(P_2) & y_8(P_2) \\
y_7(P_3) & y_8(P_3)  
\end{matrix}\right.\\
&\left.\begin{matrix} 
 (xy_7)(P_c) +\frac{d (x y_7)(P_{c}) }{dx} t_{c}&
 (xy_8)(P_c) +\frac{d (x y_8)(P_{c}) }{dx} t_{c}\\
 (x y_7)(P_1) & (x y_8)(P_1) \\
 (x y_7)(P_2) & (x y_8)(P_2) \\
 (x y_7)(P_3) & (x y_8)(P_3) 
\end{matrix}\right| + (\mbox{terms with }t_c^\ell (\ell>1)).\\
\end{split}
\end{gather*}
As an example, in the $P=P_1$ case, $\gamma(P; P_{1}, P_{2}, P_{2})$ vanishes
to second order because of 
the condition $\beta(P_{1}, P_{2}, P_{2}))=0$.
Hence the divisor of $\gamma(P; P_{1}, P_{2}, P_{3})$
must be
$$
2(P_{1}+ P_{2}+ P_{3})+(B_1 + B_2 + B_3+B_4 + B_5) - 11 \infty \sim 0,
$$
and the divisor of $\frac{y_7(P)}{y_8(P)} \gamma(P; P_{1}, P_{2}, P_{3})$
is given by
$$
2(P_{1}+ P_{2}+ P_{3})+2(B_4+B_5) - 10 \infty \sim 0.
$$
\end{proof} 

Due to Proposition \ref{prop:thetadivisor}
and  Theorem 1.1 in \cite{F1}, we have the following corollary:

\begin{corollary} \label{cor:thetadivisor}
 There is a theta characteristic, 
\begin{equation}
   \delta:=\left[\begin{array}{cc}\delta''\ \\
       \delta'\end{array}\right]\in \left(\tfrac12\ZZ\right)^{2g},
  % \label{eq2.9} %3.15
\end{equation} 
which is equal to the Riemann constant $\xi_R$ associated with $X_g$ $g=4,12$
with respect to the base point $\infty$ and the 
corresponding period matrix. 
In other words, for every $(P_1, P_2, \ldots, P_{g-1}) \in S^{g-1} X_g$,
$$
\theta(\hu(P_1, \ldots, P_{g-1}) + \xi_R) = 0.
$$
\end{corollary}

\begin{proof}
The $g=12$ case is trivial, whereas
the $g=4$ case is due to Proposition \ref{prop:W3W3}
and Proposition \ref{prop:thetadivisor}.
\end{proof}

%%%%%%%%%%%%%%%%%SM 06302013

%\subsection{Jacobi inversion formula}
\subsection{Jacobi inversion formulae over $\Theta^k$}

For $X=X_4, X_{12}$, we also introduce 
$$
S^n_m(X) := \{D \in S^n(X) \ | \
    \mathrm{dim} | D | \ge m\},
$$
where $|D|$ is the complete linear system
$\uab^{-1}(\uab(D))$ in IV.1 of 
\cite[IV.1]{ACGH}.
\begin{theorem} {\bf{(Jacobi inversion formula)}}\label{prop:4.5}
\begin{enumerate}
\item For $(P, P_1, \cdots, P_4) \in X_4 \times 
S^4(X_4) \setminus S^4_1(X_4)$,
we have
\begin{enumerate}
\item 
$\muHf{4}(P; P_1, \ldots, P_4) = \phiHf{4}(P)-$ 
$\sum_{j=1}^{4} \wpf{4 j}$ $(\hu(P_1,$ $\ldots ,P_4))$ $
\phiHf{j - 1}(P)$.

\item
$
\wpf{4, k+1}(\hu(P_1,\ldots ,P_4)) = (-1)^{3-k}\muHf{4, k}
(P_1, \ldots, P_4),$ 
$(k=0, \ldots, 3)$.
\end{enumerate}

\item For $(P, P_1, \cdots, P_{12}) \in X_{12} 
\times S^{12}(X_{12}) \setminus S^{12}_1(X_{12})$, we have
\begin{enumerate}
\item 
$\mut{12}(P; P_1, \ldots, P_{12}) =$ 
$\phit{12}(P) - \sum_{j=1}^{{12}}$ 
$\wpt{{12}, j}(\hu(P_1,\ldots ,P_{12}))$
$\phit{j - 1}(P)$.

\item
$
\wpt{{12}, k+1}(\hu(P_1,\ldots ,P_{12})) 
= (-1)^{{12}-k-1}\mut{{12}, k}(P_1, \ldots, P_{12}), \quad
(k=0, \ldots, {12}).
$
\end{enumerate}
\end{enumerate}
\end{theorem}

 \begin{proof}
Same as in Prop.4.6 of \cite[Prop. 4.6]{MP08}.
\end{proof}

%\label{Jacobi inversion formulae over $\Theta^k$}

We  introduce
\begin{equation}
	\Theta^{k} := \WW^{k} \cup [-1] \WW^{k}.
\label{eq:Theta:k}
\end{equation}
For $X_4$ case, noting Corollary \ref{cor:thetadivisor},
we have considered the locus $\WW^{k} \cup [-1] \WW^{k}$ 
$k=0,1,2$ of the theta divisor  $\WW^3$, rather than
$\hu_o(S^k X_4) \cup [-1]\hu_o(S^k X_4)$.

We have the following theorem, in which  both  sides
are obtained by taking a limit by $P_{k+1} \to \infty$ as
mentioned in Th.5.1 of \cite{MP08}. Some of the left-hand
sides in the following theorem are given by zero over zero
but they are well-defined in the limit.

\begin{theorem} \label{th:5.1}
The following relations for 
$\muHf{}$ to $X_4$, and  for $\mut{}$ to $X_{12}$ hold.
\begin{enumerate}

\item $\Theta^g$ case:
for
$(P_1, \ldots, P_g) \in S^g(X_g) \setminus S^g_1(X_g)$ and
%$u = \pm \hu(P_1, \ldots, P_g)\in \kappa^{-1}(\Theta^g)$,
$u = \hu(P_1, \ldots, P_g)\in \kappa^{-1}(\Theta^g)$,
\begin{gather*}
\frac{\sigmag{i}(u) \sigmag{g}(u) - \sigmag{g i}(u) \sigmag{}(u)}
               {\sigmag{}^2(u)}
         =(-1)^{g+1-i} \mug{g, i - 1}(P_1, \ldots, P_g),
         \quad \mbox{for } 0< i \le g.\\
\end{gather*}

\item $\Theta^{g-1}$ case:
for
$(P_1, \dots, P_{g-1}) \in S^{g-1}(X_g) \setminus S^{g-1}_1(X_g)$ and
$u = \pm \hu(P_1, \ldots, P_{g-1})\in \kappa^{-1}(\Theta^{g-1})$,
\begin{gather*}
\frac{\sigmag{i}(u)}
               {\sigmag{g}(u)}=
\left\{\begin{matrix}
         (-1)^{g-i} \mug{g-1, i-1}(P_1, \ldots, P_{g-1}) 
                 & \mbox{for } 0 < i < g,\\
          1 & \mbox{for } i = g.
       \end{matrix}\right.
\end{gather*}

\item $\Theta^{k}$ case:
for
$(P_1, \ldots, P_k) \in S^k(X_g) \setminus S^k_1(X_g)$ and
$u = \pm \hu(P_1, \ldots, P_k)\in \kappa^{-1}(\Theta^k)$,
$(k=1,2,\ldots, g-2)$,
\begin{gather*}
\frac{\sigmag{i}(u)}
               {\sigmag{k+1}(u)}=
\left\{\begin{matrix}
          (-1)^{k-i+1} \mug{k, i-1} (P_1, \ldots, P_k) 
          & \mbox{for }  0 < i \le k,\\
          1 & \mbox{for } i = k+1,\\
         0 &\mbox{for } k+ 1 < i \le  g.
       \end{matrix}\right.
\end{gather*}
\end{enumerate}
\end{theorem}

\begin{proof}
Essentially the same as in Th.5.1 of \cite[Th. 5.1]{MP08}.
\end{proof}

%%%%%%%%%%%%%%%%%SM 06302013

We should note that for the  $k=1$ case, we have
$$
\frac{\sigmat{1}(u)}
{\sigmat{2}(u)}=-x, \quad \mbox{for}\quad
(x,y_{13},y_{14},y_{15},y_{16}) \in X_{12}. 
$$

\bigskip
\bigskip

\section{Configurations of Curves and  Monstrous Moonshine}
\label{sec:Curves&MM}

We highlight some relationships between the
curves $X_4$ and $X_{12}$, and an additional smooth curve
$\hat{X}_4$ of genus 4, with affine plane model of type $(3,4)$,
and their algebraic functions; the degrees of these functions
and of the differentials on the curves are related to  numbers 
that appear in
the Monstrous Moonshine. The idea for relating the curves is to factor out a
(local) group action; this will be achieved by retaining only
some of the functions in the affine rings of the curves,
and manufacturing the relations they satisfy.

\subsection{
Covering spaces of $(3, 7, 8)$, and $(6,13,14,15,16)$ curves}

We will relate to
 $X_{12}$ certain covers and subcovers, with the goal of
presenting observations on symmetries,
acting on algebraic or transcendental functions,
 that appear to be related to the 
Monster.

First we should note that 
by sending $y_7$ and $y_8$ to $y_{14}$ and $y_{16}$,
the relations $f_{12,4}$, $f_{12,7}$ and $f_{12,9}$
for $X_{12}$ map  to $f_{14}$, $f_{15}$, and  $f_{16}$
for $X_{4}$, respectively.
Hence $X_{12}$ has a natural projection to $X_4$. More precisely,
$X_{12}$ is a double covering of $X_{4}$
as mentioned in the proof of Proposition \ref{prop:KomedaProp1}.

The Jacobian $\JJf$ has $120$ odd and 
$136$ even $\theta$ characteristics for a total of $2^8 = 256$.
 %$(\ZZ/2 \ZZ)^{8}$; 
Coble in Ch. 5 of \cite[Ch. 5]{C} 
identified several group actions on the characteristics
and Vakil \cite{V} introduced an $\mathrm{E}_8$ action. This sporadic
group 
plays a role in
the representation of the Monster \cite{CN}.
Through the double covering $X_{12}$ over $X_4$,
the basis of the tangent space of $\JJf$ is naturally embedded in those of
$\JJt$, $\iota:T\JJf \hookrightarrow T\JJt$,  
$$
\frac{1}{2}\iota(\nuIf{1}) = \nuIt{4}, \quad 
\frac{1}{2}\iota(\nuIf{2}) = \nuIt{6}, \quad 
\frac{1}{2}\iota(\nuIf{3}) = \nuIt{9}, \quad 
\frac{1}{2}\iota(\nuIf{4}) = \nuIt{11}. \quad 
$$

We have also have morphisms $X_4 \to X_6$ and $X_4 \to X_7$.
Similarly $X_{12}$ admits projections, 
from $X_{12}$ to $X_4$, $X_6$, $X_7$ and 
the following two further curves $X_2$ and $X_{30}$: 
\begin{enumerate}
\item a hyperelliptic curve $X_2$ of genus two given by
$f_{12,9}=y_{15}^2 -\hat k_2(x) k_3(x)$, and

\item a $(6, 13)$ singular curve $X_{30}$ of genus $30=g(\LA6,13\RA)$ 
given by $y_{13}^6 -\hat k_2(x)^3 k_2(x)^2 k_3(x)$. 
\end{enumerate}

\bigskip
As we showed in (\ref{eq:projhatR}), we also have the natural projections
from the affine curve
$\hat C_{12}$ to the affine curve $C_{12} \subset X_{12}$, where 
$\hat C_{12} := \Spec \hR_{12}$ and $C_{12} := \Spec R_{12}$.
Since $\hR_{12}$ is decomposed as $\CC[x]$-algebra;
$$
\hR_{12}=
\CC[x, w_3]/(w_3^6 - k_3(x))\otimes_{\CC[x]}
\CC[x, w_2]/(w_2^6 - k_2(x))\otimes_{\CC[x]}
\CC[x, \hw_2]/(\hw_2^6 -\hk_2(x)),
$$
$\hat C_{12}$ is 
can be viewed as 
a fiber product of $C_1\times_{C} C_0 \times_{C} \hat C_0$,
where $C:=\Spec \CC[x]$, 
$C_1:=\Spec \CC[x, w_3]/(w_3^6 - k_3(x))$,
$C_0:=\Spec \CC[x, w_2]/(w_2^6 - k_2(x))$
and
$\hat C_0:=\Spec \CC[x, \hw_2]/(\hw_2^6 -\hk_2(x))$.

\bigskip
Additionally,
$\hat C_{12}$ 
 has a projection to $\hat C_{4}:= \Spec \hR_4$,
($\hat C_{12} \to \hat C_{4}$)
where
$$
\hR_4:= \CC[x, \hy_4]/ (\hy_4^3 - k_3(x) k_2(x))
$$
by identifying $w_{3} w_{2}$ with  $\hy_4$ as in (\ref{eq:w6}), {\it{i.e.}},
$ \hy_4 := w_{3} w_{2}$.
This curve $\hat C_4$
is birational to the sextic affine curve $\hat C_4'$ for
$
{w}^3 = \prod_{i=0}^5(z-c_i)
$
via the change of variables:
$$
z := \frac{1}{x - \tc_0} + c_0, \quad
w:=\frac{1}{\sqrt[3]{-\prod_{i=1}^5(b_i - c_0)}}
     \frac{\hy_4}{(z-c_0)^2},\quad
c_i := \frac{1}{b_i - \tc_0} + c_0 \quad
$$
giving local coordinates at $\infty$ for 
the   Riemann surface $\hat X_4 = \hat C_4 \bigcup \hat C_4'$.
The Jacobian $\hJJf$ of $\hat X_4$ also 
has  256 theta characteristics, the order of $(\ZZ/2 \ZZ)^{8}$,
which  Coble used in  \cite{C} (cf. p. 250).
It is known that the reflection group of $\mathrm{E}_8$ is related to  an
algebraic curve of genus 4 and degree 6 via a Del Pezzo surface
and these theta characteristics \cite{Man,Z}.
The surface $y^3 = x^5 + q^2$  has a Du Val singularity
and is also connected with $\mathrm{E}_8$ and McKay correspondence 
\cite{D,DPT,Sh,St}.

\subsection{Numbers $10$ and $19$ in $L(6,13,14,15,16)$ and Norton Numbers}

Our first remark is that the complementary
 sequence of the sub-semigroup $H_{12}$
of the  non-negative integers $\NN_0$
generated by
$M_{12}:=\{6, 13, 14, 15, 16\}$, namely the Weierstrass semigroup of
 $X_{12}$ at $\infty$, is
$$
L(H_{12}):=\{1, 2, 3, 4, 5, 7, 8, 9, 10, 11, 17, 23 \}, 
$$
while  the Norton sequence is given by \cite{FMN,Mc,MS,N}
(see Appendix A),
$$
\{1, 2, 3, 4, 5, 7, 8, 9, 11, 17, 19, 23 \}.
$$
The sequences differ only by the elements $10$ and $19$.
Using the covering $\hat C_{12}$ of $C_{12}$,
we notice that 
 the order of pole of the denominator of the holomorphic
one form (\ref{eq:nuI12}) of $X_{12}$ decomposes as 
$10+19$:
$$
	\wt\left(\frac{1}{y_{13} y_{16}}\right)
	=\wt\left(\frac{1}{y_{14} y_{15}} \right)=
\wt\left(\frac{1}{w_{3}^5 w_{2}^4 \hw_2^3} \right)= 10+19,
$$
because
$\wt(w_{3}^2 w_{2}^2)=10$ and
$\wt(w_{3}^3 w_{2}^2 \hw_2^3)=19$
in $H^0(
\hat C_{12}, \cO_{\hat C_{12}})$.  This points to the significance
of the configurations of curves we consider.

\subsection{A similarity between $\sigmat{}$ and replicable functions}

For a replicable function (see Appendix A)
$$
	f(q) = \frac{1}{q}+\sum_{i=1}^\infty c_i q^i,
$$
the coefficients generated by a finite number of $h_i$'s,
$$
      c_i  \in \QQ[h_1, h_2, h_3, h_4, h_5, h_7, h_8, h_9, 
       h_{11}, h_{17}, h_{19}, h_{23}],
$$
plays an essential role in the moonshine phenomena as in \cite{FMN,Mc,MS}.
The Grunsky coefficients $h_{m,n}$ of $f(q)$ are defined by \cite{G}
$$
\sum_{m,n} h_{m, n}p^m q^n := \log\left(p q \frac{f(q) - f(p)}{p - q}\right),
$$
and  have the property,
$$
       h_{m,n} \in \QQ[h_1, h_2, h_3, h_4, h_5, h_7, h_8, h_9, 
       h_{11}, h_{17}, h_{19}, h_{23}].
$$
Further it is known that
the SL(2,$\ZZ$) action on replicable functions
 can be expressed in terms of  Faber polynomials \cite{MS}.
Since the $\tau$-function solutions to
the dKP hierarchy
 is also given by the Faber polynomials \cite{CK},
we hope that  the sigma function $\sigmat{}$ and its
derivatives may provide information on replicable functions:
an integrable system of KP-hierarchy can be typically  solved in terms
of theta, tau, sigma functions \cite{EEG,N2,Mul}.

\begin{remark}
{\rm{
\begin{enumerate}
\item 
As mentioned in Remark \ref{rmk:Schur},
the $\sigmat{}$ function over $\CC^{12}$ is an entire function
and is also expected to be expressed by the Schur function.
If the moduli parameters $(b_1, \ldots, b_7)$ are polynomials of
$p$ and $q$, it is expected that
\begin{gather*}
   \sigmat{}(u) - S_{\mathcal Y_{12}}(T)|_{T_{(i)} = t_i} 
            = \sum_{m,n} \hat h_{m,n} q^m p^n,
\end{gather*}
where 
$$
     \hat h_{m,n}  \in \QQ[\ta_1, \ta_2, \ta_3, \ta_4, \ta_5, 
\ta_7, \ta_8, \ta_9, \ta_{10}, \ta_{11}, \ta_{17}, \ta_{23}].
$$

\item  $\sigmat{}$ is a generalization of Weierstrass sigma function
whereas the Weierstrass sigma function plays an important role in \cite{HBJ};
if we consider a suitable degeneration of the curve $X_{12}$ associated with
elliptic curves, it is expected that
$\sigmat{}$ might be written in terms of Weierstrass' elliptic $\sigma$.

\item 
By the Riemann-Kempf theorem, 
using (\ref{eq:ta12}), a suitable derivative of $\sigma$,
$\sigma_{\ta_{i_1}, \cdots, \ta_{i_k}}$, is a function of 
$\{\ta_1, \ta_2, \cdots, \ta_{23-\Nt(13-k)}\}$ over 
an open subset of the subvariety $\kappa^{-1}\WW_k$
\cite{MP12}. The other $\ta$'s of the hierarchy are functions of 
$\ta_1, \ta_2, \cdots, \ta_{23-\Nt(13-k)}$ \cite{MP12}. 
This is analogous to the fact that for a suitable replicable function,
every $c_i$ belongs to a subring $\QQ[h_1, \ldots, h_\ell]$
of $\QQ[h_1, h_2, h_3, h_4, h_5, h_7, h_8, h_9, 
h_{11}, h_{17}, h_{19}, h_{23}]$, and a subset of $h$'s are functions
of the other $h$'s, e.g., 
for the case of the $j$-function, $c_i\in\QQ[h_1, h_2, h_3, h_5]$
\cite{Mah}.

\item For an $(n,s)$ curve, algebraic solutions
to the dKP equation exist \cite{MP09};
it should be possible to  generalize
them to curves covered locally by affine patches such as  $X_4$ and $X_{12}$.
As an illustration, let us consider the case that $\hk_2 = k_2$:
Then $\phit{11}\phit{9}=(\phit{10})^2$ and for $v:=y_{14}/y_{15}$ 
$=\displaystyle{\frac{d}{d \ta_2}\frac{\sigmat{1}}{\sigmat{2}}\diagup
\frac{d}{d \ta_3}\frac{\sigmat{1}}{\sigmat{2}}}$ over $ \hu(X_{12}) 
\subset\CC^{12}$,
we have a $X_{12}$ curve solution of the dKP equation,
$$
\frac{\partial}{\partial \ta_3}
\frac{\partial}{\partial \ta_1} v
+\frac{\partial}{\partial \ta_1}
\left(
v\frac{\partial}{\partial \ta_1} v\right)
-2\frac{\partial}{\partial \ta_2}
\frac{\partial}{\partial \ta_2} v=0.
$$
The proof is the same as in \cite{MP09}. It is expected that
we might have the dKP hierarchy 
     for $(\ta_1, \ta_2, \ta_3, \ta_4, \ta_5, 
\ta_7, \ta_8, \ta_9, \ta_{10}, \ta_{11}, \ta_{17}, \ta_{23})$.
It is noted that 
J. McKay, and A. Sebbar considered the relation between
 replicable functions and $\tau$-functions of the
dKP hierarchy \cite{MS}.

\end{enumerate}
}}
\end{remark}

\bigskip
\bigskip

%\setcounter{section}{0}
%\renewcommand{\thesection}{\Alph{section}}
%\section{Appendix: Norton condition}
\appendix

\section{Norton condition}
\label{sec:AppendixA}

Let $\cA$ be a
 $\QQ$ ring, with filtration assosicated  to  multiplication,
$\cA = \cup_i \cA_i$ and $\cA_j \cA_i \subset \cA_{i + j}$.
Let us consider $q = \ee^{2\pi \ii \tau}$ for 
$z \in H_+:=\{\tau \in \CC \ | \ \Im \tau \ge 0\}$
and a function
$$
  f(q) = q^{-1} + h_1 q + h_2 q^2 + \cdots
$$
where $h_j \in \cA_j$.
We also write $f(\tau) = f(q)$.

The Grunsky coefficient $h_{m,n} \in \cA_{m+n-1}$ of 
$f(q)$ is defined by \cite{G}
$$
\sum_{m,n} h_{m, n}p^m q^n := \log\left(p q \frac{f(q) - f(p)}{p - q}\right)
$$
and the Faber polynomial $F_{f,n}(f)$ \cite{S,F} is defined by
$$
F_{f,n}(f(q))
= \frac{1}{p^n} + n \sum_{m > 0} h_{m, n}p^m,
$$
where $h_{1, m} = h_{m}$.
 From the definition of Grunsky coefficients, we have these property:
\begin{lemma}
\begin{equation*}
h_{r,s}= h_{r+s-1} +\frac{1}{r+s} 
\sum_{m=1}^{r-1} 
\sum_{m=1}^{s-1} 
(n+m)h_{r+s-m-n-1}h_{m, n}.
\end{equation*}
\end{lemma}
These appear in the dispersionless KP hierarchy \cite{CK}.
% which is closely connected with one forms in Jacobian \cite{MP09}.

The replicable functions are generalizations of the 
the elliptic modular function $j(\tau)$,
which is characterized by its expansion and the following property under
the action of Hecke operators $T_n$ for every $n \ge 1$,
$$
n T_n(j(\tau)) = \sum_{ad=n, 0 \le b <d}
   j\left(\frac{a \tau + b}{d}\right) =
F_{f,n}(j(\tau)).
$$
This gives an $SL_2(\ZZ)$ action on $j$.

The replicable functions were characterized by Norton, 
as having certain properties 
under the action of the Hecke operators, as members of
the finite family of functions
$\{f^{(a)}\}$ given by
$$
\sum_{ad=n, 0 \le b <d}
   f^{(a)}\left(\frac{a \tau + b}{d}\right) =
F_{f,n}(f(\tau)).
$$

Norton considered characterized the
coefficients $h_{m,1}$ due to such action.

\begin{definition} 
If whenever  $nm =rs$ and $(n, m) =(r, s)$, we have the identity
$h_{n,m} =h_{r,s}$, we say that $\cH$ is replicable.
The condition is called {\rm{Norton condition}}.
\end{definition}

\begin{example}{\rm{
\begin{enumerate}
\item
$ h_6 = h_{3,2} = h_4 + h_1 h_2$,

\item
$h_{12} = h_{3,4} = h_6 +h_1^2 h_2 + 2 h_2 h_3 + h_1 h_4$,

\item
$h_{10} = h_{5,2} = h_6 +h_1 h_4 + h_2 h_3$,

\item
$h_{14} = h_{7,2} = h_8 +h_1 h_6 + h_2 h_5 + h_3 h_4$, and

\item
$h_{15} = h_{3,5} = h_7 +2h_2 h_4 + h_3^2  + h_5 h_1 + h_1^2 h_3 + h_1 h_2^2$.
\end{enumerate}
}}
\end{example}
The Norton condition 
is not consistent with the grading of $\cA$ unless we
modify the grading, e.g.,
$h_{-1} = 1$ with the weight $-1$.

Norton proved that \cite{Cu},
\begin{theorem}
For a replicable function $f$ with Fourier coefficients $h_{\ell}$,
every $h_n$ belongs to $\QQ[\cH]$, where
$\cH:=\{h_i \ | \ i \in \Phi\}$ and
$\Phi:=\{ 1, 2, 3, 4,$ $5, 7, 8,$ $ 9, 11, 17, 19, 23\}$.
\end{theorem}

The set of the numbers $\ell = 1, 2, 3, 4,$ $5, 7, 8,$ $ 9, 11, 17, 19, 23$
is weakly symmetric except for the pair  $4$ and $10$ or the pair 
 $10$ and $19$ as shown in Table A.1. 
%The symmetry we say is
%in the meaning of symmetric numerical semigroup and symmetric Young
%table, which are 
%related to the Serre duality and the Gorenstein ring \cite{Mat}.
\begin{gather*}
{\tiny{
\centerline{
\vbox{
	\baselineskip =10pt
	\tabskip = 1em
	\halign{&\hfil#\hfil \cr
        \multispan7 \hfil Table A.1: Weakly symmetric property of Norton numbers.\hfil \cr
	\noalign{\smallskip}
	\noalign{\hrule height0.8pt}
	\noalign{\smallskip}
%\strut\vrule& 
0 &1 & 2 & 3 & 4 & 5 & 6 & 7 & 8 & 9 & 10 & 11 & 12 & 13 & 14 & 15 & 16&
17& 18& 19 & 20 & 21 & 22 & 23 \cr
\noalign{\smallskip}
\noalign{\hrule height0.3pt}
\noalign{\smallskip}
% \strut\vrule & 
 -& 1 & 2 & 3 & 4 & 5 & -& 7 & 8 & 9 & - & 11& -& -& -& -& -& 17& - & 19& -
& - & -& 23  \cr 
% \strut\vrule & 
  23 & - & - & - & 19 & -& 17 & - & - & - & -& -
& 11& -& 9& 8& 7& -& 5& 4 &3 & 2 & 1& - \cr 
\noalign{\smallskip}
\noalign{\hrule height0.3pt}
\noalign{\smallskip}
% \strut\vrule & 
  23 & 22 & 21 & 20 & 19 & 18& 17 & 16 & 15 
& 14 & 13& 12& 11& 10& 9& 8& 7& 6 & 5& 4& 3 & 2 & 1& 0 \cr 
\noalign{\smallskip}
	\noalign{\hrule height0.8pt}
}
}
}
}}
\end{gather*}

Lastly, let's note that the complement of $\{1,2,3,4,5,7,8,9,11,17,19,23\}$
is not a numerical semigroup
because, for example,
 $10$ and $13$ belong to the complement but $23=10+13$ does not.

%\section{Appendix: Covering of curves}
\section{Covering of curves}
\label{sec:AppendixB}

As the correspondence between an algebraic variety and
a commutative ring,
we have the well-known normalization theorem \cite[p.5, p.68]{Gri}(p.5,p.68):
\begin{theorem}\label{thm:normal}
For any irreducible algebraic curve $X \subset P^2\CC$, there 
exists a compact Riemann surface $\tilde X$ and a holomorphic mapping
$s: \tilde X \to P^2\CC$ such that $s(\tilde X)=X$ and
$s$ is injective on the inverse image of the set of smooth
points of $X$.
Further the Riemann surface is unique up to its isomorphism;
if there are two Riemann surfaces $\tilde X$ and $\tilde X'$
given by normalizations of $X$, there is
a biholomorphic from $\tilde X$ to $\tilde X'$.
\end{theorem}

Using this theorem, we have the following results:

\subsection{Affine patches for the Riemann surface $X_4$}

As in the case of the $(n,s)$ curves, $X_4$
can be covered with two affine patches
(for the hyperelliptic case, 
 cf. 3.12-13 in \cite{Mu}).

\begin{proposition}
Without loss of generality, we assume that
every $b_a$ does not vanish.
By letting $\uy_7:=(y_7/x^3)$, $\uy_8:=(y_8/x^3)$ and $\ux:=1/x$,
\begin{equation}
\begin{split}
\uf_{14}&:=\uy_7^2 - \uy_8 \ux{\, } \uk_2(\ux) , \quad
\uf_{15}:=\uy_7 \uy_8 - \ux{\, }\uk_2(\ux)\uk_3(\ux),\\
\uf_{16}&:=\uy_8^2 - \uy_7 \ux{\, } \uk_3(\ux) , 
\end{split}
\end{equation}
where $\uk_3(x) = (1 - b_1 x) (1 - b_2 x) (1 - b_3 x)$ 
and $\uk_2(x) = (1 - b_4 x) (1 - b_5 x)$,
$$
\ucU_4:=
\begin{pmatrix}
\frac{\partial \uf_{14}}{\partial \ux} &
\frac{\partial \uf_{14}}{\partial \uy_7} &
\frac{\partial \uf_{14}}{\partial \uy_8} \\
\frac{\partial \uf_{15}}{\partial \ux} &
\frac{\partial \uf_{15}}{\partial \uy_7} &
\frac{\partial \uf_{15}}{\partial \uy_8} \\
\frac{\partial \uf_{16}}{\partial \ux} &
\frac{\partial \uf_{16}}{\partial \uy_7} &
\frac{\partial \uf_{16}}{\partial \uy_8} \\
\end{pmatrix},
$$
whose rank is 2 in a neighborhood of $(\ux, \uy_7, \uy_8) = (0, 0, 0)$.
\end{proposition}

\begin{proof}
Since $\ucU_4$ is 
$$
\ucU_4 =
\begin{pmatrix}
\uy_8 (\ux{\, }\uk_2(\ux))' & 2 \uy_7   & - \ux{\, }\uk_2(\ux) \\
(\ux{\, }\uk_2(\ux) \uk_3(\ux))' & \uy_8 &  \uy_7   \\
\uy_7 \uk_3'(\ux) & -\uk_3(\ux) & 2 \uy_8   \\
\end{pmatrix}.
$$
it is obvious that its rank is 2 at $(\ux, \uy_7, \uy_8) = (0, 0, 0)$.
For the case that $(\ux, \uy_7, \uy_8) \neq (0, 0, 0)$,
since the structure is the same as that of $\cU_4$, its rank is
also 2.
\end{proof}

We can define
$$
\uR_4:=\CC[\ux, \uy_7, \uy_8]/ (\uf_{14}, \uf_{15}, \uf_{16}),
$$
and $\Spec\uR_4$.
The curve $X_4$ defined by affine patches are
$\Spec R_4$ and  $\Spec \uR_4$ is non-singular.

\bigskip
\subsubsection{Homogeneous ring of $X_4$}
By partially following the arguments on the space curve in Pinkham \cite{P},
in this subsection, we mention  relations among 
$\Spec R_4$, $X_4$, 
a homogeneous ring, its Proj, the numerical
semigroup $H_4$, and the monomial curve related to $B_{H_4}$.

The coefficient ring is given by $\fCf := \CC[\hb_1, \hb_2, \hb_3, \hb_4]$
and its two projections are introduced by
$$
\varphi_b^{(4)} : \fCf \to 
\CC = \fCf/(\hb_1 - b_1, \hb_2 - b_2, \hb_3 - b_3, \hb_4 - b_4),
$$
and
$$
\varphi_0^{(4)} : \fCf \to 
\CC = \fCf/(\hb_1, \hb_2, \hb_3, \hb_4).
$$
We are now considering 
$$
R_4^{(b,z)} := \fCf[x, y_7, y_8, z]/ (\fbz{14}, \fbz{15}, \fbz{16}),
$$ 
where
for $\kbz3 := (x-\hb_1 z^3)(x-\hb_2 z^3) (x-\hb_3 z^3)$ and 
$\kbz2 :=(x-\hb_4 z^3)(x-\hb_5 z^3)$,
$$
\fbz{14}(x, y_7, y_8, z) := y_7^2 - y_8 \kbz2, \quad
\fbz{15}(x, y_7, y_8, z) := y_7 y_8 - \kbz2(x) \kbz3(x) , \quad
$$
$$
\fbz{16}(x, y_7, y_8, z) := y_8^2 - y_7 \kbz3(x).
$$
Then we have $\varphi_b^{(4)} (\fb{a}(x, y_7, y_8, 1)) = f_a(x, y_7, y_8)$ and
$\varphi_0^{(4)} (\fb{a}) = \fZ{a} $ for $ a=14, 15, 16$.
By letting $R_4^{(b)} := R_4^{(b,z)}/(z -1)$,
we have natural homomorphisms,
$$
\begin{CD}
\Spec B_{H_4} @>{}>>  \Spec R_4^{(b)}  @<{}<< \Spec R_4 \\
  @V{}VV @V{}VV @V{}VV \\
\Spec \CC @>{}>> \Spec \fCf @<{}<< \Spec \CC \\
\end{CD}.
$$
This shows that $\Spec R_4$ and $\Spec B_{H_4}$ are 
fibers  over the moduli space of pointed curves with given
Weierstrass semigroup $H_4$. Note that
$R_4^{(b)}$ is related to the moduli space 
$\cM_4$ of $(3,7,8)$ curves \cite[Ch.IV]{P}(Ch.IV).
%$\cM_4$ of $(3,7,8)$ curves \cite{P}.

We deal with the projectivization by defining
$\Rz{4}:= \varphi_b^{(4)} R_4^{(b,z)}$.
Here $z^3$ has the same weight as $x$.
The weighted homogeneous ideal in $\Rz{4}$ under $\fGmf$-action  provides
$\Proj \Rz{4}$. We have
$$
       \Rz{4}/(z - 1) \approx R_4, \quad
       \Rz{4}/(z) \approx B_{H_4}. \quad
$$
Thus we have the unique point $\infty$ as a ramified
point by letting $z=0$, which recovers (\ref{eq:rel}).
More precisely, $Z_3$, $Z_7$ and $Z_8$ correspond to
$z^3/x$, $z^7/y_7$, and $z^8/y_8$ whereas
$\ux$, $\uy_7$ and $\uy_8$ correspond to
$z^3/x$, $y_7 z^2/x^3$, and $y_8 z/x^3$;
for example, the relation $\fbz{14}(x, y_7, y_8, z)$ becomes
\begin{equation*}
\begin{split}
\fbz{14}(x, y_7, y_8, z) &=
\frac{y_7^2 y_8 x^2}{z^{14}}
\left(
\left( \frac{z^8}{y_8} \right)
\left(\frac{z^3}{x}\right)^2 
-\left(\frac{z^7}{y_7}\right)^2 
\left(1 - \hb_4 \frac{z^3}{x}\right) 
\left(1 - \hb_5 \frac{z^3}{x}\right)
\right)\\
&=\frac{x^6}{z^{4}}
\left(
\left( \frac{y_7z^2}{x^3} \right)^2
-\left(\frac{y_8z}{x^3}\right) 
\left( \frac{z^3}{x} \right)
\left(1 - \hb_4 \frac{z^3}{x}\right) 
\left(1 - \hb_5 \frac{z^3}{x}\right)
\right),\\
\end{split}
\end{equation*}
around $z = 0$.
Hence we can identify $\Proj \Rz{4}$ 
with $X_4$ and the unique infinity point $\infty$ with the point $z=0$, where
$z$ can be regarded as element of $\fGmf$.

The following relations hold, for  functions $w_i$ which
are algebraic over the ring $R_4$, and
parametrizations of $y_7$ and $y_8$:
\begin{equation}
y_7 = w_3w_2^2, \quad y_8 = w_3^2 w_2,
\end{equation}
where
\begin{equation}
w_2^3 = k_2,\quad w_3^3= k_3.
\end{equation}
%When we consider $\tilde R_4:=\CC[x, w_2, w_3]/
Since $\tilde 
R_4:=\CC[x, w_2, w_3]/
(w_2^3 - k_2(x), w_3^3- k_3(x))$,
$=\CC[x, w_2]/ (w_2^3 - k_2(x))\otimes_{
\CC[x]}$
$\CC[x, w_3]/ (w_3^3- k_3(x))$,
its spectrum has a fiber structure.

\bigskip

\subsection{Riemann surface $X_{12}$ and affine curves}

We consider a differnt affine model for $X_{12}$  than the
one in Section 4.
As for the  $(3,7,8)$ curve, we %follow the argument in \cite{Mu}
can parametrize this curve with two smooth charts;
we construct the one that contains $\infty$:
\begin{proposition}
Without loss of generality, we assume that
every $b_a$ does not vanish.
Let us define $\uy_{a} = y_{a}/x^3$ for $a=13, 14, 15, 16$.
\begin{equation*}
\begin{split}
\uf_{12,1} := &\uy_{13}^{2} -\ux{\, } \uhk_{2}(\ux) \uy_{14}, \quad
\uf_{12,2} := \uy_{13} \uy_{14} - \ux{\,}\uk_{2}(\ux) \uy_{15}, \quad
\uf_{12,3} :=\uhk_2(x) \uy_{14}^{2} - \uy_{13} \uy_{15} \uk_2(\ux) \\
\uf_{12,4} :=& \uy_{14}^{2} - \ux{\,}\uk_{2}(\ux) \uy_{16} , \quad
\uf_{12,5} :=\uy_{13}\uy_{16} - \uy_{14} \uy_{15}, \quad
\uf_{12,6} := \uy_{15}^{2} - \ux{\,}\uhk_2(\ux)\uk_3(\ux), \\
\uf_{12,7} :=& \uy_{14}\uy_{16} - \ux{\, }\uk_2(\ux)\uk_3(\ux), \quad
\uf_{12,8} := \uy_{15} y_{16} - \uk_3(\ux)\uy_{13}, \quad
\uf_{12,9} := \uy_{16}^{2} - \uk_3(\ux)\uy_{14},
\end{split}
\end{equation*}
where $\uhk_2(x) = (1 - b_6 x) (1 - b_7 x)$. 
For every $(\ux, \uy_{13}, \uy_{14}, \uy_{15}, \uy_{16})$ which is
zero of every $(\uf_{12,a})_{a = 1, \cdots, 9}$,
we have
$$
\rank \ucU_{12} = 4, \quad
\ucU_{12} := 
\left(
 \begin{pmatrix}
\frac{\partial}{\partial \ux} \uf_{12,a} &
\frac{\partial}{\partial \uy_{13}} \uf_{12,a} &
\frac{\partial}{\partial \uy_{14}} \uf_{12,a} &
\frac{\partial}{\partial \uy_{15}} \uf_{12,a} &
\frac{\partial}{\partial \uy_{16}} \uf_{12,a} \\
 \end{pmatrix}_{a = 1, \cdots, 9}
\right).
$$
around $(\ux, \uy_{13}, \uy_{14}, \uy_{15}, \uy_{16}) = (0, 0, 0, 0, 0)$.
\end{proposition}

\begin{proof}
$\ucU_{12}$ is 
$$
\begin{pmatrix}
-(\ux{\, }\uhk_2)' \uy_{14} & 2 \uy_{13} &-\uhk_{2}&      &     & \\
-(\ux{\, }\uk_{2})' \uy_{15} & \uy_{14}   & \uy_{13}  & - \ux{\, }\uk_2 &\\
(\ux{\, }\uhk_2)'\uy_{14}^2 - \uy_{13} \uy_{15} \ux{\, }\uk_2' &
-\uy_{15} \ux{\, }\uk_2 & 2\ux{\, }\uhk_2 \uy_{14} & -\uy_{13}  \ux{\, }\uk_2 & \\
- (\ux{\, }\uk_{2})' \uy_{16} & & 2\uy_{14} & & \uk_{2}  \\
    &\uy_{16} & - \uy_{15} & - \uy_{14} &\uy_{13} \\
- (\ux{\, }\uhk_2\uk_3)' & & & 2 \uy_{15} & \\
- (\ux{\, }\uk_2 \uk_3)' & & \uy_{16} &  &\uy_{14} \\
- \uk_3'\uy_{13} & - \uk_3 & & \uy_{16} & \uy_{15} \\ 
-\uk_3' \uy_{14} & & - \uk_3 & & 2 \uy_{16} 
\end{pmatrix}.
$$
At $(\ux, \uy_{13}, \uy_{14}, \uy_{15}, \uy_{16}) = (0, 0, 0, 0, 0)$,
it is obvious that its rank is 4.
For the case that 
$(\ux, \uy_{13}, \uy_{14}, \uy_{15}, \uy_{16}) \neq (0, 0, 0, 0, 0)$,
since the structure is the same as that of $\cU_{12}$, its rank is
also 4.
\end{proof}

\bigskip
\bigskip
The scheme related to $\Spec R_{12}$ and $\Spec\uR_{12}$
 is denoted by $X_{12}$.
It has a unique $\infty$ point as $1/x =0$.

\bigskip
\subsubsection{Homogeneous ring of $X_{12}$}
As mentioned in Section 4,
we partially follow the arguments on the space curve by Pinkham \cite{P}
and describe $X_{12}$ more precisely.
The coefficient ring is given by 
$$
\fCt := \CC[\hb_1, \hb_2, \hb_3, \hb_4, \hb_5, \hb_6, \hb_7]
$$
and its two projections are introduced by
$$
\varphi_b^{(12)} : \fCt \to 
\CC = \fCt/(\hb_1 - b_1, \hb_2 - b_2, \hb_3 - b_3, \hb_4 - b_4,
\hb_5 - b_5, \hb_6 - b_6, \hb_7 - b_7),
$$
and
$$
\varphi_0^{(12)} : \fCt \to 
\CC = \fCt/(\hb_1, \hb_2, \hb_3, \hb_4, \hb_5, \hb_6, \hb_7).
$$

We are now considering
$$
R_{12}^{(b,\hz)} := \fCt[x, y_{13}, y_{14}, y_{15}, y_{16},\hz]/ \cP^{(b,\hz)} ,
$$ 
where
for 
$\kbw3 := (x-\hb_1 \hz^6)(x-\hb_2 \hz^6)(x-\hb_3 \hz^6)$, 
$\kbw2 := (x-\hb_4 \hz^6)(x-\hb_5 \hz^6)$, 
$\hkbw2 := (x-\hb_6 \hz^6)(x-\hb_7 \hz^6)$, 
\begin{equation*}
    \cP^{(b,\hz)} := ( \fb{12, 1}, \fb{12, 2}, \fb{12, 3}, \fb{12, 4},
\fb{12, 5}, \fb{12, 6}, \fb{12, 7}, \fb{12, 8}, \fb{12, 9}),
\end{equation*}
\begin{equation*}
\begin{split}
\fbw{12,1} := &y_{13}^2 - \hkbw{2} y_{14}, \quad
\fbw{12,2} := y_{13} y_{14} - \kbw{2} y_{15}, \quad
\fbw{12,3} :=\hkbw2(x) y_{14}^2 - y_{13} y_{15} \kbw2(x) \\
\fbw{12,4} :=& y_{14}^2 - \kbw{2} y_{16} , \quad
\fbw{12,5} :=y_{13}y_{16} - y_{14} y_{15}, \quad
\fbw{12,6} := y_{15}^2 - \hkbw2(x)\hkbw3(x), \\
\fbw{12,7} :=& y_{14}y_{16} - \kbw2(x)\kbw3(x), \quad
\fbw{12,8} := y_{15} y_{16} - \kbw3(x)y_{13}, \\
%\fbw{12,9} :=& \hkbw2(x) y_{16}^2 - \kbw3(x)y_{13}^2.
\fbw{12,9} :=& y_{16}^2 - \kbw3(x)y_{14}.
\end{split}
\end{equation*}
Then we have $\varphi_b^{(12)} 
(\fb{12, a}(x, y_{13}, y_{14}, y_{15}, y_{16},1))$ $=
f_{12,a}$, and essentially
$\varphi_0^{(12)} (\fb{12, a}(x, y_{13}, y_{14},$ $
 y_{15}, y_{16},1))$
$ = \fZ{12, a}$ for $a=1, 2, \cdots, 9$.

By letting $R_{12}^{(b)} := R_{12}^{(b,\hz)}/(\hz -1)$,
we have natural homomorphisms,
$$
\begin{CD}
\Spec B_{H_{12}} @>{}>>  \Spec R_{12}^{(b)}  @<{}<< \Spec R_{12} \\
  @V{}VV @V{}VV @V{}VV \\
\Spec \CC @>{}>> \Spec \fCt @<{}<< \Spec \CC \\
\end{CD}.
$$
This shows that $\Spec R_{12}$ and $\Spec B_{H_{12}}$ are 
fibers over the moduli space corresponding to the 
numerical semigroup $H_{12}$. We note that
$R_{12}^{(b)}$ is related to the moduli space 
$\cM_{12}$ of $(6,13,14,15,16)$ curves.

We deal with projective space by defining
$\Rw{12}:= \varphi_b^{(12)} R_{12}^{(b,\hz)}$.
Here $\hz^6$ has the same weight as $x$.
Every homogeneous ideal in $\Rw{12}$ under the $\fGmt$-action provides
$\Proj \Rw{12}$, which is identified with $X_{12}$.
We have
$$
       \Rw{12}/(\hz - 1) \approx R_{12}, \quad
       \Rw{12}/(\hz) \approx B_{H_{12}}. \quad
$$
Thus we have the unique point $\infty$ as ramified
point by letting $\hz=0$, which recovers (\ref{eq:Z12}), similarly to
(\ref{eq:rel});
$Z_6$, and $Z_{a}$ $(a=13,14,15,16)$ correspond to
$\hz^6/x$ and $\hz^a/y_a$ whereas
$\ux$, and $\uy_{a}$ $(a=13,14,15,16)$ correspond to
$\hz^6/x$ and $\hz^{18-a}y_{a}/x^3$.

\bigskip
\bigskip

\section*{Acknowledgements}
We would like to express our thanks to John McKay for posing to us
the problem of the Moonshine: his work on replicable functions
led him to observations and conjectures (cf. \cite{MS})
which we propose to relate to our work on the $\sigma$ function;
John McKay hosted us at 
Concordia in 2004 and shared his vision most generously.  
Further the second author (S.M) has been studied the geometrical
interpretations of the replicable function by the guide of John McKay.
On the interpretation, Yuji Kodama gave his attentions to the Norton number.
John McKay taught him the relations among the Coble's work, sextic
curves and $\mathrm{E}_8$ action. The second author is 
also grateful to Kenichi Tamano, Victor Enolskii, Dmitry Leykin, 
Takao Kato and Akira Ohbuchi for valuable 
discussions and comments.
Further
we would like to express our thanks to the referee for his critical
comments, especially for his  pointing out
the mathematical mistake in the first version in 
which we assumed that the Riemann constant of $H_4$ was given in the
same way as that of $(r,s)$ curve.

\end{document}